\documentclass[10pt, reqno]{amsart}
\usepackage{amsmath,amsthm,amsfonts,amssymb,mathrsfs,bm,graphicx,stmaryrd,caption}
\usepackage{subcaption}
\usepackage{epstopdf}
\usepackage{mathtools}
\usepackage{xcolor} 
\usepackage[normalem]{ulem} 
\usepackage{dsfont}
\usepackage{multicol}
\usepackage{adjustbox}
\usepackage[colorlinks=true,linkcolor=blue,citecolor=blue,breaklinks]{hyperref}
\usepackage{url}

\usepackage{breakurl}
\usepackage{bbm}
\usepackage{upgreek}
\newcommand{\scT}{\mathscr{T}}

\newcommand{\LL}{\mathbb{L}}
\newcommand{\cZ}{\mathcal{Z}}

\newcommand{\cQ}{\mathcal{Q}}

\newcommand{\TP}{\mathrm{TP}}
\newcommand{\RR}{\mathbb{R}}

\newcommand{\hor}{\mathrm{hor}}

\newcommand{\PP}{\mathbb{P}}

\newcommand{\cJ}{\mathcal{J}}
\newcommand{\fp}{\mathfrak{p}}

\newcommand{\fraki}{\mathfrak{i}}
\newcommand{\frakm}{\mathfrak{m}}
\newcommand{\fB}{\mathfrak{B}}
\newcommand{\cB}{\mathcal{B}}
\newcommand{\vt}{\mathrm{vert}}
\newcommand{\cT}{\mathcal{T}}

\newcommand{\EE}{\mathbb{E}}
\newcommand{\peak}{\mathrm{Peak}}

\newcommand{\NN}{\mathbb{N}}

\newcommand{\latapx}{\mathrm{LatApx}}

\newcommand{\cP}{\mathcal{P}}
\newcommand{\ZZ}{\mathbb{Z}}

\newcommand{\cI}{\mathcal{I}}

\newcommand{\cK}{\mathcal{K}}
\newcommand{\switch}{\mathrm{Switch}}
\newcommand{\basin}{\mathrm{Basin}}
\newcommand{\hitset}{\mathrm{HitSet}}
\newcommand{\thinexc}{\mathrm{ThinExc}}
\newcommand{\geodwt}{\mathrm{GeodWt}}

\newcommand{\pivot}{\mathrm{Pivot}}
\newcommand{\loc}{\mathrm{Loc}}
\newcommand{\disvol}{\mathrm{DisVol}}

\usepackage[letterpaper,hmargin=1.0in,vmargin=1.0in]{geometry}
\parskip 	\smallskipamount

\newtheorem{theorem}{Theorem}
\newtheorem{lemma}[theorem]{Lemma}

\newtheorem{proposition}[theorem]{Proposition}

\theoremstyle{definition}

\newcommand{\ind}{\mathbbm{1}}
\newcommand{\wgt}{\mathrm{Wgt}}

\newcommand{\cF}{\mathcal{F}}
\newcommand{\cA}{\mathcal{A}}
\newcommand{\cC}{\mathcal{C}}

\newcommand{\cE}{\mathcal{E}}

\newcommand{\frakf}{\mathfrak{f}}

\newcommand{\cM}{\mathcal{M}}

\newcommand{\cS}{\mathcal{S}}

\newcommand{\fL}{\mathfrak{L}}

\newcommand{\cover}{\mathrm{Cover}}

\newcommand{\slope}{\mathrm{slope}}

\newcommand{\0}{\mathbf{0}}
\newcommand{\n}{\mathbf{n}}

\newcommand{\exc}{\mathrm{Exc}}

\newcommand{\coarse}{\mathrm{Coarse}}
\newcommand{\bddflips}{\mathrm{BddFlips}}
\newcommand{\smallbasin}{\mathrm{SmallBasin}}
\newcommand{\nearmax}{\mathrm{NearMax}}

\newcommand{\nbd}{\mathrm{Nbd}}
\newcommand{\bigchange}{\mathrm{BigChange}}

\newcommand{\hightf}{\mathrm{HighTF}}

\newcommand{\TF}{\mathrm{TF}}

\newcommand{\euc}{\mathrm{euc}}
\newcommand{\NI}{\mathrm{NI}}
\DeclareMathOperator*{\argmax}{\arg\! \max}

\usepackage[backend=biber,style=alphabetic,doi=false,maxalphanames=10,maxnames=50]{biblatex}
\AtBeginBibliography{\small}

\begin{document}
\title[]{Geodesic switches and exceptional times in dynamical Brownian last passage percolation}%

\author[]{Manan Bhatia}
\address{Manan Bhatia, Department of Mathematics, Massachusetts Institute of Technology, Cambridge, MA, USA}
\email{mananb@mit.edu}
\date{}
\makeatletter
\begingroup
\let\thefootnote\relax
\footnotetext{This preprint is one of two works that together replace the earlier preprint [arXiv:2504.12293v1]. The companion article \cite{Bha25} proves a quantitative ``near-existence'' result for non-trivial bigeodesics in dynamical exponential last passage percolation.}
\endgroup
\begin{abstract}
  We consider Brownian last passage percolation evolving dynamically via a discrete resampling procedure. Using $\Gamma_{(0,0)}^{(n,n),r}$ to denote a geodesic from $(0,0)$ to $(n,n)$ at time $r$, we prove that the expected total number of coarse-grained changes (or ``switches'') accumulated by $\Gamma_{(0,0)}^{(n,n),r}$ away from its endpoints during a time interval $[s,t]$ is at most $n^{5/3+o(1)}(t-s)$; we expect the exponent $5/3$ to be tight. Using the above estimate, we establish that the set $\scT$ of exceptional times at which a non-trivial bi-infinite geodesic exists a.s.\ has Hausdorff dimension at most $1/2$. Further, for any fixed direction $\theta$, we show that the set $\scT^\theta\subseteq \scT$ of times at which a non-trivial bi-infinite geodesic directed along $\theta$ exists a.s.\ has Hausdorff dimension equal to $0$.
\end{abstract}
\maketitle
\setcounter{tocdepth}{1}
\tableofcontents
\section{Introduction}
\label{sec:intro}
Exponential last passage percolation (LPP) is a simple planar model of random geometry defined as follows. Let $\{\omega_z\}_{z\in \ZZ^2}$ be a field of i.i.d.\ $\mathrm{exp}(1)$ random variables. For any $u\leq v\in \ZZ^2$, where the inequality holds co-ordinate wise, and any lattice path $\gamma$ from $u$ to $v$ taking only up and right steps, one defines $\ell(\gamma)=\sum_{z\in \gamma}\omega_z$. Thereafter, one defines the passage time $T_u^v$ as the maximum of $\ell(\gamma)$ over all such up-right paths from $u$ to $v$. Any path $\gamma$ attaining the above maximum is called a geodesic and is denoted as $\Gamma_u^v$.

Last passage percolation models such as exponential LPP are expected to belong to the Kardar-Parisi-Zhang (KPZ) \cite{KPZ86} university class-- a family of growth models believed to have common universal behaviour. In fact, this has been rigorously established for exponential LPP-- the work \cite{DV21} shows that exponential LPP has a unique scaling limit, the directed landscape \cite{DOV18}, which is believed to be the universal scaling limit of all last passage percolation models amongst other models in the KPZ universality class.

While the ``static'' model of exponential LPP has been well-studied, one can also consider a dynamically evolving variant. Namely, attach independent clocks to every vertex $z\in \ZZ^2$ and simply resample the vertex weight $\omega_z$ whenever the corresponding clock rings. This leads to a process of vertex weights $\{\omega_z^t\}_{z\in \ZZ^2,t\in \RR}$ and corresponding dynamically evolving passage times $\{T_{p}^{q,t}\}_{p\leq q\in \ZZ^2,t\in \RR}$ and geodesics $\{\Gamma_{p}^{q,t}\}_{p\leq q\in \ZZ^2,t\in \RR}$. In contrast to the static exponential LPP, the behaviour of the dynamical version is not very well-understood. Following the notion of noise-sensitivity as introduced in \cite{BKS99}, it is expected that geodesics in dynamical LPP are noise-sensitive in the sense that a microscopic amount of noise leads to a macroscopic change in the structure of geodesics. To be specific, with $|\cdot|$ denoting the cardinality of a set, the behaviour of the expected overlap $\EE|\Gamma_{(0,0)}^{(n,n),0}\cap \Gamma_{(0,0)}^{(n,n),t}|$ has recently been studied for LPP models \cite{Cha14,GH24,ADS24}. Specifically, for the model of Brownian LPP (BLPP) endowed with the Ornstein-Uhlenbeck dynamics, the work \cite{GH24} shows that the analogously defined expected overlap exhibits a phase transition at the critical time scale $n^{-1/3}$ in the sense that it is $\Theta(n)$ for $t\ll n^{-1/3-o(1)}$ and is $o(n)$ for $t\gg n^{-1/3}$.

Instead of comparing $\Gamma_{(0,0)}^{(n,n),0}$ and $\Gamma_{(0,0)}^{(n,n),t}$ for a fixed time $t$, one can attempt to analyse the sample paths of the trajectory $\Gamma_{(0,0)}^{(n,n),r}$ as $r$ varies across the time interval $[0,t]$. Heuristically, in a neighbourhood of any time $r$ at which $\Gamma_{(0,0)}^{(n,n),r}\neq \Gamma_{(0,0)}^{(n,n),r^-}$, where $r^-$ refers to the configuration just before $r$, we expect the geodesic to be chaotic in the sense that it rapidly oscillates between the two available near-optimal choices.
The goal of this paper is to analyse this phenomenon for a dynamical model of BLPP and to accurately estimate (Theorem \ref{prop:30}) the expected total number of such rapid oscillations that the geodesic undergoes as time is varied, a quantity which we shall later call geodesic switches. Roughly, we shall show that as the time parameter $r$ is varied in an interval $[s,t]$, the geodesic $\Gamma_{(0,0)}^{(n,n),r}$, in expectation, accumulates at most $n^{5/3+o(1)}(t-s)$ many switches, where these switches are measured at a coarse-grained scale; we expect the $5/3$ exponent herein to be optimal.

Thereafter, we shall use this understanding of geodesics switches to prove new upper bounds (Theorems \ref{thm:3}, \ref{thm:5}) on the dimension $\dim \scT$ of the set $\scT$ of exceptional times at which non-trivial bi-infinite geodesics (or bigeodesics) exist in a dynamical model of BLPP. Note that for \emph{static} models of first and last passage percolation, bigeodesics are conjectured to a.s.\ not exist, and this has been confirmed for exponential LPP \cite{BHS22,BBS20}. For static BLPP, \cite[Theorem 5.4]{RS24} rules out non-trivial bigeodesics that are locally leftmost or locally rightmost. Our results for $\scT$ complement those obtained in the companion paper \cite{Bha25}, where for dynamical exponential LPP, we establish a near-existence result for the analogously defined exceptional times. We now give a precise definition of the dynamical BLPP model considered in this paper; note that all the rigorous results of the paper are for this model and \emph{not} for exponential LPP.
\subsection{Dynamical BLPP}
\label{sec:model-definition}
Regarding notation, for $a<b\in \RR$, we define the discrete interval $[\![a,b]\!]=[a,b]\cap \ZZ$ and for sets $A,B\subseteq \RR$, we shall use $B_A$ to denote $A\times B\subseteq \RR^2$. Now, start with a bi-infinite sequence of i.i.d.\ standard Brownian motions $\{W_n\}_{n\in \ZZ}$. For points $(x,m)\leq (y,n)\in \ZZ_\RR$, and any non-decreasing list $\{z_k: z_k\in [x,y], k\in [\![m-1,n]\!]\}$ such that $z_{m-1}=x, z_n=y$, consider the set
\begin{equation}
  \label{eq:632}
  \xi=\bigcup_{k\in [\![m,n]\!]}\{k\}_{[z_{k-1},z_k]} \cup \bigcup_{k\in [\![m,n-1]\!]}[k,k+1]_{\{z_k\}},
\end{equation}
which we shall refer to as a ``staircase'' from $(x,m)$ to $(y,n)$ following the terminology often used in previous works. In short, we shall say that $\xi\colon (x,m)\rightarrow (y,n)$ is a staircase. For $i\in [\![m-1,n]\!]$, we simply define $\xi(i)=z_i$. For any staircase $\xi\colon (x,m)\rightarrow (y,n)$, we now associate the weight
\begin{equation}
  \label{eq:331}
  \wgt(\xi)=\sum_{i=m}^{n} W_i(\xi(i))-W_i(\xi(i-1)),
\end{equation}
and define the last passage time 
\begin{equation}
  \label{eq:332}
  T_{(x,m)}^{(y,n)}=\sup_{\xi:(x,m)\rightarrow (y,n)}\wgt(\xi),
\end{equation}
where the supremum above is over all staircases between $(x,m)$ and $(y,n)$. It can be shown that almost surely, a staircase attaining the above supremum always exists for all points $(x,m)\leq (y,n)$-- such a staircase is known as a geodesic and is denoted by $\Gamma_{(x,m)}^{(y,n)}$. Further, for any fixed $(x,m)\leq (y,n)$, there is a unique geodesic $\Gamma_{(x,m)}^{(y,n)}$ (see \cite[Lemma B.1]{Ham19}). This completes the definition of the static model of BLPP and we now move to defining a discrete dynamics on this model.

We shall now define a process $\{W_n^t\}_{n\in \ZZ,t\in \RR}$ such that for each fixed $t\in \RR$, $\{W_n^t\}_{n\in \ZZ}$ is simply a sequence of i.i.d.\ Brownian motions. To define the above process, we first define the sequence $\{X_{i,n}^0\}_{i,n\in \ZZ}$-- a family of i.i.d.\ standard Brownian motions on the interval $[0,1]$. Now, we associate i.i.d.\ exponential clocks to each $(i,n) \in \ZZ^2$ independently of $\{X_{i,n}^0\}_{i,n\in \ZZ}$. Whenever the clock associated to $(i,n)$ rings (say at time $t>0$), we independently resample the path $X_{i,n}^t$ (see Figure \ref{fig:W_to_X}). This defines the process $\{X_{i,n}^t\}_{i,n\in \ZZ,t\geq 0}$. By a simple Kolmogorov extension argument, we can in fact extend the above to define the process $\{X_{i,n}^t\}_{i,n\in \ZZ,t\in \RR}$-- note that this process is stationary in $t$.
\begin{figure}
  \begin{subfigure}{0.6\textwidth}
    \centering
    \adjustbox{valign=c}{\includegraphics[width=0.8\linewidth]{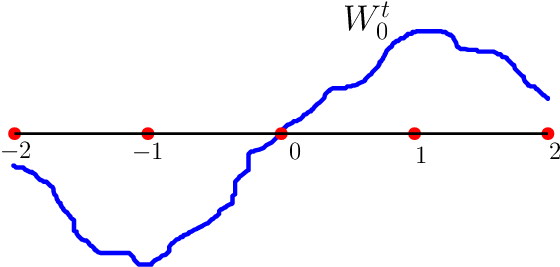}}
  \end{subfigure}\begin{subfigure}{0.4\textwidth}
    \centering
    \adjustbox{valign=c}{\includegraphics[width=0.8\linewidth]{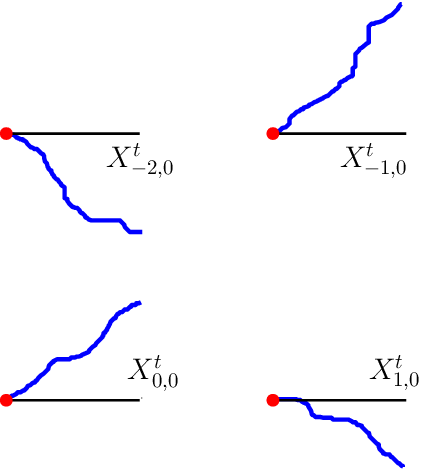}}
  \end{subfigure}
  \caption{Given the Brownian motions $W_n^t$ associated to the dynamical BLPP, we can consider the processes $X_{i,n}^t(x)= W_n^t(x+i)-W_n(x)$ for $x\in [0,1]$. The BLPP dynamics is defined by associated an independent exponential clock to each such $(i,n)\in \ZZ^2$ and then freshly resampling $X_{i,n}^t$ at any $t$ at which the clock corresponding to $(i,n)$ rings.}
\label{fig:W_to_X}
\end{figure}
With the above at hand, we simply define $W_n^t$ such that for all $x\in \RR$, we have
\begin{equation}
  \label{eq:333}
  W_n^t(x)-W_n^t(\lfloor x\rfloor)=X_{\lfloor x \rfloor,n }^t(x-\lfloor x \rfloor).
\end{equation}
It is easy to see that the process $\{W_n^t\}_{n\in \ZZ,t\in \RR}$ is stationary in $t$. Now, we simply define $T^t$ by replacing the family $\{W_n\}_{n\in \ZZ}$ in the definition \eqref{eq:332} to $\{W_n^t\}_{n\in \ZZ}$. Similarly, we define geodesics $\Gamma_{(x,m)}^{(y,n),t}$ associated to the BLPP $T^t$. Note that for any fixed $t$, the BLPP $T^t$ is marginally distributed as a static BLPP-- that is, for any fixed $t$, we have $T^t\stackrel{d}{=}T$.

\subsection{Main results 1: Geodesics switches and hitsets}
\label{sec:estim-geod-hits}
We shall be measuring geodesic switches at a coarse-grained scale, and in order to discuss this, we now develop some notation. First, for sets $K_1,K_2\subseteq \RR^2$, let $\cM_{K_1}^{K_2}$ be defined by
\begin{equation}
  \label{eq:63}
  \cM_{K_1}^{K_2}=\{(i,m)\in \ZZ^2: \exists p\in K_1, q\in K_2, w\in \{m\}_{[i,i+1]}:  p\leq w\leq q\}. 
\end{equation}
In case $K_1,K_2$ are singletons, say $K_1=\{p\}, K_2=\{q\}$, we shall simply write $\cM_p^q$ instead of $\cM_{\{p\}}^{\{q\}}$; we shall use this notational convention for all the objects defined in this paper whenever we are working with singletons. Now, let $\cT_{K_1}^{K_2,[s,t]}$ denote the discrete set of times $r$ at which $X_{i,m}^r$ is resampled for some $(i,m)\in \cM_{K_1}^{K_2}$. To be precise, for any $r$ as above, $X_{i,m}^{r}$ shall denote the Brownian path obtained after the resampling has occurred and $X_{i,m}^{r^-}$ denotes the path just prior to the resampling. Similarly, for $r$ as above, the LPPs $T^{r^-}$ and $T^r$ shall refer to the LPPs just before and after the resampling, and we shall also consider the corresponding geodesics $\Gamma_p^{q,r^-},\Gamma_p^{q,r}$. Now, for a set $K\subseteq \RR^2$, consider the coarse-grained approximation of $K$ defined by
\begin{equation}
  \label{eq:351}
  \coarse(K)=\{(i,m)\in \ZZ^2: \{m\}_{[i,i+1]}\cap A \neq \emptyset\}.
\end{equation}
We are now ready to define geodesic switches. For fixed points $p\leq q\in \ZZ_\RR$, we define
  \begin{equation}
  \label{eq:89}
  \switch_{p}^{q,[s,t]}(K)=\sum_{r\in \cT_{p}^{q,[s,t]}}\left|\coarse(K\cap \Gamma_{p}^{q,r})\setminus \coarse(K\cap \Gamma_{p}^{q,r^-})\right|,
\end{equation}
where we note that the above is well-defined since it turns out that $\Gamma_p^{q,r}$ is a.s.\ unique for all $r\in \RR$ (see Lemma \ref{lem:104}). 
Intuitively, measuring at a coarse-grained scale and only in the set $K$, $\switch_{p}^{q,[s,t]}(K)$ counts the total number of changes accumulated by the geodesic $\Gamma_p^{q,r}$ as we vary time from $r=s$ to $r=t$; we refer the reader to Figure \ref{fig:coarse} for a visual depiction of the above definition. We now state the first main result of this paper; note we shall use the boldface letters $\0,\mathbf{n}$ to denote $(0,0),(n,n)\in \ZZ^2$ respectively. 
\begin{theorem}
  \label{prop:30}
   Fix $\beta\in (0,1/2)$ and $\varepsilon>0$. For all $n$ large enough, and all $[s,t]\subseteq \RR$, we have
  \begin{equation}
    \label{eq:32}
    \EE[\switch_{\0}^{\n,[s,t]}([\![\beta n, (1-\beta)n]\!]_\RR)]\leq n^{5/3+\varepsilon}(t-s).
    \end{equation}
  \end{theorem}
  The crucial aspect in the result above is the exponent $5/3$ which we expect to be optimal. To build intuition, %
  we note that the much weaker bound $\EE[\switch_{\0}^{\n,[s,t]}( [\![\beta n,(1-\beta)n]\!])]\leq (t-s)O(n^3)$ is easy-- indeed, it can be checked that $\EE|\cT_{\0}^{{\n},[s,t]}|= (t-s)O(|\cM_{\0}^{\n}|)= (t-s)O(n^2)$ and at every time $r\in \cT_{\0}^{\n,[s,t]}$, for some absolute constant $C$, we have the deterministic bound $|\coarse(\Gamma_{\0}^{\n,r})\setminus \coarse(\Gamma_{\0}^{\n,r^-})|\leq |\coarse(\Gamma_{\0}^{\n,r})|\leq Cn$.

\begin{figure}
  \centering
  \captionsetup{width=0.95\linewidth}
  \includegraphics[width=0.7\linewidth]{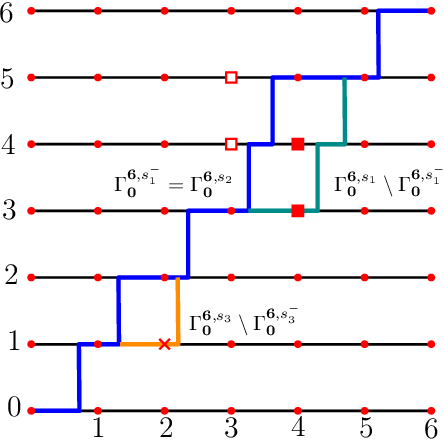}
  \caption{Here, in the setting of dynamical BLPP, we look at the geodesic between $\0=(0,0)$ and $\mathbf{6}=(6,6)$ as time is varied. In the given figure, it happens to be the case that $\cT_{\0}^{\mathbf{6},[0,1]}=\{s_1,s_2,s_3\}$ for some $s_1<s_2<s_3\in (0,1)$. Here, $\coarse(\Gamma_{\0}^{\mathbf{6},s_1^-})=\coarse(\Gamma_{\0}^{\mathbf{6},0})\subseteq \cM_{\0}^{\mathbf{6}}$ is equal to the set $\{(-1,0),(0,0), (0,1), (1,1), (1,2), (2,2), (2,3), (3,3), (3,4), (3,5), (4,5), (5,5), (5,6), (6,6)\}$, and thus $|\hitset_{\0}^{\mathbf{6},\{0\}}(\RR^2)|=14$. At time $s_1$, the geodesic changes from $\Gamma_{\0}^{\mathbf{6},s_1^-}$ to $\Gamma_{\0}^{\mathbf{6},s_1}$ and the set $\Gamma_{\0}^{\mathbf{6},s_1}\setminus \Gamma_{\0}^{\mathbf{6},s_1^-}$ is shown in cyan. Note that $\coarse(\Gamma_{\0}^{\mathbf{6},s_1})\setminus \coarse(\Gamma_{\0}^{\mathbf{6},s_1^-})= \{(4,3), (4,4)\}$-- this is marked by red squares. Thereafter, at time $s_2$, the geodesic happens to change back to the original blue path $\Gamma_{\0}^{\mathbf{6},s_1^-}$, and the set $\coarse(\Gamma_{\0}^{\mathbf{6},s_2})\setminus \coarse(\Gamma_{\0}^{\mathbf{6},s_2^-})$ consists of two elements and is marked by red hollow squares. Finally, at time $s_3$, there is another change in the geodesic and the set $\coarse(\Gamma_{\0}^{\mathbf{6},s_3})\setminus \coarse(\Gamma_{\0}^{\mathbf{6},s_3^-})$ is a singleton and is marked by a red cross. Here, $\switch_{\0}^{\mathbf{6},[0,1]}(\RR^2)= 2+2+1=5$ but $|\hitset_{\0}^{\mathbf{6},[0,1]}(\RR^2)|=14+2+1=17$. Note that $17=|\hitset_{\0}^{\mathbf{6},[0,1]}(\RR^2)|\leq |\hitset_{\0}^{\mathbf{6},\{0\}}(\RR^2)|+ \switch_{\0}^{\mathbf{6},[0,1]}(\RR^2)=19$, and the difference $19-17=2$ is explained by the intervals $\{4\}_{[3,4]}, \{5\}_{[3,4]}$ corresponding to the hollow red squares being revisited by the geodesic at time $s_2$.
  }
  
  \label{fig:coarse}
\end{figure}

The estimate Theorem \ref{prop:30} on the number of switches will allow us to obtain an estimate on the total trace of geodesics between points in disjoint sets as time is varied. To state this, we now introduce some notation. With the aim of having a coarse-grained approximation of the set swept by all geodesics $\Gamma_{p}^{q,r}$ for $p\in K_1,q\in K_2$ as the time $r$ varies over the interval $[s,t]$, we define the set $\hitset_{K_1}^{K_2,[s,t]}(K)$ by
\begin{equation}
  \label{eq:703}
    \hitset_{K_1}^{K_2,[s,t]}(K)= \bigcup_{p\in K_1\cap \ZZ_{\RR}, q\in K_2\cap \ZZ_{\RR},r\in [s,t]} \coarse(\Gamma_p^{q,r}\cap K),
  \end{equation}
  where in the above, the union is being taken over all possible geodesics between the points $p$ and $q$. For $n\in \ZZ$, with $L_n$ denoting the line segment $\{n\}_{[n-|n|^{2/3}, n+|n|^{2/3}]}$, we have the following result on the cardinality of the hitset between $L_{-n}$ and $L_{n}$.
\begin{theorem}
  \label{thm:6}
Fix $\gamma\in (0,1)$. Then for any $\varepsilon>0$ and for all large $n$, we have
    \begin{equation}
      \label{eq:706}
      \EE[|\hitset_{L_{-n}}^{L_n,[s,t]}([\![-(1-\gamma)n,(1-\gamma)n]\!]_{\RR})|]\leq n^{1+\varepsilon}+ n^{5/3+\varepsilon}(t-s).
    \end{equation}
  \end{theorem}
  \subsection{Main results 2: Dimension upper bounds on exceptional times}
\label{sec:main-results-2}
For BLPP, a bigeodesic simply refers to a doubly-infinite staircase whose every finite portion is a geodesic between its endpoints. Note that any entirely vertical or horizontal staircase is trivially a bigeodesic and thus we are only interested in non-trivial ones. It is believed that in static LPP models, (non-trivial) bigeodesics almost surely do not exist, and this is known for exponential LPP \cite[Theorem 1]{BHS22} (see \cite[Theorem 1]{BBS20} for another proof). For static BLPP, \cite[Theorem 5.4]{RS24} rules out locally leftmost and locally rightmost non-trivial bigeodesics, that is, bigeodesics whose finite portions are always the leftmost, or always the rightmost, geodesics between their endpoints. The companion paper \cite{Bha25} analyses exceptional times at which analogous bigeodesics exist in exponential LPP, and shows that, in a certain quantitative sense, such exceptional times are ``very close'' to existing.

The result Theorem \ref{thm:6} allows us to obtain a quantitative upper bound on the analogous set of exceptional times for dynamical BLPP defined as
\begin{equation}
  \label{eq:1}
  \scT=\{t\in \RR: \exists \textrm{ a non-trivial bigeodesic for the BLPP } T^t\}.
\end{equation}
\begin{theorem}
  \label{thm:3}
  For dynamical BLPP, we have $\dim \scT\leq 1/2$ almost surely.
\end{theorem}
Instead of looking at the set $\scT$ above, one might also fix a direction $\theta\in (0,\infty)$ and look at the more restrictive set $\scT^\theta\subseteq \scT$ defined as the set of times $t\in \RR$ such that there exists a $\theta$-directed bigeodesic $\Gamma^t$ for the BLPP $T^t$, where we say that an unbounded set $A\subseteq \RR^2$ is $\theta$-directed if for any sequence $(x_n,y_n)\in A$ with $|y_n|\rightarrow \infty$, we have $\lim_{n\rightarrow \infty}x_n/y_n=\theta$. In static LPP and first passage percolation models, since one expects there a.s.\ to be no bigeodesics at all, one also expects there to be no bigeodesics in a fixed direction, and indeed, this comparatively weaker statement is known for first passage percolation \cite{LN96} and for static BLPP \cite[Theorem 3.1(vi)]{SS23}. A priori, it is possible that in the dynamical BLPP case, there is a positive Hausdorff dimension set of exceptional times when bigeodesics in a fixed direction do exist, but this is not so, as is established in the following result.
\begin{theorem}
  \label{thm:5}
  For dynamical BLPP and a fixed direction $\theta\in (0,\infty)$, we a.s.\ have $\dim \scT^\theta =0$.
\end{theorem}
We mention that we do not expect Theorems \ref{thm:3}, \ref{thm:5} to be tight-- indeed, as conjectured in \cite[Conjecture 4]{Bha25}, we expect that $\dim \scT=0$ almost surely and from the heuristic calculation therein, one can also intuit that $\scT^\theta$ in Theorem \ref{thm:5} should almost surely be empty.

\paragraph{\textbf{Notational comments}} For $p\neq q\in \RR^2$, we use $\LL_p^q$ to denote the line joining $p$ and $q$. For a set $A\subseteq \RR^2$ and $r>0$, we shall often consider the $r$-horizontal neighbourhood $B_r(A)=\{(x,y): \exists (x',y)\in A \textrm{ satisfying } |x-x'|\leq r\}$. Frequently, for $a<b\in \RR$, we shall work with the discrete intervals $[\![a,b]\!]=[a,b]\cap \ZZ$.
  Often, we shall use the boldface letters $\0,\mathbf{m},\mathbf{n}$ to denote $(0,0),(m,m),(n,n)\in \RR^2$. For sets $A,B\subseteq \RR$, we use $B_A$ to denote $A\times B\subseteq \RR^2$.
  For points $p=(x_1,y_1),q=(x_2,y_2)\in \RR^2$, we shall write $p\leq q$ if $x_1\leq x_2$ and $y_1\leq y_2$. For points $p=(x_1,y_1)\neq (x_2,y_2)\in \RR^2$, we define $\slope(p,q)=\frac{x_2-x_1}{y_2-y_1}$: note that this is the inverse of the usual definition of the slope of a line. %
For a bounded set $A\subseteq \RR^2$, we define $|A|_{\vt}=\inf\{b-a: a<b\in \RR, A\subseteq [a,b]_{\RR}\}$ and define $|A|_{\hor}=\sum_{n\in \ZZ}\textrm{Leb}(A\cap \{n\}_{\RR})$, where $\textrm{Leb}$ here refers to one dimensional Lebesgue measure on $\{n\}_{\RR}$. For a finite set $A\subseteq \ZZ^2$, we shall use $|A|$ to denote the cardinality of $A$. For $\theta\in \RR$, an unbounded set $A\subseteq \RR^2$ is said to be $\theta$-directed if for any sequence $(x_n,y_n)\in A$ with $|y_n|\rightarrow \infty$, we have $\lim_{n\rightarrow \infty}x_n/y_n=\theta$.
  Throughout the paper, to avoid unnecessary clutter, we shall use the same notations $T_p^{q,t}$ and $\Gamma_p^{q,t}$ to denote passage times and geodesics for both dynamical exponential LPP and dynamical BLPP.

\paragraph{\textbf{Acknowledgements}} We thank Riddhipratim Basu for the discussions and Duncan Dauvergne for the email exchange regarding the extension of results from \cite{Dau23} to BLPP. The author acknowledges the partial support of the NSF grant DMS-2153742 and the MathWorks fellowship.

\section{Last passage percolation preliminaries}
\label{sec:usef-estim-last}
In this section, we shall collect certain useful results and estimates related to BLPP.
\subsection{The line ensemble $\cP$ associated to BLPP}
\label{sec:ensemble}
BLPP comes associated with a non-intersecting line ensemble that will be very useful for us, and we now discuss this. For a point $(x,m)\leq (y,n)\in \ZZ_\RR$ and $k\in \NN$, we first consider the set $\Pi_{(x,m)}^{(y,n),k}$ consisting of tuples $\mathbf{\xi}=(\xi_1,\dots,\xi_k)$, where the $\xi_i$ are staircases from $(x,m)$ to $(y,n)$ with the additional property that the sets $\xi_i\cap \RR_{(x,y)}$ are mutually disjoint. Then we define
\begin{equation}
  \label{eq:338}
  T( (x,m)^k;(y,n)^k)=\sup_{\mathbf{\xi}\in \Pi_{(x,m)}^{(y,n),k}}\sum_{i=1}^k\wgt(\xi_i).
\end{equation}
Now, for $k\in [\![1,n+1]\!]$ and $x\geq 0$, we define $P_{k,n}(x)= T( \0^k;(x,n)^k)- T( \0^{k-1};(x,n)^{k-1})$, where we use the convention $T( \0^0;(x,n)^0)=0$. The utility of the above is that the line ensemble $\{P_{k,n}\}_{k=1}^{n+1}$ turns out to have the same law \cite{OY02} as a Dyson's Brownian motion-- a sequence of $k+1$ independent Brownian motions conditioned not to intersect. The following moderate deviation estimate for BLPP, obtained via the connection to Dyson's Brownian motion, will be useful throughout the paper.
\begin{proposition}[{\cite{LR10},\cite[Theorem 3.1]{DV21+}}]
  \label{prop:47}
  Fix $k\in \NN$. Then for some constants $C_1,C_2,c_1,c_2$ depending on $k$, for all $x>0$ and for all $0 < \alpha < 5n^{2/3}$, we have
  \begin{align}
    \label{eq:578}
    \PP(P_{k,n}(x)\geq 2\sqrt{nx}+ \alpha\sqrt{x}n^{-1/6}) &\leq C_1e^{-c_1\alpha^{3/2}},\nonumber\\
    \PP(P_{k,n}(x)\leq 2\sqrt{nx}- \alpha\sqrt{x}n^{-1/6}) &\leq C_2e^{-c_2\alpha^{3}}.
  \end{align}
  Further, for some $k$-dependent constants $C_3,c_3$, all $\alpha \geq 5n^{2/3}$ and all $n$, we have
  \begin{equation}
    \label{eq:579}
    \PP(|P_{k,n}(x)-2\sqrt{nx}|\geq \alpha \sqrt{x} n^{-1/6})\leq C_3e^{-c_3\alpha^2 n^{-1/3}}.
  \end{equation}
\end{proposition}
Typically, one is interested in the values $P_{k,n}(y)$ with $y$ lying a $n^{2/3}$ window of $n$. For this reason, for $x\in [-n^{1/3}/2,\infty)$, it is convenient to define
\begin{equation}
  \label{eq:339}
  \cP_{k}(x)=n^{-1/3}(P_{k,n}(n+2n^{2/3}x)-2n-2n^{2/3} x),
\end{equation}
where we note that in the notation $\cP_k$, we have suppressed the dependency on the parameter $n$. Often, we shall write $\cP=\{\cP_{k}\}_{k=1}^{n+1}$. It can be checked that for all $x>-n^{1/3}/2$, the line ensemble is non-intersecting in the sense that one has the ordering $\cP_{1}(x)> \cP_{2}(x)> \dots > \cP_{n+1}(x)$.

The line ensemble $\cP$ encodes a wealth of information about the geodesic structure and is very useful. Further, in the limit $n\rightarrow \infty$, $\cP$ converges to the so-called Airy-line ensemble \cite{CH14}, which is an ensemble consisting of infinitely many non-interescting lines. Just as $\cP$ encodes passage times in BLPP, the Airy line ensemble encodes passage times in the directed landscape-- which is expected to be the universal scaling limit of all LPP models, and is known \cite{DV21} to be the scaling limit of BLPP and exponential LPP. It is often the case that one can convert problems regarding geodesics in BLPP and the directed landscape to understanding the above line ensembles, and this approach has been often been fruitful over the past decade with some examples being \cite{Ham20,Bha22,Dau23+,Bus24}.

By using Proposition \ref{prop:47} along with a Taylor expansion, one can obtain (see \cite[Proposition 4.6]{Heg21}) the following moderate deviation estimate for the line ensemble $\cP$.

\begin{lemma}
  \label{lem:74}
Fix $k\in \NN$. There exist $k$-dependent constants $C,c$ such that for all $x$ satisfying $|x|\leq n^{1/9}$ and all sufficiently large $n$ depending on $k$ and $\alpha>0$, we have

{
  \begin{equation}
    \label{eq:231}
    \PP(|\cP_{k}(x)+x^2|\geq \alpha)\leq Ce^{-c\alpha^{3/2}}.
  \end{equation}
\par}
\end{lemma}
{
\begin{proof}
For $|x|\leq n^{1/9}$, we have $n+2n^{2/3}x\asymp n$, and Taylor expansion gives
\[
\left|n^{-1/3}\left(2n+2n^{2/3}x-2\sqrt{n(n+2n^{2/3}x)}\right)-x^2\right|
\leq Cn^{-1/3}|x|^3\leq C.
\]
Apply Proposition \ref{prop:47} at the horizontal coordinate $n+2n^{2/3}x$ and use \eqref{eq:339}. In its moderate-deviation range, the upper tail has exponent $3/2$ and the lower tail has exponent $3$. In its large-deviation range, $\alpha^2n^{-1/3}\geq c\alpha^{3/2}$. These bounds give the result after absorbing the bounded centering error and small values of $\alpha$ into the constants.
\end{proof}
}

\subsection{Brownianity of the line ensemble $\cP$}
\label{sec:brown}
The primary reason why BLPP is often more tractable than exponential LPP is that the line ensemble $\cP$ enjoys the so-called Brownian Gibbs property \cite{CH14}. Roughly, this property states that if we start with a fixed set of intervals $\{[a_i,b_i]\}_{i=1}^{n+1}$ and only reveal each $\cP_i$ outside the respective interval $[a_i,b_i]$, then to construct the entirety of $\cP$, we set $\cP_i\lvert_{[a_i,b_i]}=B_i$, where the $\{B_i\}_{i=1}^{n+1}$ are independent Brownian Bridges respectively connecting $\cP_i(a_i)$ to $\cP_i(b_i)$ which are additionally conditioned on the event that the resulting $\cP$ be a non-intersecting line ensemble.

It turns out, that by exploiting the above resampling property, it can be shown that for a fixed $k\in \NN$, each individual line $\cP_k$ itself locally ``looks like'' Brownian motion in the sense of local absolute continuity. In the past few years, there have been a series of works \cite{Ham22, CHH23, Dau23} establishing rigorous and progressively stronger versions of the above and a particularly fruitful strategy has been to use such comparison results to estimate probabilities of events for $\cP$ via a corresponding calculation for Brownian motion.

In this work, we shall also require such a Brownian comparison result. Specifically, we shall need certain recently proven Brownianity estimates from \cite{Dau23}-- a work in the setting of the Airy line ensemble as opposed to the prelimiting BLPP line ensemble $\cP$. It turns out that the arguments from \cite{Dau23} for the Airy line ensemble can be adapted to yield corresponding results for the line ensemble $\cP$ as well, and we give outline these adaptations in an appendix (Section \ref{sec:brownian-regularity}). The main results from this appendix which we shall require are Proposition \ref{prop:24} and Proposition \ref{prop:25}, and these shall only be used to prove Proposition \ref{prop:26}, a twin peaks result for routed weight profiles in BLPP. Since the precise statements of the results from Section \ref{sec:brownian-regularity} are involved, we refrain from stating them here-- we suggest that the reader refers to the appendix later as needed.

\subsection{Invariance of static BLPP under Brownian scaling}
\label{sec:scaling}
Applying diffusive scaling to the constituent Brownian motions in static BLPP yields the following useful invariance statement.
\begin{proposition}
  \label{prop:46}
 For any $\beta>0$, as processes in $ (x,m)\leq (y,n)\in \ZZ_\RR$, we have the distributional equality
  \begin{equation}
    \label{eq:581}
    T_{(\beta x,m)}^{(\beta x, n)}\stackrel{d}{=} \sqrt{\beta} T_{(x,m)}^{(y,n)}.
  \end{equation}
\end{proposition}
\subsection{Transversal fluctuation estimates for BLPP}
\label{sec:transv-fluct-estim}
We shall also frequently need estimates controlling the deviation of geodesics in LPP from the straight line connecting their endpoints. Such estimates are by now standard and we now state a version for BLPP; recall from the notational comments earlier that for $A\subseteq \RR^2$, $B_r(A)\coloneqq \{(x,y): \exists (x',y)\in A \textrm{ satisfying } |x-x'|\leq r\}$ and that $\LL_p^q$ refers to the line joining $p$ and $q$.
\begin{proposition}[{\cite[Corollary 1.5]{GH20}}]
  \label{prop:38}
 There exist constants $C,c$ such that for all $n$ and all $\alpha\leq n^{1/10}$, we have
  \begin{equation}
    \label{eq:461}
    \PP(\Gamma_{\0}^{\n}\not\subseteq B_{\alpha n^{2/3}}(\LL_{\0}^{\n}))\leq Ce^{-c\alpha^3}.
  \end{equation}
\end{proposition}
The following mesoscopic transversal fluctuation estimate shall also be useful for us.
\begin{proposition}[{\cite[Lemma 2.4]{BBBK25}}]
  \label{prop:53}
  There exist constants $C,c,m_0$ such that for all $n$ large enough, all $m_0\leq m \leq n$, and all $\alpha\leq m^{1/10}$,
  \begin{equation}
    \label{eq:629}
   \PP (\Gamma_{\0}^{\n}\cap [0,m]_{\RR}\not\subseteq B_{\alpha m^{2/3}}(\LL_{\0}^{\mathbf{m}}))\leq Ce^{-c\alpha^3}.
  \end{equation}
\end{proposition}
\subsection{A useful result relating dynamical and static BLPP}
\label{sec:usef-result-relat}
Recall that, just after \eqref{eq:63}, for any sets $K_1,K_2\subseteq \RR^2$ and any finite interval $[s,t]\subseteq \RR$, we had defined the set $\cT_{K_1}^{K_2,[s,t]}$. By using the definition of the discrete resampling dynamics on BLPP, it is easy to see that $\cT_{K_1}^{K_2,[s,t]}\sim \mathrm{Poi}((t-s)|\cM_{K_1}^{K_2}|)$ and thus $\cT_{K_1}^{K_2,[s,t]}$ is a.s.\ a finite set. The following simple result shall be very useful for extracting information about dynamical BLPP using results on static BLPP.
\begin{lemma}
  \label{prop:48}
  Fix a finite interval $[s,t]\subseteq \RR$ and bounded sets $K_1,K_2\subseteq \RR^2$. Then conditional on the random finite set $\cT_{K_1}^{K_2,[s,t]}$, for any $r\in \cT_{K_1}^{K_2,[s,t]}$, $T^r$ is distributed as a static BLPP.
\end{lemma}
\subsection{Uniqueness of geodesics in dynamical BLPP}
For static BLPP, it is true \cite[Lemma B.1]{Ham19} that almost surely, for all rational points $p\leq q\in \ZZ_\RR$ the geodesic $\Gamma_p^q$ is unique. In fact, the same holds for dynamical BLPP as well, and we now record this for later use.
\label{sec:uniq-geod-dynam}
\begin{lemma}
  \label{lem:104}
  Almost surely, for all rational points $p\leq q\in \ZZ_\RR$ and for all $t\in \RR$, there is unique geodesic $\Gamma_p^{q,t}$.
\end{lemma}

\begin{proof}
By a countable union argument, it suffices to work with fixed points $p\leq q\in \ZZ_\RR$. Again, by a countable union argument, we need only show that for any fixed $K>0$, there is a unique geodesic $\Gamma_p^{q,t}$ for all $t\in [-K,K]$. Consider the finite set of times $\cT_{p}^{q,[-K,K]}$ at which some $(i,m)\in \cM_{p}^{q}$ gets resampled. Now, by Lemma \ref{prop:48}, we know that conditional on the set $\cT_{p}^{q,[-K,K]}$, for any $t\in \cT_{p}^{q,[-K,K]}$, $T^t$ is distributed as a static BLPP. However, for static BLPP, we already know that the geodesic $\Gamma_p^q$ is almost surely unique. This completes the proof.
\end{proof}

\subsection{Directedness of infinite geodesics in dynamical BLPP}
\label{sec:dynam-blpp-geod}
We shall now discuss semi-infinite geodesics in BLPP, by which we mean any semi-infinite staircase $\xi$ emanating from a point $p\in \ZZ_\RR$ such that any segment of it is a geodesic between its endpoints. For static BLPP, it can be shown \cite[Theorem 3.1]{SS23} that any semi-infinite geodesic is $\theta$-directed for some $\theta\in [0,\infty]$ and that, almost surely, simultaneously for all $\theta\in [0,\infty]$ and $p\in \ZZ_\RR$, a $\theta$-directed semi-infinite geodesics emanating from $p$ exists. 

 A priori, it is not clear whether the same holds uniformly in time for dynamical LPP; for dynamical BLPP. The following result which we shall establish later in an appendix, shows that the above indeed does hold.
\begin{proposition}
  \label{prop:21}
  Almost surely, for all $t\in \RR$, every semi-infinite geodesic $\Gamma^t$ is $\theta$-directed for some $\theta\in [0,\infty]$. Further, for any fixed $p\in \ZZ_\RR$, almost surely, simultaneously for every $\theta\in [0,\infty]$ and $t\in \RR$, there exists a $\theta$-directed semi-infinite geodesic emanating from $p$.
\end{proposition}
In this paper, the primary objects of interest are bigeodesics as opposed to semi-infinite geodesics and the following result for bigeodesics shall be very useful to us.
 \begin{proposition}
   \label{prop:22}
  Fix $\varepsilon>0$. Almost surely, for all $t\in \RR$, any non-trivial bigeodesic $\Gamma^t$ is $\theta$-directed for some $\theta\in (0,\infty)$ and satisfies $\Gamma^t\cap [-n,n]_{\RR}\subseteq B_{n^{2/3+\varepsilon}}(\LL_{(-\theta n,-n)}^{(\theta n,n)})$ for all $n$ large enough. %
\end{proposition}
Note that the above, in particular, states that non-trivial bigeodesics in dynamical BLPP are never axially directed, that is, their corresponding angle $\theta$ is never equal to $0$ or $\infty$. In an appendix (Section \ref{sec:direction}), we will provide the proofs of Propositions \ref{prop:21}, \ref{prop:22}-- the proof of Proposition \ref{prop:21}, consists of adapting the classical argument by Newman \cite{New95} and Howard-Newman \cite{HN01} for first passage percolaion to the dynamical BLPP case, where now one needs to ensure that the transversal fluctuation estimates used therein all hold uniformly in time. For the part of Proposition \ref{prop:22}, where we rule out axially directed non-trivial bigeodesics, we shall undertake an adaptation of the corresponding arguments for the static exponential LPP case from \cite[Section 5]{BHS22}.

\subsection{Routed distance profiles and an estimate for their number of peaks}
\label{s:routed}
Understanding the structure of near-geodesics, or paths which are close to being geodesics but are not quite geodesics, shall be crucial for the proofs of Theorems \ref{thm:3}, \ref{thm:5}. To do so, the following definition will be useful-- for points $p_1\leq q \leq  p_2\in \ZZ_\RR$, we define
\begin{equation}
  \label{eq:580}
  Z_{p_1}^{p_2}(q)= T_{p_1}^{q}+ T_{q}^{p_2},
\end{equation}
that is, $Z_{p_1}^{p_2}(q)$ refers to the optimal weight of a staircase $\xi\colon p_1\rightarrow p_2$ which in addition is forced to pass via $q$.
Typically, we shall fix an $m\in \ZZ$ and look at the profile $x\mapsto Z_{p_1}^{p_2}(x,m)$-- borrowing terminology from \cite{GH24}, we refer to the above as a routed weight profile. Later, shall require an estimate showing that the routed weight profile $Z_{\0}^{\n}(\cdot,m)$ does not have too many well-separated peaks for any $m\in [\![0,n]\!]$. Such an estimate was shown for exponential LPP in \cite[Proposition 3.10]{SSZ24} and a similar result, but for BLPP weight profiles (as opposed to routed weight profiles), was shown in \cite[Theorem 1.5]{CHH23}. We now give a definition and then state the result that we shall require. For any fixed $m$, let $\peak(\alpha)$ denote the set of $(i,m)\in \cM_{\0}^{\n}$ for which there exists $x\in [i,i+1]$ such that $|T_{\0}^{\n}-Z_{\0}^{\n}(x,m)|\leq \alpha$. Note that in the above, it is of course needed that $\0\leq (x,m)\leq \n$ and as a result, in the above, we must have $[i,i+1]\subseteq [-1,n+1]$. As a result, for any $m\in \ZZ$, we have the deterministic inequality
\begin{equation}
  \label{eq:589}
  |\peak(\alpha)\cap \{m\}_\RR|\leq n+2.
\end{equation}
In fact, the above quantity is typically much smaller and the following estimate in this direction will be useful for us.
\begin{proposition}
  \label{lem:31}
Fix $\delta>0$. Then for all $n$ large enough, with probability at least $1-Ce^{-cn^{3\delta/4}}$, we have $|\peak(n^{\delta})\cap \{m\}_{\RR}|\leq n^{1200\delta}$ for all $m\in [\![0,n]\!]$.
\end{proposition}

Given the well-developed understanding of the Brownianity of BLPP weight profiles, the above result is not hard to obtain-- we shall provide a proof in an appendix (Section \ref{sec:app-peaks}). Throughout the paper, and especially for proving Proposition \ref{lem:31}, we shall also need to work with a close variant of the routed distance profile $Z_{p_1}^{p_2}$ that we defined in this section. Indeed, for points $p_1\leq q\leq q+(0,1)\leq p_2\in \ZZ_\RR$, we define
\begin{equation}
  \label{eq:688}
  Z_{p_1}^{p_2,\bullet}(q)= T_{p_1}^{q}+ T_{q+(0,1)}^{p_2}.
\end{equation}
Again, for $m\in [\![0,n-1]\!]$, we shall often work with the process $x\mapsto Z_{\0}^{\n,\bullet}(x,m)$, and the advantage now is that the above process a sum of two independent ``locally Brownian processes'', and is thus itself locally absolutely continuous to Brownian motion. We note that this property is not true for the routed profile $x\mapsto Z_{\0}^{\n}(x,m)$.%

\subsection{A twin peaks estimate for routed weight profiles}
\label{sec:twin-peaks-out}
As we shall outline in Section \ref{sec:out-fixed}, for the proof of Theorems \ref{thm:3}, \ref{thm:5}, it will be important to argue that the profile $x\mapsto Z_{\0}^{\n,\bullet}(x,m)$ cannot have many well-separated peaks. Specifically, we shall need the following result.
\begin{proposition}
  \label{prop:26}
Fix $\beta'\in (0,1/2)$ and $\delta\in (0,1/6)$. For all $\ell\leq n$ and all $m\in [\![\beta'n, (1-\beta')n]\!]$, %
consider the event $\TP_{\ell,m}$ defined by
  \begin{equation}
    \label{eq:305}
    \TP_{\ell,m}=\{\exists x:
|x-\Gamma_{\0}^{\n}(m)|\geq \ell^{2/3-\delta}, |T_\0^\n-Z_{\0}^{\n,\bullet}(x,m)|\leq
\ell^{\delta}\}.
\end{equation}
Then there exists a constant $C$ such that for all $n$ large enough and all $\ell,m$ as above, we have
  \begin{equation}
    \label{eq:299}
      \PP(\TP_{\ell,m})\leq C \ell^{-1/3+2\delta}.
  \end{equation}
\end{proposition}
Results of the above type are by now standard, and the basic idea is to exploit the local absolute continuity of the routed profile with Brownian motion and do a computation for the latter. In particular, the result \cite[Theorem 1.3]{GH20} is for routed distance profiles in BLPP and is very similar to the above. However, there are still a significant subtle difference because of which we need to provide an argument for Proposition \ref{prop:26}. Namely, \eqref{eq:305} considers all $x$ satisfying the lower bound $|x-\Gamma_{\0}^{\n}(m)|\geq \ell^{2/3-\delta}$ and in particular, this not require any upper bound on $x$; for instance, it also covers considerably large values of $|x-\Gamma_{\0}^{\n}(m)|$, e.g.\ $|x-\Gamma_{\0}^{\n}(m)|\sim \ell^{2/3} \log^{1/3}(\ell)$, which the estimate \cite[Theorem 1.3]{GH20} does not cover since the profile $y\mapsto Z_{\0}^{\n}(y,m)$ cannot be effectively compared to Brownian motion over such large stretches of $y$, in the sense that the associated Radon-Nikodym derivative is rather large.

As mentioned in Section \ref{sec:brown}, in order to prove Proposition \ref{prop:26}, we shall first, in an appendix (Section \ref{sec:brownian-regularity}), obtain a BLPP version of certain recently proven Brownianity estimates \cite{Dau23} for weight profiles in the directed landscape. Subsequently, in another appendix (Section \ref{sec:twin-peaks}), we shall prove Proposition \ref{prop:26} with the help of the above.

\section{Outline of the proofs}
\label{sec:outline-proofs}
In this section, we give a detailed outline of the proofs of all the main results of this paper.
 \subsection{Estimating the expected number of geodesic switches}
 \label{sec:out-fixed}
The first goal is to outline the proof of Theorem \ref{prop:30}. Let us first give a heuristic explanation of why we expect the above to hold-- since the intuition is very discrete in nature, we work with exponential LPP to fix ideas and shall later shift to BLPP, which is the model for which the formal statements of Theorem \ref{prop:30} and Theorem \ref{thm:5} hold. %
 Locally, we consider the quantity
 \begin{equation}
   \label{eq:570}
   S_{-\n}^{\n,[s,t]}= \sum_{r\in [s,t]} |\Gamma_{-\n}^{\n,r}\setminus \Gamma_{-\n}^{\n,r^-}|,
 \end{equation}
which represents geodesic switches in exponential LPP. Note that in \eqref{eq:570}, the sum though seemingly over a continuous set, is in fact discrete since there are only finitely many $r\in [s,t]$ for which the summand is non-zero. Now, let us intuitively discuss why we expect to have
 \begin{equation}
   \label{eq:554}
   \EE S_{-\n}^{\n,[s,t]}\leq n^{5/3+\delta}(t-s).
 \end{equation}
First, by the KPZ 1:2:3 exponents, we know that for any fixed $r$, the geodesic $\Gamma_{-\n}^{\n,r}$ stays in an $O(n^{2/3})$ spatial window around the line $\LL_{-\n}^{\n}$. To advance the heuristic discussion, let us assume that such a geodesic always stays in the set $B_{n^{2/3}}(\LL_{-\n}^{\n})$, which is a set that contains $n^{5/3}$ many lattice points-- thus, we assume that the geodesic is determined by the weights of the above vertices.

 With the above in mind, suppose that we are working with the static exponential LPP model and now, uniformly at random, we resample the weight of a uniform point $\fp\in B_{n^{2/3}}(\LL_{-\n}^{\n})$-- let us locally use $\Gamma_{-\n}^{\n}$ and $\Gamma_{-\n}^{\n,+}$ to denote the geodesic before and after the above resampling. The question now is -- what is the expectation $\EE |\Gamma_{-\n}^{\n,+}\setminus \Gamma_{-\n}^{\n}|$? As we shall explain now, we expect the above to be an $O(1)$, or more formally, an $O(n^\delta)$ quantity. Note that this would suffice for proving \eqref{eq:554} since by the previous paragraph, in time $t-s$, in expectation, we have $O((t-s)n^{5/3})$ many updates in the region $B_{n^{2/3}}(\LL_{-\n}^{\n})$. Thus, we would have
 \begin{equation}
   \label{eq:555}
   \EE S_{-\n}^{\n,[s,t]}\leq  \EE [|\Gamma_{-\n}^{\n,+}\setminus \Gamma_{-\n}^{\n}|]O((t-s)n^{5/3}) \leq (t-s)n^{5/3+\delta}.
 \end{equation}
 
 We now focus on the quantity $\EE |\Gamma_{-\n}^{\n,+}\setminus \Gamma_{-\n}^{\n}|$. For doing so, we shall need some notation-- for a point $p$ such that $-\n\leq p\leq \n$, we define the routed passage time
 \begin{equation}
   \label{eq:659}
   \cZ_{-\n}^{\n}(p)=T_{-\n}^{p}+T_{p}^{\n}-\omega_p,
 \end{equation}
 which we note is simply the weight of the best path from $-\n$ to $\n$ which in addition is forced to go via $p$; note that $\omega_p$ is subtracted to avoid counting it twice. In order to have $\EE |\Gamma_{-\n}^{\n,+}\setminus \Gamma_{-\n}^{\n}|\neq 0$, we must have $|T_{-\n}^{\n}- \cZ_{-\n}^{\n}(\fp)|\leq |\omega_\fp-\omega^+_{\fp}|$
 , where we are using $\omega^+_{\fp}$ to denote the weight obtained after resampling. Since the exponential distribution has light tails, intuitively, we can think of all of $\omega_\fp,\omega^+_{\fp}$ and $|\omega_\fp-\omega^+_{\fp}|$ as $O(1)$ quantities. Thus, in order to have $\EE |\Gamma_{-\n}^{\n,+}\setminus \Gamma_{-\n}^{\n}|\neq 0$, we must have
 \begin{equation}
   \label{eq:548}
   |T_{-\n}^\n- \cZ_{-\n}^{\n}(\fp)|=O(1).
 \end{equation}
 that is, we need to have a ``twin-peak'' event in the static LPP environment. %
Indeed, since $T_{-\n}^{\n}=\max_{q: \phi(q)=\phi(\fp)}\cZ_{-\n}^{\n}(q)$, and since $\fp\notin \Gamma_{-\n}^{\n}$, \eqref{eq:548} expresses that on the line $\{q: \phi(q)=\phi(\fp)\}$, the value of the routed weight profile $q\mapsto \cZ_{-\n}^{\n}(q)$ at the point $q=\fp$ is within $O(1)$ of its global maximum.

 In order to estimate $\EE [|\Gamma_{-\n}^{\n,+}\setminus \Gamma_{-\n}^{\n}|]$, we now do a conditioning on the location of the point $\fp$ with respect to the geodesic $\Gamma_{-\n}^{\n}$; to be specific, for $k\in \NN$, we condition on the event $F_k=\{|\Gamma_{-\n}^{\n}(\phi(\fp)) - \psi(\fp)|=k\}$, which simply asks that $\fp$ be exactly at a spatial distance $k$ from the geodesic. Now, conditional on the above, if we do have $|\Gamma_{-\n}^{\n,+}\setminus \Gamma_{-\n}^{\n}|\neq 0$, then intuitively, by the KPZ 1:2:3 scaling, we should have (see Figure \ref{fig:outlinetwin})
 \begin{figure}
   \centering
   \includegraphics[width=0.6\linewidth]{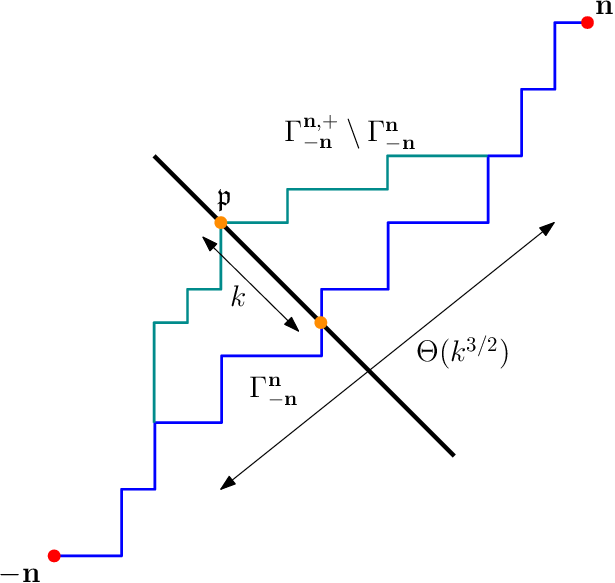}
   \caption{Here, the weight of a point $\mathfrak{p}$ at a horizontal distance $k$ from the geodesic $\Gamma_{-\n}^{\n}$ has been resampled. For the geodesic to undergo a change, that is, in order to have $\Gamma_{-\n}^{\n}\neq \Gamma_{-\n}^{\n,+}$, it is extremely likely that a ``twin-peaks'' event has to occur on the anti-diagonal line passing through $\mathfrak{p}$. That is, we must have $|T_{-\n}^{\n}-\mathcal{Z}_{-\n}^{\n}(\mathfrak{p})|=O(1)$; by heuristic calculations for a random walk conditioned to be positive, we expect the above probability to be $\Theta(k^{-3/2})$. Now, if we indeed have $\Gamma_{-\n}^{\n}\neq \Gamma_{-\n}^{\n,+}$, then by the KPZ 1:2:3 scaling, we expect $|\Gamma_{-\n}^{\n,+}\setminus \Gamma_{-\n}^{\n}|=\Theta(k^{3/2})$. As a result, we should have $\EE |\Gamma_{-\n}^{\n,+}\setminus \Gamma_{-\n}^{\n}|= \Theta( k^{-3/2}\times k^{3/2})=\Theta(1)$, which we note does not depend on $k$.}
   \label{fig:outlinetwin}
 \end{figure}
 \begin{equation}
   \label{eq:556}
   |\Gamma_{-\n}^{\n,+}\setminus \Gamma_{-\n}^{\n}|\approx k^{3/2}.
 \end{equation}
 Thus, in view of \eqref{eq:548}, we should have
\begin{align}
  \label{eq:549}
  \EE [|\Gamma_{-\n}^{\n,+}\setminus \Gamma_{-\n}^{\n}|\big\lvert F_k, \phi(\fp)]\leq k^{3/2}\PP(|T_{-\n}^{\n}- \cZ_{-\n}^{\n}(\fp)|=O(1)\lvert F_k,\phi(\fp))
\end{align}
Now for simplicity of notation, we assume $\phi(\fp)=0$-- it shall be evident later that the upcoming discussion is valid verbatim for all $\phi(\fp)$ bounded away and between $-2n$ and $2n$. With $\mathfrak{f}_{-\n}^{\n}$ denoting the one dimensional profile defined by $\mathfrak{f}_{-\n}^{\n}(x)=\cZ_{-\n}^{\n}((x,-x))$, we have
\begin{align}
  \label{eq:552}
  &\PP(|T_{-\n}^{\n}- \cZ_{-\n}^{\n}(\fp)|=O(1)\lvert F_k,\phi(\fp))\nonumber\\
  &\leq  \PP( \max \mathfrak{f}_{-\n}^{\n}(x) - \mathfrak{f}_{-\n}^{\n}( \argmax \mathfrak{f}_{-\n}^{\n} + k)=O(1)).
\end{align}
Thus, the task now is to estimate the probability above, where the weight profile $\mathfrak{f}_{-\n}^{\n}$ admits a near maximum at a distance $k$ from its global maximizer. Unfortunately, for exponential LPP, the local behaviour of such weight profiles is not well-understood and as a result, it appears difficult to estimate the above probability. In BLPP and in the directed landscape, such routed distance profiles (see Proposition \ref{prop:52}) have locally Brownian behaviour owing to the Brownian Gibbs property (see Section \ref{sec:brown}) and thus one might expect that for the discrete exponential LPP model, the profiles $\mathfrak{f}_{-\n}^{\n}$ above to locally behave as a simple random walk and thereby, the profile $x\mapsto \mathfrak{f}_{-\n}^{\n}(\argmax \mathfrak{f}_{-\n}^{\n} + x)$ to behave as a simple random walk around its maximum. By a calculation for walks conditioned to stay positive, it can be computed that the random walk probability above is $\Theta(k^{-3/2})$. As a result of this calculation and \eqref{eq:549}, it can be obtained that $\EE [|\Gamma_{-\n}^{\n,+}\setminus \Gamma_{-\n}^{\n}|\big\lvert F_k] \leq k^{3/2}\times k^{-3/2} = O(1)$ for all $k$ and thereby $\EE [|\Gamma_{-\n}^{\n,+}\setminus \Gamma_{-\n}^{\n}|]=O(1)$ which is what we set out to show.

As mentioned earlier, the above picture is not rigorous since the above comparison of routed weight profiles to random walks is not available for exponential LPP. For this reason, we shall instead work with BLPP where, due to the Brownian Gibbs property, a corresponding comparison to Brownian motion can be made. This is the primary reason why Theorems \ref{thm:3}, \ref{thm:5} are proved in the setting of BLPP as opposed to exponential LPP.

Further, we caution that the simplistic picture presented above is heuristic and the actual argument proceeds considerably differently due to additional difficulties-- the primary among them being the misleading expression \eqref{eq:556}. Indeed, in the notation used above, it is possible that conditional on $F_k$ and on $\{|\Gamma_{-\n}^{\n,+}\setminus \Gamma_{-\n}^{\n}|\neq 0\}$, the quantity $k^{-3/2}|\Gamma_{-\n}^{\n,+}\setminus \Gamma_{-\n}^{\n}|$ has heavy enough tails such that its expectation grows as a power of $n$. In order to handle this, we take a different route where we do an additional averaging argument using that the geometry of the excursion $\Gamma_{-\n}^{\n,+}\setminus \Gamma_{-\n}^{\n}$ cannot be too ``thin''-- for this, in part, we use estimates from \cite{GH20} (see Proposition \ref{prop:10}).
\subsection{Upper bounding the size of the hitset by geodesic switches}
\label{sec:out-ubhitset}
We now outline the proof of Theorem \ref{thm:6}; first let us discuss the simpler case when one looks at the hitset between the points $-\n$ and $\n$ as opposed to the segments $L_{-n}$ and $L_{n}$; also, we set $\gamma=1/2$. The following relation between hitsets and switches shall form the backbone of our strategy:
 \begin{equation}
   \label{eq:544}
   |\hitset_{-\n}^{\n,[s,t]}([\![-n/2,n/2]\!]_\RR)|\leq |\hitset_{-\n}^{\n,\{s\}}([\![-n/2,n/2]\!]_\RR)|+ \switch_{-\n}^{\n,[s,t]}([\![-n/2,n/2]\!]_\RR).
 \end{equation}
We refer the reader to Figure \ref{fig:coarse} for a depiction of an inequality of the above type. Indeed, the reasoning behind the above inequality is as follows-- any interval $\{m\}_{[i,i+1]}$ that is hit by a geodesic from $-\n$ to $\n$ during the dynamical time interval $[s,t]$ must either already be hit by such a geodesic in the static environment $T^s$ or there must exist a time $r\in (s,t]$ for which $(i,m)\in \coarse(\Gamma_{u}^{v,r})\setminus \coarse(\Gamma_{u}^{v,r^-})$. In other words, if we define the field $(i,m)\mapsto H_n(i,m)$ by
 \begin{equation}
   \label{eq:538}
   H_n(i,m)=\#\{r\in [s,t]: (i,m)\in \coarse(\Gamma_{u}^{v,r})\setminus \coarse(\Gamma_{u}^{v,r^-})\},
 \end{equation}
 then, since $H_n$ is integer valued, we have the identity $\mathbbm{1}(H_n(i,m)>0)\leq H_n(i,m)$ for all $(i,m)$. Now, one can take an expectation on both sides of the above and additionally summing up over $(i,m)\in [\![-n/2,n/2]\!]_{\ZZ}$ would yield \eqref{eq:544}.

Now, by using that $\Gamma_{-\n}^{\n,s}$ is a staircase from $-\n$ to $\n$, it is easy to see that there is a fixed $C>0$ for which one has the deterministic bound
 \begin{equation}
   \label{eq:545}
   \EE|\hitset_{-\n}^{\n,\{s\}}([\![-n/2,n/2]\!]_\RR)|\leq Cn.
 \end{equation}
 Combining the above with \eqref{eq:544} and Theorem \ref{prop:30} yields that for any fixed $\varepsilon>0$, for all $n$ large enough, we have
 \begin{equation}
   \label{eq:2}
   \EE|\hitset_{-\n}^{\n,[s,t]}([\![-n/2,n/2]\!]_\RR)|\leq n^{1+\varepsilon}+n^{5/3+\varepsilon}(t-s),
 \end{equation}
and this is a simplified version of Theorem \ref{thm:6}, where the segments $L_{-n},L_{n}$ are replaced by $-\n,\n$ respectively.

 \subsection{Accessing geodesics between on-scale segments by those between sprinkled Poissonian points}
\label{sec:out-segment}
Having obtained \eqref{eq:2}, the goal now is to discuss how one could upgrade to Theorem \ref{thm:6}.
That is, we want to go from a point-to-point estimate to a corresponding estimate between two segments of length $n^{2/3}$ each-- the intuition being that of geodesic coalescence. Indeed, intuitively (see \cite[Theorem 3.10]{BHS22} for a statement in exponential LPP), the set of all geodesics from $L_{-n}$ to $L_{n}$ coalescence into $O(1)$ many segments in the middle region $[-n/2,n/2]_{\RR}$ and thus, up to constants, we expect the point-to-point estimate to hold unchanged modulo an extra multiplicative constant. However, the above is not rigorous since the $O(1)$ geodesics above could possibly be highly exceptional and thus not behave at all as geodesics between fixed points.

Now, for a moment, let us pretend that there exists a deterministic family of pairs of points $\{(p_i,q_i)\}_{i\in \cI}$ with $|\cI|\leq n^{o(1)}$ for which we are assured that for each $r\in [s,t]$, we always have
\begin{equation}
  \label{eq:720}
  \bigcup_{p\in L_{-n},q\in L_{n}}\Gamma_p^{q,r}\cap [\![-n/2,n/2]\!]_{\RR}\subseteq \bigcup_{i\in \cI}\Gamma_{p_i}^{q_i,r}\cap [\![-n/2,n/2]\!]_{\RR}.
\end{equation}
If the above holds, then we can directly upgrade \eqref{eq:2} to Theorem \ref{thm:6} as we would a.s.\ have
\begin{equation}
  \label{eq:719}
  |\hitset_{L_{-n}}^{L_n,[s,t]}([\![-n/2,n/2]\!]_{\RR})|\leq \sum_{p\in \cI}|\hitset_{p_i}^{q_i,[s,t]}([\![-n/2,n/2]\!]_{\RR})|,
\end{equation}
and could take expectations of both sides, where we note that the right hand side now only considers point-to-point hitsets. While it is too much to hope that a \emph{deterministic} family $\{(p_i,q_i)\}_{i\in \cI}$ as above exists, it turns out that if one instead sprinkles the points $\{(p_i,q_i)\}_{i\in \cI}$ according to a appropriate Poisson process that is \emph{independent} of the dynamical BLPP, then with superpolynomially high probability, a version of \eqref{eq:720} does hold, and this is enough for our application. In the remainder of this section, we discuss the above in more detail, and in order to elucidate the ideas, we shall again work in the simpler setting of dynamical exponential LPP instead of dynamical BLPP. Correspondingly, instead of the sets $L_n$, we shall work with $\ell_{n}=\{z: \phi(z)=2n, |\psi(z)|\leq |n|^{2/3}\}$.

Suppose that we want to control the cardinality of the set
\begin{equation}
  \label{eq:722}
  \bigcup_{p\in \ell_{-n}q\in \ell_{n},r\in [0,1]}(\Gamma_p^{q,r}\cap \{z: |\phi(z)|\leq n\}),
\end{equation}
and want to use the strategy alluded to in the previous paragraph for doing so.
The key ingredient is the following coalescence estimate for exponential LPP proved recently in the work \cite{BB23}; note that in the following, we fix $\mu\in (0,1)$ and use $S_{n,\mu}$ to denote the set of all $(p,q)$ such that $\phi(p)\in [-2n,-3n/2]$ and $\phi(q)\in [3n/2,2n]$ with $\slope(p,q)\in (\mu^{-1},\mu)$. %

\begin{proposition}[{\cite[Proposition 53]{BB23}}]
  \label{prop:45}
For any $(p,q)\in S_{n,\mu}$, let $\underline{V}^{\mathrm{ELPP}}_n(p,q)$ denote the set of points $z$ which are to the right of $\Gamma_p^q$ and satisfy $\Gamma_z^q\cap \{z: |\phi(z)|\leq n\}=\Gamma_p^q\cap \{z: |\phi(z)|\leq n\}$. Then there exist positive constants $C,c,K,\theta,\alpha$ such that for all $\varepsilon^{\alpha} n\geq K$, we have
  \begin{equation}
    \label{eq:564}
    \PP(|\underline{V}^{\mathrm{ELPP}}_n(p,q)|\leq \varepsilon n^{5/3})\leq Ce^{-c\varepsilon^{-\theta}}.
  \end{equation}
\end{proposition}
\begin{figure}
  \centering
  \includegraphics[width=0.6\linewidth]{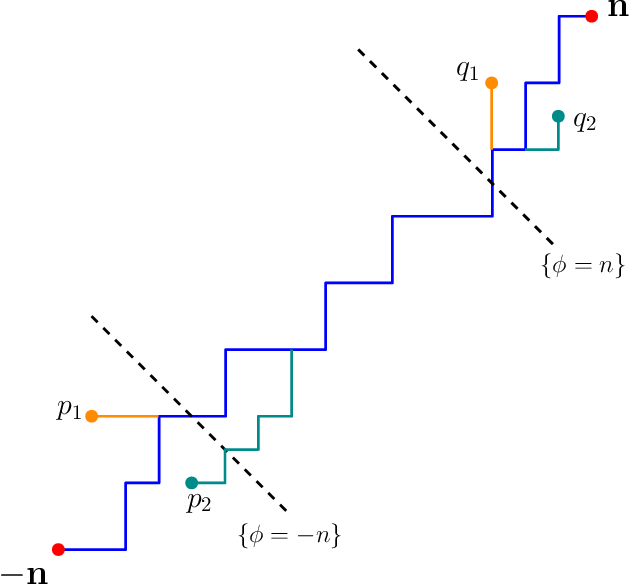}
  \caption{Here, we have $(p_1,q_1)\in \basin_n^{\mathrm{ELPP}}(\Gamma_{-\n}^{\n})$ since the geodesic $\Gamma_{p_1}^{q_1}$ only disagrees with $\Gamma_{-\n}^{\n}$ outside the region between the dashed lines. This should be contrasted with the pair $(p_2,q_2)$ which does not belong to the set $\basin_n^{\mathrm{ELPP}}(\Gamma_{-\n}^{\n})$. With some work, Proposition \ref{prop:45} can be used to show that $\basin_n^{\mathrm{ELPP}}(\Gamma_{-\n}^{\n})\geq \varepsilon^2 n^{10/3}$ with stretched exponentially high probability in $\varepsilon^{-1}$.}
  \label{fig:outlinebasin}
\end{figure}
The utility of the above estimate is that, with some effort, it translates to a bound on the `basin of attraction' of a geodesic, which we now define. For $(p,q)\in S_{n,\mu}$, we define $\basin_n^{\mathrm{ELPP}}(\Gamma_p^q)$ as the set of points $(p',q')\in S_{n,\mu}$ such that
\begin{equation}
  \label{eq:721}
  \Gamma_{p'}^{q'}\cap \{z: |\phi(z)|\leq n\}=\Gamma_p^q\cap \{z: |\phi(z)|\leq n\}.
\end{equation}
We refer the reader to Figure \ref{fig:outlinebasin} for an illustration of the above definition. With some work (see Proposition \ref{prop:12} for the corresponding BLPP result), it can be shown that for all $(p,q)\in S_{n,\mu}$, as long as $\varepsilon \geq n^{-\zeta}$ for a constant $\zeta$, we have
\begin{equation}
  \label{eq:565}
  \PP(|\basin_n^{\mathrm{ELPP}}(\Gamma_p^q)|\leq \varepsilon^2 n^{10/3})\leq Ce^{-c\varepsilon^{-\theta}}.
\end{equation}
The above estimate is very useful since with superpolynomially high probability (say if we take $\varepsilon=n^{-\delta}$), it yields a large volume set of pairs $\basin_n^{\mathrm{ELPP}}(\Gamma_p^q)$ such that geodesics between them coalesce with the geodesic $\Gamma_p^q$ in the central region. In fact, since the decay in \eqref{eq:565} is sufficiently rapid, by a simple union bound argument, one can obtain a version of \eqref{eq:565} which holds uniformly in the dynamics, in the sense that there is an event $\cE_n$ with $\PP(\cE_n)\geq 1-Ce^{-cn^{\theta\delta}}$, on which we have
\begin{equation}
  \label{eq:566}
  |\basin_n^{\mathrm{ELPP}}(\Gamma_p^{q,r})|\geq n^{10/3-2\delta}
\end{equation}
for all $p\in \ell_{-n}, q\in \ell_n$ and $r\in [0,1]$. %
On this event $\cE_n$, the idea is to use a Poisson process of typical pairs to access all geodesics corresponding to all the geodesics between points in $\ell_{-n}$ and $\ell_n$. Indeed, let $\cQ_{n}^{\mathrm{ELPP}}$ be a Poisson point process on $\ZZ^2\times \ZZ^2$ with intensity $n^{-10/3+3\delta}$ which is independent of the dynamical exponential LPP. By the properties of Poisson processes, for any fixed set $A\subseteq \ZZ^2\times \ZZ^2$, $|\cQ_n^{\mathrm{ELPP}}\cap A|$ is distributed as $\mathrm{Poi}(n^{-10/3+3\delta}|A|)$. As a result, conditional on the event $\cE_n$, for each $p\in \ell_{-n},\ell_{n},r\in [0,1]$, we have the stochastic domination
\begin{equation}
  \label{eq:567}
  |\cQ_n^{\mathrm{ELPP}}\cap \basin_n^{\mathrm{ELPP}}(\Gamma_p^{q,r})|\stackrel{\mathrm{S.D.}}{\geq} \mathrm{Poi}(n^{\delta}),
\end{equation}
where we note that the $n^\delta$ above is obtained by multiplying $n^{10/3-2\delta}$ with $n^{-10/3+3\delta}$. As a result, by using the tails of the Poisson distribution along with a union bound over $p\in \ell_{-n},q\in \ell_n$, we have for some constants $C,c$
\begin{equation}
  \label{eq:568}
\PP(  \cQ_n^{\mathrm{ELPP}}\cap \basin_n^{\mathrm{ELPP}}(\Gamma_p^{q,r})\neq\emptyset\textrm{ for all } p\in \ell_{-n}, q\in \ell_{n},r\in [0,1]\big\lvert \cE_n)\geq 1-Ce^{-cn^{\delta}}.
\end{equation}
\begin{figure}
  \centering
  \includegraphics[width=0.7\linewidth]{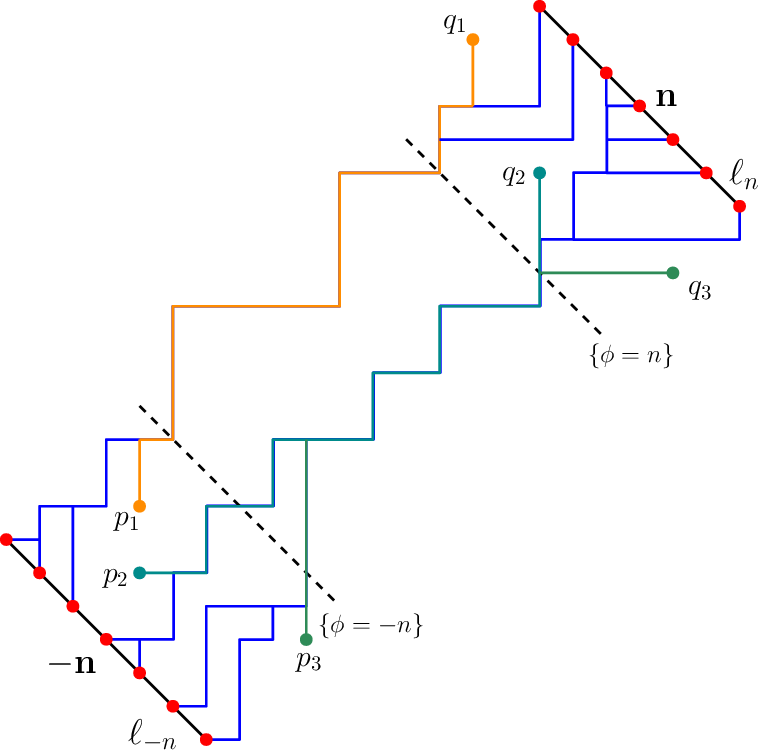}
  \caption{Here, the points $(p_1,q_1), (p_2,q_2), (p_3,q_3)$ all belong to the Poisson process $\cQ_n^{\mathrm{ELPP}}$, and we depict geodesics for $T^0$ between the points $u\in \ell_{-n}$ and $u+2\n\in \ell_{n}$. Here, the high probability event from \eqref{eq:568} occurs and thus the portion of all such geodesics $\Gamma_u^{u+2\n,0}$ in between the dotted lines is entirely covered by the union $\Gamma_{p_1}^{q_1,0}\cup \Gamma_{p_2}^{q_2,0}\cup\Gamma_{p_3}^{q_3,0}$.}
  \label{fig:outlinepoisson}
\end{figure}
Locally, we use $\cA_n$ to denote the event considered above. Now, by the definition of the set $\basin_n^{\mathrm{ELPP}}(\Gamma_p^q)$, on the event $\cA_n\cap \cE_n$, for any $p\in \ell_{-n},q\in \ell_{n}$, there exists $(p',q')\in \cQ_n^{\mathrm{ELPP}}\cap S_{n,\mu}$ for which \eqref{eq:721} holds (see Figure \ref{fig:outlinepoisson}). Further, by a transversal fluctuation argument, we can further assume on a high probability event $\mathrm{Tran}_{n}$ that $p',q'\in B_{n^{2/3+\delta}}(\LL_{-\n}^{\n})$. As a result, on the event $\cA_n\cap \cE_n\cap \mathrm{Tran}_{n}$, we have
  \begin{equation}
  \label{eq:723}
  \bigcup_{p\in \ell_{-n}q\in \ell_{n},r\in [0,1]}(\Gamma_p^q\cap \{z: |\phi(z)|\leq n\})\subseteq \bigcup_{(p',q')\in \cQ_n^{\mathrm{ELPP}}\cap S_{n,\mu}\cap B_{n^{2/3+\delta}}(\LL_{-\n}^{\n})^2} (\Gamma_p^q\cap \{z: |\phi(z)|\leq n\}).
\end{equation}
Finally, note that since $\cQ_n^{\mathrm{ELPP}}$ is a Poisson process of rate $n^{-10/3+3\delta}$, and since we have the inequality $|S_{n,\mu}\cap B_{n^{2/3+\delta}}(\LL_{-\n}^{\n})^2|\leq n^{10/3+2\delta}$, $|\cQ_n^{\mathrm{ELPP}}\cap B_{n^{2/3+\delta}}(\LL_{-\n}^{\n})^2|$ should typically contain at most $O(n^{5\delta})$ many pairs of points. Thus, we have obtained the exponential LPP version of the family of $n^{o(1)}$ many independently sprinkled points that were alluded to in the discussion just before \eqref{eq:722}.

The strategy presented above can be viewed as a general tool which could be potentially useful in any setting where one knows how to prove an estimate that holds for the geodesic between fixed points and wants to upgrade it to one that holds uniformly for all geodesics between on-scale regions. To summarise, one first independently sprinkles a Poissonian family of pairs and then uses the result Proposition \ref{prop:45} to cover the central portion of all geodesics using geodesics only between the sprinkled points, which can be controlled by using the point-to-point estimate that one already knows how to obtain.

Finally, we note that in the above presentation, we worked with exponential LPP to make the ideas easier to follow. In the actual proofs, we shall be working with BLPP. Since the result Proposition \ref{prop:45} from \cite{BB23} is in the context of exponential LPP, we shall first need to state an analogous result for BLPP. This is stated as Proposition \ref{prop:11}, and in an appendix (Section \ref{sec:app-vol}), we shall give a short discussion of how the proof from \cite{BB23} directly adapts to yield this BLPP statement.

\subsection{The relation between the hitset and $\dim \scT^\theta$}
\label{sec:relat-betw-hits}
We now discuss the proof of Theorem \ref{thm:5}-- we shall not separately discuss Theorem \ref{thm:3} here since the core ideas in the proofs of both these results are the same. Further, we shall discuss Theorem \ref{thm:5} for the case $\theta=1$, and the same proof technique applies to other values of $\theta$ as well. By a countable union argument, it suffices to consider the set $\scT_\0^1$ defined as the set of times $t\in [0,1]$ at which there exists a bigeodesic $\Gamma^t$ additionally satisfying $\0\in \coarse(\Gamma^t)$ and show that $\dim \scT_\0^1=0$ almost surely. Further, by KPZ scaling, one might expect that for any $t\in \scT_\0^1$ and the corresponding bigeodesic $\Gamma^t$,
\begin{equation}
  \label{eq:658}
  |\Gamma^t(m)-m|=O(m^{2/3})
\end{equation}
holds for large values of $|m|$, where we recall that $\Gamma^t(m)$ is simply the largest value for which $(\Gamma^t(m),m)\in \Gamma^t$. Indeed, as is stated in Proposition \ref{prop:22}, \eqref{eq:658} can be made rigorous, and this is done in an appendix (Section \ref{sec:direction}) by uniformly controlling the transversal fluctuation of geodesics as the dynamics proceeds.

\begin{figure}
  \centering
  \includegraphics[width=0.5\linewidth]{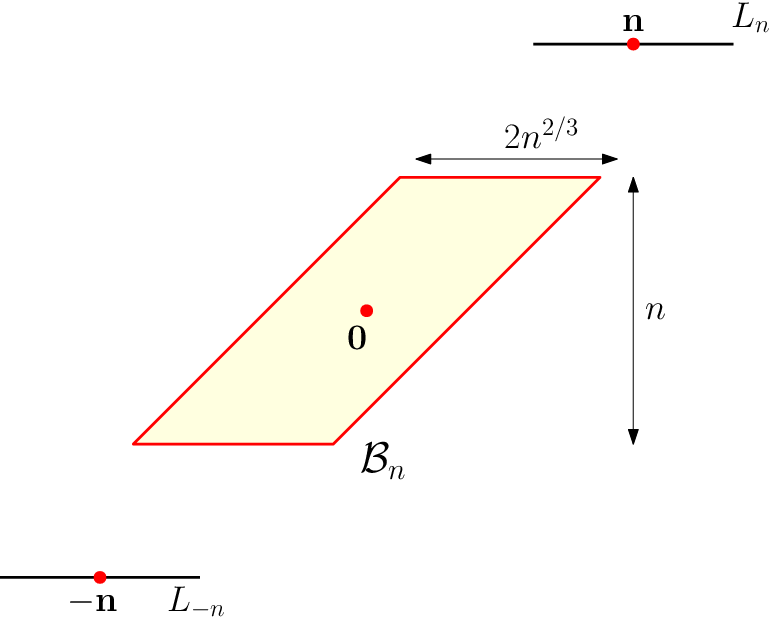}
  \caption{Here, $L_{-n}$ and $L_{n}$ are horizontal intervals of length $2n^{2/3}$ around the points $-\n$ and $\n$ respectively. Further, $\cB_n$ is an on-scale parallelogram around the point $\0$. Note that $|\cB_n|= \Theta(n^{5/3})$.}
  \label{fig:box}
\end{figure}
Now, for $n\in \ZZ$, consider the line segment $L_n$ defined by $L_n=\{n\}_{[n-|n|^{2/3},n+|n|^{2/3}]}$ (see Figure \ref{fig:box}). In view of the previous paragraph, for $s<t$, we now need to obtain an estimate on the quantity $\PP(\0\in \Gamma_u^{v,r} \textrm{ for some } r\in [s,t], u\in L_{-n}, v\in L_{n})$. It is easy to see that this is upper bounded by the quantity
\begin{equation}
  \label{eq:536}
    \PP(\0\in \coarse(\Gamma_u^{v,r}) \textrm{ for some } r\in [s,t], u\in L_{-n}, v\in L_{n}),
  \end{equation}
  and this is the quantity that we shall bound instead. Now, at an intuitive level, it is plausible that if we replace the point $\0$ in the above by any other point $p\in \ZZ^2$ which is in an on-scale $n^{2/3}\times n$ parallelogram $\cB_n$ around $\0$, then the above quantity has the same order-- we shall simply define $\cB_n= B_{n^{2/3}}(\LL_{-\n}^{\n})\cap [\![-n/2,n/2]\!]_\RR$ (see Figure \ref{fig:box}). %

In view of the above discussion, it is plausible that we would have
  \begin{align}
    \label{eq:527}
    &\PP(\0\in \coarse(\Gamma_u^{v,r}) \textrm{ for some } r\in [s,t], u\in L_{-n}, v\in L_{n})\nonumber\\
    &\sim \frac{1}{|\ZZ^2\cap\cB_n|}\sum_{p\in \ZZ^2\cap \cB_n} \PP(p\in \coarse(\Gamma_u^{v,r}) \textrm{ for some } r\in [s,t], u\in L_{-n}, v\in L_{n})\nonumber\\
    &\leq \frac{1}{|\ZZ^2\cap \cB_n|}\EE|\hitset_{L_{-n}}^{L_n,[s,t]}([\![-n/2,n/2]\!]_\RR)|\nonumber\\
    &\lesssim n^{-5/3}(n^{1+\delta}+n^{5/3+\delta}(t-s)),
  \end{align}
where to obtain the above, we have used Theorem \ref{thm:6} with $\gamma=1/2$ and have also used that $|\ZZ^2\cap \cB_n|=\Theta(n^{5/3})$. 
Applying the above with $[s,t]=[0,n^{-2/3}]$ yields $\PP(\0\in \coarse(\Gamma_u^{v,r}) \textrm{ for some } r\in [0,n^{-2/3}], u\in L_{-n}, v\in L_{n})=O(n^{-2/3+\delta})$. Thus, by covering $[0,1]$ by $n^{2/3}$ many intervals $I_i$ each of size $n^{-2/3}$, then in expectation, we would have $O(n^\delta)$ many intervals $I_i$ for which we have $\{\0\in \coarse(\Gamma_u^{v,r}) \textrm{ for some } r\in I_i, u\in L_{-n}, v\in L_{n}\}$, thereby yielding $\dim \scT_\0^1=0$ and proving Theorem \ref{thm:5}. 
\section{Geodesic switches in dynamical BLPP}
\label{sec:switches}
The goal of this section is to prove Theorem \ref{prop:30}. We shall do this by carefully tracking the contribution to the quantity $\EE[\switch_{\0}^{\n,[s,t]}([\![\beta n, (1-\beta)n]\!]_\RR)]$ originating from different scales and locations, and we now introduce some notation to make this precise. For a bounded set $A\subseteq \RR^2$, $0<\ell_1<\ell_2$, and $m\in \ZZ$, we say that the event $\loc^{\ell_1,\ell_2,m}(A)$ occurs, if we have
\begin{equation}
  \label{eq:593}
  A\subseteq [m-\ell_2,m+\ell_2]_{\RR}, |A|_{\vt}\in [\ell_1,\ell_2],
\end{equation}
where in the above, we use the notation $|A|_{\vt}$ defined at the end of Section \ref{sec:intro}. Now, for $n\in \NN$, $m\in [\![0,n]\!]$, and $1\leq \ell \leq n$, we consider the quantity (see Figure \ref{fig:loc})
\begin{figure}
  \centering
  \includegraphics[width=0.7\linewidth]{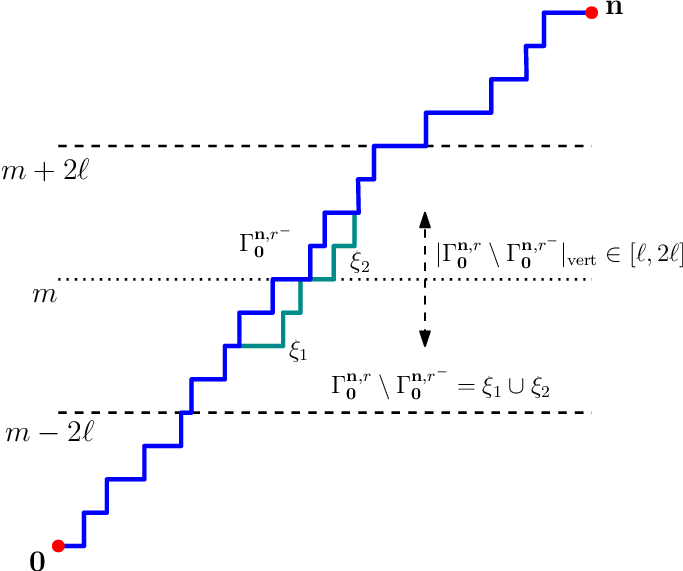}
  \caption{Here, we have a $r\in \cT_{\0}^{\n,[0,1]}$ for which the event $\loc^{\ell,2\ell,m}( \Gamma_{\0}^{\n,r}\setminus\Gamma_{\0}^{\n,r^-})$ does occur. The blue path is the geodesic $\Gamma_{\0}^{\n,r^-}$ while the set $\Gamma_{\0}^{\n,r}\setminus \Gamma_{\0}^{\n,r^-}$ is depicted in cyan. Note that here, $\overline{\Gamma_{\0}^{\n,r}\setminus \Gamma_{\0}^{\n,r^-}}$ is not an excursion about $\Gamma_{\0}^{\n,r^-}$ but is (see Lemma \ref{lem:116}) the union of two excursions $\xi_1,\xi_2$ and necessarily, at least one of the events $\loc^{\lceil\ell/2\rceil,2\ell,m}(\xi_1)$ and $\loc^{\lceil\ell/2\rceil,2\ell,m}(\xi_2)$ must occur (see Lemma \ref{lem:130}). We note that in the figure, $\ell$ is comparable to $n$, but we are also interested in the case when $\ell$ is large but much smaller than $n$.}
  \label{fig:loc}
\end{figure}
\begin{equation}
  \label{eq:142}
  \switch_{\0}^{\n,[s,t]}(\ell,m)=\sum_{r\in \cT_\0^{\n,[s,t]}}|\coarse(\Gamma_\0^{\n,r})\setminus \coarse(\Gamma_\0^{\n,r^-})|\ind(\loc^{\ell,2\ell,m}( \Gamma_{\0}^{\n,r}\setminus\Gamma_{\0}^{\n,r^-})).
\end{equation}
The entirety of this section shall be focused on proving the following estimates on the expectation of the above quantity.
\begin{proposition}
  \label{lem:49}
Fix $\beta\in (0,1/2)$ and $\delta\in (0,1/40)$. There exists a constant $C$ such that for all $m\in [\![\beta n,(1-\beta)n]\!]$, all $\ell\in [n^\delta,n]$, and all $n$ large enough, we have
  \begin{equation}
    \label{eq:147}
    \EE[\switch_{\0}^{\n,[s,t]}(\ell,m)]\leq C(t-s)\ell^{5/3}n^{2500\delta}.
  \end{equation}
  Further, for all $\ell \leq n^{\delta}$, we have
  \begin{equation}
    \label{eq:158}
    \EE[\switch_{\0}^{\n,[s,t]}(\ell,m)]\leq C(t-s)n^{2500\delta}.
  \end{equation}
\end{proposition}
In the above, it is the $\ell^{5/3}$ term that is crucial for us. The term $n^{2500\delta}$ above has not been carefully optimized and is unimportant for us as $\delta$ will be taken to be small. Now, we use Proposition \ref{lem:49} to complete the proof of Theorem \ref{prop:30} and then spend the rest of the section proving Proposition \ref{lem:49}.
\begin{proof}[Proof of Theorem \ref{prop:30} assuming Proposition \ref{lem:49}]
  It suffices to fix $\varepsilon\in(0,1]$. We write
  \begin{equation}
    \label{eq:159}
    \EE[\switch_{\0}^{\n,[s,t]}( [\![\beta n , (1-\beta) n]\!]_\RR)]\leq \sum_{i=0}^{\log_2 n} \sum_{m\in [\![\beta n,(1-\beta)n]\!]\cap 2^i\ZZ} \EE [\switch_{\0}^{\n,[s,t]}(2^i,m)].
  \end{equation}
  We split the right hand side above into two parts and bound them separately. First, by using \eqref{eq:158} in Proposition \ref{lem:49}, for some constant $C$, we have
  \begin{align}
    \label{eq:160}
    \sum_{i=0}^{\delta\log_2 n} \sum_{m\in [\![\beta n,(1-\beta) n]\!]\cap 2^i\ZZ} \EE [\switch_{\0}^{\n,[s,t]}(2^i,m)]&\leq \sum_{i=0}^{\delta\log_2 n} \sum_{m\in [\![\beta n,(1-\beta)n]\!]\cap 2^i\ZZ} C(t-s)n^{2500\delta}\nonumber\\
                                                                                                                                                        &\leq (t-s)\sum_{i=0}^{\delta \log_2 n}C2^{-i}n^{1+2500\delta}\leq 2Cn^{1+2500\delta}(t-s).
  \end{align}
  Also, by using \eqref{eq:147} in Proposition \ref{lem:49}, for some constants $C,C'$, we have
  \begin{align}
    \label{eq:161}
    \sum_{i=\delta\log_2 n}^{\log_2 n} \sum_{m\in [\![\beta n,(1-\beta)n]\!]\cap 2^i\ZZ} \EE [\switch_{\0}^{\n,[s,t]}(2^i,m)]&\leq \sum_{i=\delta\log_2 n}^{\log_2 n} \sum_{m\in [\![\beta n,(1-\beta) n]\!]\cap 2^i\ZZ} C(t-s)(2^i)^{5/3}n^{2500\delta}\nonumber\\
    &\leq  C(t-s)\sum_{i=\delta\log_2 n}^{\log_2 n}n(2^i)^{2/3}n^{2500\delta}\leq C'n^{5/3+2500\delta}(t-s).
  \end{align}
Adding up the estimates \eqref{eq:160} and \eqref{eq:161} and choosing $\delta=\varepsilon/5000$ gives a bound $Cn^{5/3+\varepsilon/2}(t-s)$, which is at most $n^{5/3+\varepsilon}(t-s)$ for all sufficiently large $n$. This completes the proof.
\end{proof}
The proof of Proposition \ref{lem:49} shall be broken down into a few steps. First, in Section \ref{sec:dyn-exc}, we shall discuss the connection between the sets $\Gamma_{\0}^{\n,r}\setminus \Gamma_{\0}^{\n,r^-}$ appearing in the definition of geodesic switches with ``excursions'' about the path $\Gamma_{\0}^{\n,r^-}$ which, in addition, are also  ``near-geodesics''. Next, by using the twin peaks estimate Proposition \ref{prop:26}, we shall obtain (Section \ref{sec:excursions}) an estimate on the probability of such excursions being present, where we shall need quantification based on the scale and location of the excursions. In Section \ref{sec:qualifying}, we shall use this along with a result (Proposition \ref{lem:31}) on the total number of peaks for routed distance profiles to obtain an estimate on the size of the union of all possible excursions discussed above. Subsequently, in Section \ref{sec:geodswitch-scale}, we shall use this estimate to complete the proof of Proposition \ref{lem:49}. %
We now begin with the first step.

\subsection{Relating the sets $\Gamma_{\0}^{\n,r}\setminus \Gamma_{\0}^{\n,r^-}$ to excursions}
\label{sec:dyn-exc}
As we shall see soon, we shall undertake a quantitative analysis of ``excursions'' about the geodesic $\Gamma_{\0}^{\n}$ which are not quite geodesics but whose weight closely rivals the passage time between their endpoints. In order for this to be useful later, we shall need to establish a connection between the objects $\Gamma_{\0}^{\n,r}\setminus \Gamma_{\0}^{\n,r^-}$ appearing in the definition of geodesic switches with excursions, and doing so is the goal of this section.

We begin by giving the precise definition of an \textbf{excursion}; recall the definition and notation corresponding to staircases from Section \ref{sec:model-definition}. For a staircase $\xi'$ and points $u\leq v\in \xi'\cap \ZZ_\RR$, a staircase $\xi$ from $u$ to $v$ is said to be an excursion about $\xi'$ if it satisfies
\begin{equation}
  \label{eq:591}
  \xi\cap \xi'= \{ u, v\},
\end{equation}
We are now ready to make a connection between $\Gamma_{\0}^{\n,r}\setminus\Gamma_{\0}^{\n,r^-}$ and excursions; note that we use $\overline{A}$ to denote the usual topological closure of a set $A\subseteq \RR^2$.
\begin{figure}
  \centering
  \begin{subfigure}{0.5\textwidth}
    \centering
    \adjustbox{valign=c}{\includegraphics[width=\linewidth]{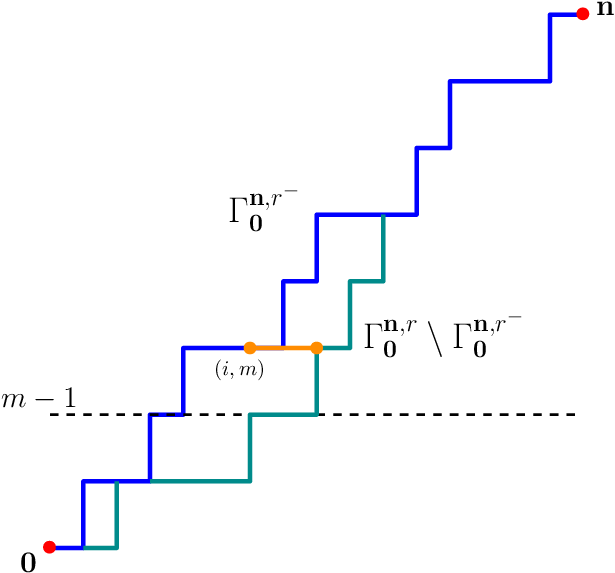}}
  \end{subfigure}\begin{subfigure}{0.5\textwidth}
    \centering
    \adjustbox{valign=c}{\includegraphics[width=\linewidth]{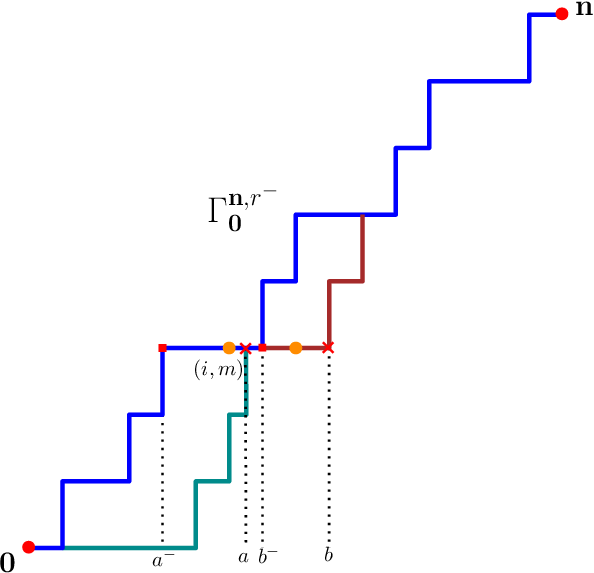}}
  \end{subfigure}
  \caption{In these figures, $(i,m)\in \cM_{\0}^{\n}$ is the unique point for which the path $W_{i,m}^r$ has been resampled. \textit{Left panel}: The displayed configuration is impossible since it leads to the geodesic $\Gamma_{\0}^{\n,r^-}$ being non-unique. %
    \textit{Right panel}: In contrast, this scenario is possible and here, $\overline{\Gamma_{\0}^{\n,r}\setminus \Gamma_{\0}^{\n,r^-}}$ consists of two excursions-- one shown in cyan and the other in brown. Note that here, $[a^-,b^-]\cap [a,b]\neq \emptyset$ (see \eqref{eq:633}).}
  \label{fig:oneortwo}
\end{figure}
\begin{lemma}
  \label{lem:116}
  Fix $[s,t]\subseteq \RR$. Almost surely, for every $r\in \cT_{\0}^{\n,[s,t]}$, exactly one of the following hold.
  \begin{enumerate}
  \item $\Gamma_{\0}^{\n,r}=\Gamma_{\0}^{\n,r^-}$.
  \item $\overline{\Gamma_{\0}^{\n,r}\setminus \Gamma_{\0}^{\n,r^-}}$ is an excursion about $\Gamma_{\0}^{\n,r^-}$. %
  \item There exist points $p_1=(x_1,t_1), p_2=(x_2,t_2), p_3 = (x_3,t_3), p_4=(x_4,t_4) \in \Gamma_{\0}^{\n,r^-}\cap \ZZ_\RR$ such that $t_1<t_2=t_3<t_4$ and $\overline{\Gamma_{\0}^{\n,r}\setminus \Gamma_{\0}^{\n,r^-}}$ is precisely the union of two excursions about $\Gamma_{\0}^{\n,r^-}$, one being a staircase from $p_1$ to $p_2$ and one being a staircase from $p_3$ to $p_4$.
  \end{enumerate}
\end{lemma}

\begin{proof}
  We begin by noting that by a standard argument (see Lemma \ref{lem:104}), almost surely, the geodesics $\Gamma_{\0}^{\n,r}$ and $\Gamma_{\0}^{\n,r^-}$ are unique for all $r\in \cT_{\0}^{\n,[s,t]}$. Also, note that almost surely, for every $r\in \cT_{\0}^{\n,[s,t]}$, there exists precisely one $(i,m)\in \cM_{\0}^{\n}$ for which $X_{i,m}^{r^-}\neq X_{i,m}^{r}$. For the remainder of the argument, we shall work on the almost sure sets where both the above hold. That is, we shall now work with a fixed $r\in \cT_{\0}^{\n,[s,t]}$ and the corresponding $(i,m)\in \cM_{\0}^{\n}$. %

  An easy but very useful observation is the following-- since $(i,m)$ is the unique element of $\cM_{\0}^{\n}$ for which $X_{i,m}^{r^-}\neq X_{i,m}^r$, for any two points $p\leq q\in [\![-\infty,m-1]\!]_\RR$, we must have $T_p^{q,r^-}=T_p^{q,r}$ and thus must also have $\Gamma_p^{q,r^-}=\Gamma_p^{q,r}$. As a result, $\Gamma_{\0}^{\n,r}\cap [0,m-1]_\RR$ is also a $T^{r^-}$ geodesic, and thus, since $\Gamma_{\0}^{\n,r^-}$ is unique, the set $S^\downarrow=(\overline{\Gamma_{\0}^{\n,r}\setminus \Gamma_{\0}^{\n,r^-}})\cap [0,m-1]_{\RR}$ is either empty or a staircase with its upper endpoint lying on $\{m-1\}_\RR$. By an analogous argument, the set $S^\uparrow=(\overline{\Gamma_{\0}^{\n,r}\setminus \Gamma_{\0}^{\n,r^-}})\cap [m+1,n]_{\RR}$ is either empty or a staircase with its lower endpoint lying on $\{m+1\}_\RR$.

  We are now ready to complete the proof. First, if both the sets $S^\downarrow$ and $S^\uparrow$ are empty, then we must necessarily have $\Gamma_{\0}^{\n,r}=\Gamma_{\0}^{\n,r^-}$ and (1) in statement of the lemma must hold.

  Now, we consider the case when exactly one of the above sets is non-empty. Without loss of generality, let us assume that $S^\downarrow \neq\emptyset$. In this case, since $S^\uparrow=\emptyset$, we must have $\Gamma_{\0}^{\n,r}(m)= \Gamma_{\0}^{\n,r^-}(m)$, and as a result, $\overline{\Gamma_{\0}^{\n,r}\setminus \Gamma_{\0}^{\n,r^-}} \subseteq [0,m]_{\RR}$ must be an excursion about $\Gamma_{\0}^{\n,r^-}$, and this is (2) in the statement of the lemma.

  Finally, we consider the case when both the sets $S^\downarrow$ and $S^\uparrow$ are non-empty. Consider the intervals $[a,b]$ and $[a^-,b^-]$ defined by $\{m\}_{[a,b]}= \Gamma_{\0}^{\n,r}\cap \{m\}_{\RR}$ and $\{m\}_{[a^-,b^-]}= \Gamma_{\0}^{\n,r^-}\cap \{m\}_{\RR}$ and note that both $[a,b]$ and $[a^-,b^-]$ must necessarily be non-empty. If we have $[a,b]\cap [a^-,b^-]=\emptyset$, then we have $\overline{\Gamma_{\0}^{\n,r}\setminus \Gamma_{\0}^{\n,r^-}}= S^\downarrow \cup (\Gamma_{\0}^{\n,r}\cap [m-1,m+1]_{\RR}) \cup S^{\uparrow}$ which is an excursion about $\Gamma_{\0}^{\n,r^-}$. Otherwise, if $[a,b]\cap [a^-,b^-]\neq \emptyset$ (see the right panel in Figure \ref{fig:oneortwo}), then we have two excursions about $\Gamma_{\0}^{\n,r^-}$, namely
  \begin{equation}
    \label{eq:633}
   S^\downarrow\cup [m-1,m]_{\{a\}}\cup\{m\}_{[a,a^-]} \textrm{ and } S^\uparrow\cup [m,m+1]_{\{b\}}\cup\{m\}_{[b^-,b]},
  \end{equation}
where we note that the intervals $[a,a^-]$, $[b^-,b]$ are possibly empty. Thus, we have (3) as in the statement of the lemma with $t_2=t_3=m$. This completes the proof.

\end{proof}
With the help of the above result, we can now obtain the following lemma.
\begin{lemma}
  \label{lem:130}
 Fix an interval $[s,t]\subseteq \RR$. For all $m\in [\![0,n]\!]$ and $1\leq \ell\leq n$, almost surely, for any $t\in \cT_{\0}^{\n,[s,t]}$ such that $\loc^{\ell,2\ell,m}( \Gamma_{\0}^{\n,r}\setminus\Gamma_{\0}^{\n,r^-})$ occurs, there must exist an excursion $\xi\subseteq \overline{\Gamma_{\0}^{\n,r}\setminus\Gamma_{\0}^{\n,r^-}}$ about $\Gamma_{\0}^{\n,r^-}$ for which $\loc^{\lceil \ell/2\rceil ,2\ell,m}(\xi)$ occurs (see Figure \ref{fig:loc}).
\end{lemma}
\begin{proof}
  If $\loc^{\ell,2\ell,m}( \Gamma_{\0}^{\n,r}\setminus\Gamma_{\0}^{\n,r^-})$ occurs, then in particular, $|\Gamma_{\0}^{\n,r}\setminus\Gamma_{\0}^{\n,r^-}|_{\vt}\geq \ell$ and thus $\Gamma_{\0}^{\n,r}\neq \Gamma_{\0}^{\n,r^-}$. Thus, by Lemma \ref{lem:116}, either (2) or (3) therein must occur. In case (2) occurs, $\overline{\Gamma_{\0}^{\n,r}\setminus \Gamma_{\0}^{\n,r^-}}$ is an excursion about $\Gamma_{\0}^{\n,r^-}$ in which case we set $\xi= \overline{\Gamma_{\0}^{\n,r}\setminus\Gamma_{\0}^{\n,r^-}}$ and this implies that the event $\loc^{\lceil \ell/2\rceil,2\ell,m}(\xi)\subseteq \loc^{\ell,2\ell,m}(\xi)$ occurs.

  If (3) from Lemma \ref{lem:116} holds instead, then we obtain two staircases  $\xi_1,\xi_2\subseteq \overline{\Gamma_{\0}^{\n,r}\setminus\Gamma_{\0}^{\n,r^-}}$, which are both excursions about the path $\Gamma_{\0}^{\n,r^-}$ and satisfy
  \begin{equation}
    \label{eq:634}
    |\xi_1|_{\vt}+|\xi_2|_{\vt}= |\Gamma_{\0}^{\n,r}\setminus\Gamma_{\0}^{\n,r^-}|_{\vt}.
  \end{equation}
  Since we are assuming that $\loc^{\ell,2\ell,m}( \Gamma_{\0}^{\n,r}\setminus\Gamma_{\0}^{\n,r^-})$ occurs, we must have $|\Gamma_{\0}^{\n,r}\setminus\Gamma_{\0}^{\n,r^-}|_{\vt}\geq \ell$ and as a result of \eqref{eq:634}, for at least one $i\in \{1,2\}$, we must have
  \begin{equation}
    \label{eq:635}
    |\xi_i|_{\vt}\geq \lceil \ell/2\rceil.
  \end{equation}
Finally, any such excursion $\xi_i$ must satisfy $\loc^{\lceil \ell/2\rceil ,2\ell,m}(\xi_i)$, and this completes the proof.  
\end{proof}

As mentioned in the beginning, we shall in fact have to focus on excursions which are also ``near-geodesics''. The primary reason behind this is the following basic lemma arguing that BLPP passage times are unlikely to change much if only one weight increment is resampled; recall that for a static BLPP $T$, we often work with the processes defined for $i,m\in \ZZ$ and $x\in [0,1]$ by $X_{i,m}(x)= W_m(x+i)-W_m(i)$.
\begin{lemma}
  \label{lem:30}
There exist constants $C,c$ such that for any $i,m\in \ZZ$, with $T$ denoting an instance of BLPP and $\widetilde{T}$ denoting the BLPP obtained by resampling just $X_{i,m}$ to a fresh sample $\widetilde{X}_{i,m}$, we have, for all $r>0$,
  \begin{equation}
    \label{eq:139}
    \PP(|\widetilde T_\0^\n-T_\0^\n| \geq r)\leq Ce^{-cr^2}.
  \end{equation}
\end{lemma}
\begin{proof}
  Using $\wgt$ and $\widetilde{\wgt}$ to denote weights of staircases for the BLPPs $T$ and $\widetilde{T}$ respectively, note that for any staircase $\xi$ from $\0$ to $\n$, we must have
  \begin{equation}
    \label{eq:661}
    |\wgt(\xi)-\widetilde{\wgt}(\xi)|\leq (\max_x X_{i,m}(x)-\min_x X_{i,m}(x))+ (\max_x \widetilde{X}_{i,m}(x)- \min_x \widetilde{X}_{i,m}(x)).
  \end{equation}
  By standard estimates, we know that if $B$ is a standard Brownian motion on $[0,1]$, then $\PP(\max_x B(x)-\min_x B(x)\geq r)\leq Ce^{-cr^2}$. We now apply this estimate to $W_{i,m}$ and $\widetilde{W}_{i,m}$ and as a result, obtain that with probability at least $1-Ce^{-cr^2}$, we have $|\wgt(\xi)-\widetilde{\wgt}(\xi)|\leq r$ for all staircases $\xi$ from $\0$ to $\n$. This completes the proof.
\end{proof}
The broad intuition now is as follows-- for any fixed $[s,t]\subseteq \RR$ and any $r\in \cT_{\0}^{\n,[s,t]}$, almost surely, only one increment changes from its previous value $X_{i,m}^{r^-}$ to $X_{i,m}^{r}$. Further, by the above lemma, the difference $T_{\0}^{\n,r}-T_{\0}^{\n,r^-}$ is very small. As a result, if the difference $\Gamma_{\0}^{\n,r}\setminus \Gamma_{\0}^{\n,r^-}$ is indeed non-empty, then its $T^{r^-}$ weight must be very close to the $T^{r^-}$ passage time between its endpoints. However, by Lemma \ref{lem:116}, the set $\Gamma_{\0}^{\n,r}\setminus \Gamma_{\0}^{\n,r^-}$ naturally consists of excursions and as a result, it is vital for us to consider excursions about $\Gamma_{\0}^{\n,r^-}$ which are also ``near-geodesics''.

\subsection{Excursions about $\Gamma_{\0}^{\n}$  which are also near-geodesics}
\label{sec:excursions}

As discussed above, we will be interested in excursions $\xi$ about $\Gamma_\0^\n$ whose weight is very close to the passage times between their endpoints. Further, we will need to carefully track their scale and location, and in view of this, for $\delta>0,1\leq \ell\leq n,m\in [\![0,n]\!]$, we define the event $\exc_{\delta}^\ell(m)$ by
\begin{equation}
    \label{eq:57.1}
    \left\{
      \exists u\leq v\in \Gamma_{\0}^{\n}\cap \ZZ_\RR, \textrm{ an excursion } \xi: u\rightarrow v \textrm{ about } \Gamma_{\0}^{\n}\big \lvert   T_u^v-\wgt(\xi)\leq \ell^{\delta}, \loc^{\lceil \ell/2\rceil ,2\ell,m}(\xi) \textrm{ occurs}\right\}.
  \end{equation}

The goal of this section is to prove the following estimate on the probability of the above event.
\begin{proposition}
  \label{lem:2}  
  Fix $\beta\in (0,1/2),\delta\in (0,1/40)$. Then there exists a constant $C$ such that for all $n$ large enough, all $\ell\in [n^{\delta},n]$ and all $m\in [\![\beta n,(1-\beta)n]\!]$, we have $\PP(\exc_\delta^{\ell}(m))\leq C\ell^{-1/3+2\delta}$.
\end{proposition}
As we described in Section \ref{sec:out-fixed}, the twin peaks estimate Proposition \ref{prop:26} shall be a core ingredient in the proof of the above. In order to be able to use the above result, we first need to argue that it is unlikely for an excursion as in $\exc_\delta^\ell(m)$ to be too ``thin'', and we now give a definition. First, we shall need some notation-- for points $(x,s)\leq (y,t)\in \ZZ_\RR$, we define
\begin{equation}
  \label{eq:586}
  Q_{(x,s)}^{(y,t)}= 2(t-s) + (y-x).
\end{equation}
We shall use the above quantity frequently and it can be thought of as the first order term in the Taylor expansion of $\EE T_{(x,s)}^{(y,t)}$ when $\slope( (x,s) , (y,t) )$ is very close to $1$.
Now, for constants $\chi\in (0,1),D>0$, and $\delta>0$, let $\thinexc_\delta^\ell$ denote the event (see Figure \ref{fig:thinthick}) that there exist points $(x,s)\leq (y,t)\in \Gamma_{\0}^{\n}\cap \ZZ_\RR$ and an excursion $\xi\colon (x,s)\rightarrow (y,t)$ about $\Gamma_\0^\n$ such that we have the following:
\begin{enumerate}
\item $t-s\in [\ell/2,2\ell]$
\item $|\xi(r)-\Gamma_{\0}^{\n}(r)|\leq \ell^{2/3-\delta}$ for more than a $(1-\chi)$ fraction of $r\in [\![s-1,t]\!]$.
\item $\wgt(\xi)-Q_{(x,s)}^{(y,t)}\geq -D\ell^{1/3+\delta}$.
\end{enumerate}
In the work \cite{GH20}, the following result on the probability of the event $\thinexc_\delta^\ell$ was obtained.
\begin{figure}
  \centering
  \includegraphics[width=0.7\linewidth]{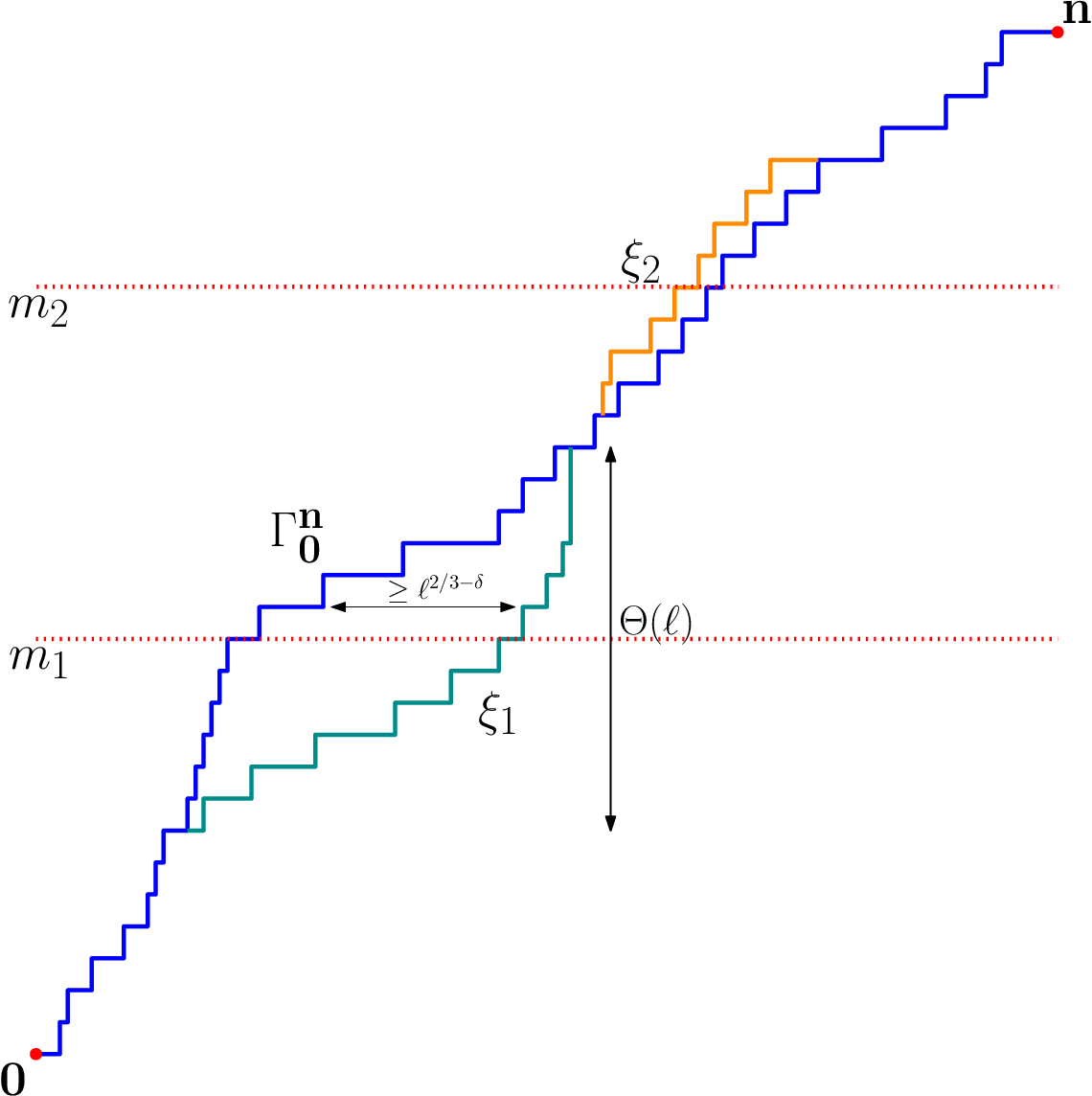}
  \caption{Here, the excursion $\xi_1$ (cyan) maintains at least an $\ell^{2/3-\delta}$ separation from $\Gamma_{\0}^{\n}$ for at least a $\chi$ fraction of its life time-- such an excursion which is also a near geodesic would be guaranteed on the event $\exc_{\delta}^{\ell}(m_1)\cap (\thinexc_\delta^\ell)^c\cap \geodwt^\ell_\delta$. In contrast, the excursion $\xi_2$ (orange) is consistently closer than $\ell^{2/3-\delta}$ to $\Gamma_{\0}^{\n}$ and is thus considered ``thin''.}
  \label{fig:thinthick}
\end{figure}
\begin{proposition}[{\cite[Theorem 1.9]{GH20}}]
  \label{prop:10}
 There exists a choice of $\chi\in (0,1),D>0$ in the definition of $\thinexc_\delta^\ell$, and there there exists a constant $d$ such that for any $\delta\in (0,1/40)$, for all $\ell \in [n^\delta,n]$ %
  and all $n$ large enough depending on $\delta$, we have $\PP(\thinexc_\delta^\ell)\leq e^{-d \ell^{\delta/2}}$. %
\end{proposition}
The above result fixes the choice of $\chi,D$ used in the definition of the event $\thinexc_\delta^\ell$ for the rest of the section. Now, intuitively, due to the above result, we know that whenever we have an excursion $\xi$ for which $\loc^{\lceil \ell/2\rceil,2\ell,m}(\xi)$ holds and which is additionally ``thin'' in the sense of satisfying (3) in the definition above, it is very likely that it at least has a $D\ell^{1/3+\delta}$ shortfall in weight compared to the mean of the passage time between its endpoints (say we call them $u\leq v$). We emphasize that the above shortfall is measured with respect to $Q_u^v$. Ideally, we would like to have a version of the above where the shortfall is measured with respect to the weight $T_u^v$, and for this, it suffices to show that $T_u^v$ cannot be much smaller than $Q_u^v$ (note that the points $u,v\in \Gamma_{\0}^{\n}$ and are thus not deterministic). The following lemma from \cite{GH20} provides the concentration estimate needed for the above.

\begin{proposition}[{\cite[Theorem 1.6]{GH20}}]
  \label{prop:15}
  There exist constants $H,h,r_0,n_0\in \NN$ such that for all $n\geq n_0$ and $\ell$ satisfying $\ell \geq h$ and $r\in \RR$ satisfying $r\geq r_0$, with probability at least $1-He^{-h^{-1}r^3\log(n/\ell)}$, the following occurs. For any points $(x,s)\leq (y,t)\in \Gamma_\0^\n\cap \ZZ_\RR$ and $\ell/2\leq t-s\leq 2\ell$ along with $r\leq \ell^{1/64}$, we have
  \begin{equation}
    \label{eq:122}
    |T_{(x,s)}^{(y,t)}-Q_{(x,s)}^{(y,t)}|\leq H^2r^2\ell^{1/3}\log^{2/3} (n/\ell).
  \end{equation}
\end{proposition}

Now, let $\geodwt_\delta^\ell$ denote the event that for any points $(x,s)\leq  (y,t)\in \Gamma_{\0}^{\n}\cap \ZZ_\RR$ with $t-s\in [\ell/2,2\ell]$, we have $T_{(x,s)}^{(y,t)}-Q_{(x,s)}^{(y,t)}\geq -(D/2)\ell^{1/3+\delta}$. Now, the above lemma immediately yields the following.
\begin{lemma}
  \label{lem:32}
  There exist constants $C,c$ such that for any fixed $\delta\in (0,1/40)$, all $\ell\in [n^\delta,n]$ and all $n$ large enough, we have $\PP(\geodwt_\delta^\ell)\geq 1-Ce^{-c\ell^{3\delta/2}}$.
\end{lemma}
\begin{proof}
  We take $r=(H^{-1}\ell^{\delta}/2)^{1/2}$ in Proposition \ref{prop:15}. Now, for all $n$ large enough, we obtain that, for some constant $c$, with probability at least $1-He^{-c\ell^{3\delta/2}\times \log(n/\ell)}$, we have
  \begin{equation}
    \label{eq:123}
    |T_{(x,s)}^{(y,t)}-Q_{(x,s)}^{(y,t)}|\leq \ell^{\delta/2}\ell^{1/3}\log^{2/3}(n/\ell)/2
  \end{equation}
  for all $(x,s),(y,t)$ as in the definition of $\geodwt_\delta^\ell$. Finally, since we are working with $\ell\geq n^{\delta}$, it is easy to see that $1\leq \log(n/\ell)\leq D\ell^{\delta/2}$ for all $n$ large enough, and this implies that with probability at least $1-He^{-c\ell^{3\delta/2}}$, we have
  \begin{equation}
    \label{eq:126}
    |T_{(x,s)}^{(y,t)}-Q_{(x,s)}^{(y,t)}|\leq (D/2)\ell^{1/3+\delta}
  \end{equation}
for all $(x,s),(y,t)$ as before. This completes the proof.
\end{proof}
Equipped with the twin peaks estimate Proposition \ref{prop:26} and the above results, we are now ready to prove Proposition \ref{lem:2}.
\begin{proof}[Proof of Proposition \ref{lem:2}]
Note that on the event
  \begin{equation}
    \label{eq:349}
    \exc_{\delta}^{\ell}(m)\cap (\thinexc_\delta^\ell)^c\cap \geodwt^\ell_\delta,
  \end{equation}
  the excursion $\xi$ from the event $\exc_\delta^\ell(m)$ can have $|\xi(s)-\Gamma_{\0}^{\n}(s)|\leq \ell^{2/3-\delta}$ for at most a $(1-\chi)$ fraction of the length of $\xi$ (see Figure \ref{fig:thinthick}). Indeed, if this were not true, then by the definition of $\thinexc_\delta^\ell(m)$, with $u,v$ denoting the endpoints of $\xi$, we would have $\wgt(\xi)-Q_u^v<-D\ell^{1/3+\delta}$. Further, by the definition of $\geodwt^\ell_{\delta}$, then we have $T_u^v-Q_u^v\geq -(D/2)\ell^{1/3+\delta}$. As a result of this, we obtain that $T_u^v-\wgt(\xi)\geq (D/2)\ell^{1/3+\delta}$ but this contradicts the definition of $\exc_\delta^\ell(m)$ which requires that $T_u^v-\wgt(\xi)\leq \ell^\delta$ which is strictly smaller than $(D/2) \ell^{1/3+\delta}$ for all $n$ large enough and $\ell\geq n^{\delta}$.

Now, we note that whenever we have an excursion $\xi\colon u\rightarrow v$ as above with $T_u^v-\wgt(\xi)\leq \ell^{\delta}$, then for all $j$ for which $\xi(j)$ is defined, we must necessarily have
  \begin{equation}
    \label{eq:350}
    T_\0^\n-Z_{\0}^{\n,\bullet}(\xi(j),j)\leq \ell^\delta,
  \end{equation}
  where we are using the routed distance profile $Z_{\0}^{\n,\bullet}$ as defined in Section \ref{s:routed}. Indeed, to see this, consider the staircase $\Gamma_{\0}^u\cup \xi \cup \Gamma_v^{\n}$ passing through $(\xi(j),j)$ and note that
  \begin{equation}
    \label{eq:594}
    \wgt(\Gamma_{\0}^u\cup \xi \cup \Gamma_v^{\n})= \wgt(\Gamma_{\0}^u\cup \Gamma_u^v \cup \Gamma_v^{\n})+ (\wgt(\xi)-T_u^v)= T_\0^\n+(\wgt(\xi)-T_u^v)\geq T_\0^\n-\ell^\delta.
  \end{equation}

A consequence of the above discussion is the following-- on the event $\exc_\delta^\ell(m)\cap
(\thinexc^\ell_\delta)^c\cap \geodwt_\delta^\ell)$, there must exist an excursion $\xi$ as in the definition of the event $\exc_\delta^\ell(m)$ with $|\xi|_{\vt}\geq \ell/2$ such that, for at least $\chi\ell/2$ many choices of $j\in [\![m-2\ell,m+2\ell]\!]$, the event
\begin{equation}
  \label{eq:471}
  \{\exists x:
|x-\Gamma_{\0}^{\n}(j)|\geq \ell^{2/3-\delta}, T_\0^\n-Z_{\0}^{\n,\bullet}(x,j)\leq
\ell^{\delta}\}
\end{equation}
holds. We are now ready to bring the twin peaks estimate Proposition \ref{prop:26} into the picture. Indeed, by the discussion above along with \eqref{eq:350} and Proposition \ref{prop:26}, for some constant $C$, we can write
  \begin{align}
    \label{eq:61}
    &\EE[(\chi \ell/4) \mathbbm{1}(\exc_\delta^\ell(m)\cap
(\thinexc^\ell_\delta)^c\cap \geodwt_\delta^\ell)]\nonumber\\ &\leq
\EE[\sum_{j=(m-2\ell)\vee (\chi \ell/8)}^{(m+2\ell)\wedge (n-\chi\ell/8)}\mathbbm{1}(\exists x:
|x-\Gamma_{\0}^{\n}(j)|\geq \ell^{2/3-\delta}, T_\0^\n-Z_{\0}^{\n,\bullet}(x,j)\leq
\ell^{\delta})]\nonumber\\ &= \sum_{j=(m-2\ell)\vee (\chi\ell /8)}^{(m+2\ell)\wedge (n-\chi \ell/8)}\PP(\exists x:
|x-\Gamma_{\0}^{\n}(j)|\geq \ell^{2/3-\delta}, T_\0^\n-Z_{\0}^{\n,\bullet}(x,j)\leq
\ell^{\delta})\nonumber\\ &\leq 4\ell \times C\ell^{-1/3+2\delta}.
  \end{align}
  Note that for applying Proposition \ref{prop:26}, it is important that for some constant $\beta'\in (0,1/2)$, we only work with $j$ satisfying $j\in [\![\beta' n, (1-\beta')n]\!]$, and this is true for the values of $j$ used above with $\beta'=(\chi \beta)/(16+\chi)$.
  Finally, on rearranging, \eqref{eq:61} immediately implies that
  \begin{equation}
    \label{eq:62}
    \PP(\exc_\delta^\ell(m)\cap (\thinexc^\ell_\delta)^c\cap \geodwt_\delta^\ell(m))\leq 16C\chi^{-1}\ell^{-1/3+2\delta}.
  \end{equation}
  Combining this with Proposition \ref{prop:10} and Lemma \ref{lem:32} now completes the proof.
\end{proof}

\subsection{Estimates on the size of the union of all qualifying excursions}
\label{sec:qualifying}

The primary difficulty with using Proposition \ref{lem:2} is the following-- on the event $\exc_\delta^\ell(m)$ which has $O(\ell^{-1/3+2\delta})$ probability, there might be a large number of possibilities of the excursion $\xi$. That is, a priori, it is possible (see Figure \ref{fig:peak}) that there is a low probability of having one qualifying excursion $\xi$, but when they exist, they are abundant and hit many distinct intervals $\{m\}_{[i,i+1]}$ for $(i,m)\in \cM_\0^\n$. The goal now is to provide estimates which rule this out.

\begin{figure}
  \centering
  \includegraphics[width=0.6\linewidth]{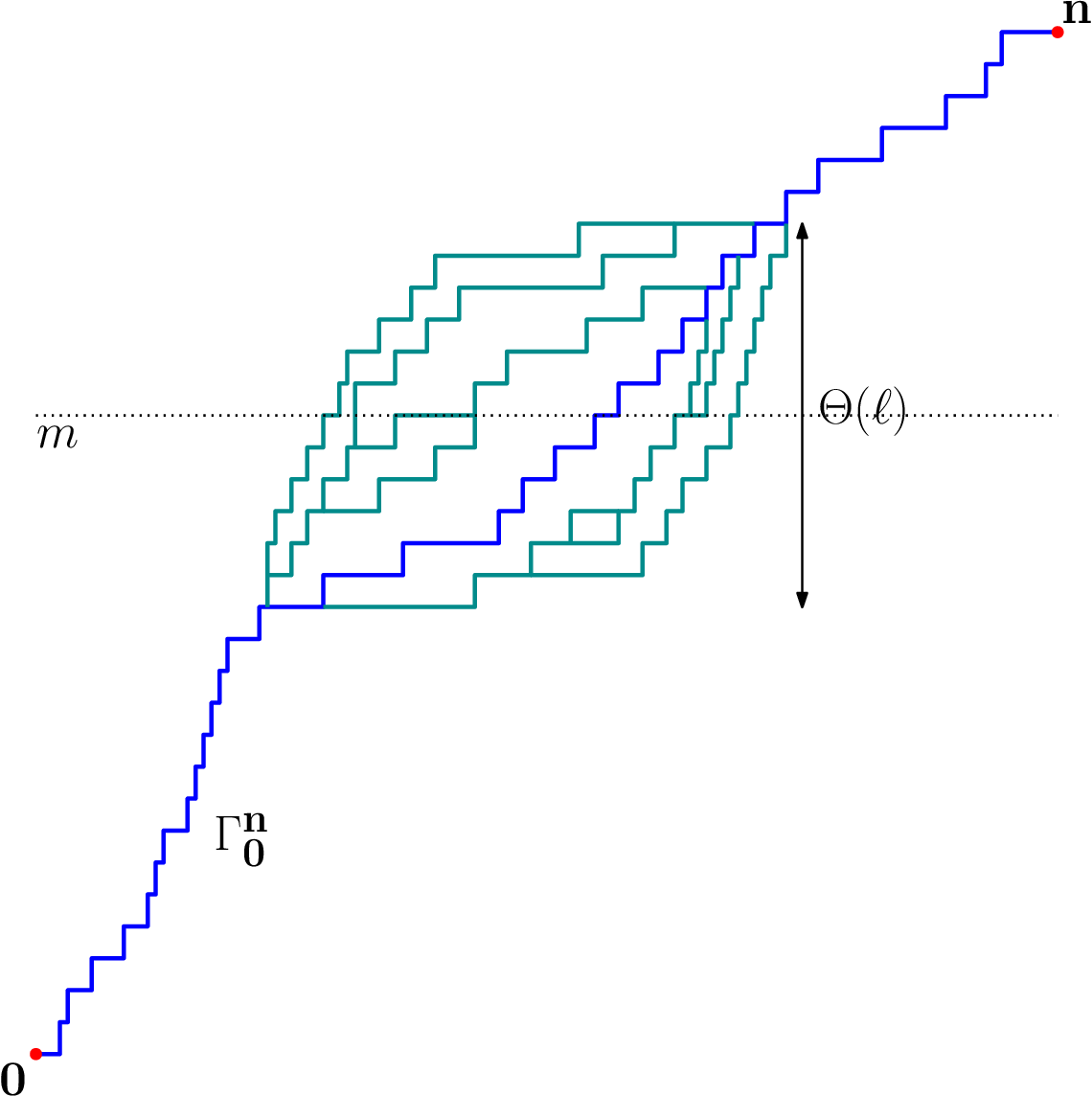}
  \caption{%
   As shown in the figure, it is, a priori, possible that on the rare event $\exc_\delta^\ell(m)$, there typically are numerous available choices (cyan) of the excursion $\xi$ appearing in the definition of $\exc_\delta^\ell(m)$, and that the union $\coarse(\bigcup \xi)$ is very large. The goal of Section \ref{sec:qualifying} is to rule out the above described pathological behaviour.}
  \label{fig:peak}
\end{figure}
Recall the sets $\peak(\alpha)\subseteq \cM_\0^\n$ defined in Section \ref{s:routed}. We now define a set $\pivot_\delta^\ell(m)\subseteq [\![m-2\ell,m+2\ell]\!]_\ZZ$ as follows,%
\begin{equation}
  \label{eq:133}
  \pivot_\delta^\ell(m)=
  \begin{cases}
    \bigcup_{j=m-2\ell}^{m+2\ell} \peak(\ell^\delta)\cap \{j\}_{\RR}, & \textrm{ if } \ell\geq n^\delta \textrm{ and } \exc_\delta^{\ell}(m) \textrm{ occurs},\\
   \bigcup_{j=m-2\ell}^{m+2\ell} \peak(n^{\delta^2})\cap \{j\}_{\RR}, &\textrm{ if } \ell<n^\delta,\\
    \emptyset, & \textrm{ otherwise}.
  \end{cases}  
\end{equation}
Intuitively, $\pivot_\delta^\ell(m)$ can be thought of a coarse grained version of the set of all qualifying excursions. %
We now have the following results bounding the cardinality of the above.
\begin{lemma}
  \label{lem:29}
Fix $\beta\in (0,1/2)$ and $\delta\in (0,1/40)$. There exist a constant $C'$ such that the following estimates hold for all $n$ large enough, $\ell\in [n^\delta,n]$, and $m\in [\![\beta n,(1-\beta)n]\!]$,
  \begin{equation}
    \label{eq:353}
    \EE[|\pivot_\delta^\ell(m)|]\leq C'\ell^{2/3}n^{1300\delta}, \EE[|\pivot_\delta^\ell(m)|^2]\leq C'\ell^{5/3}n^{2500\delta}.
  \end{equation}
\end{lemma}
\begin{proof}
  Consider the event $\cC$ defined by
  \begin{equation}
    \label{eq:361}
    \cC= \{|\peak(n^{\delta})\cap \{j\}_{\RR}|\leq n^{1200\delta} \textrm{ for all }j\in [\![m-2\ell,m+2\ell]\!]\}.
  \end{equation}
  and note that $\peak(\ell^\delta)\subseteq \peak(n^{\delta})$ as $\ell\leq n$. Then by applying Proposition \ref{lem:31}, we obtain that there exist constants $C,c$ such that we have  \begin{equation}
    \label{eq:590}
    \PP(\cC)\geq 1- Ce^{-cn^{3\delta/4}}.
  \end{equation}
  Also, recall that (see \eqref{eq:589}) deterministically, we have $|\peak(\ell^\delta)\cap \{j\}_\RR|\leq n+2$ almost surely for all $j$-- we shall use this crude bound on the event $\cC^c$.
Now, by the above discussion, for all $\ell \geq n^\delta$, we have
  \begin{align}
    \label{eq:67}
    \EE\left[\sum_{j=m-2\ell}^{m+2\ell}|\peak(\ell^{\delta})\cap \{j\}_{\RR}| \mathbbm{1}(\exc_\delta^\ell(m))\right]&\leq   \EE\left[\sum_{j=m-2\ell}^{m+2\ell}|\peak(n^{\delta})\cap \{j\}_{\RR}| \mathbbm{1}(\exc_\delta^\ell(m)\cap \cC)\right]+ 4\ell (n+2) \PP(\cC^c)\nonumber\\                                                                            &\leq 4\ell n^{1200\delta}\PP(\exc_\delta^\ell(m)) + Cn^2 e^{-cn^{3\delta/4}}\nonumber\\
                                                                                                                     &\leq C'\ell^{2/3+2\delta}n^{1200\delta}\nonumber\\
    &\leq C'\ell^{2/3}n^{1300\delta},
  \end{align}
  where the last term in the first line uses the deterministic bound $|\peak(\ell^{\delta})\cap \{j\}_{\RR}|\leq n+2$ and the second term in the second line is obtained using the bound \eqref{eq:590}. The third line uses $\PP(\exc_\delta^\ell(m))= O(\ell^{-1/3+2\delta})$ from Proposition \ref{lem:2} and that $\ell\geq n^\delta$. Finally, the last inequality holds as $\ell\leq n$. %
  This completes the proof of  the first inequality in \eqref{eq:353}.

  Similarly, to obtain the second inequality, we write
  \begin{align}
    \label{eq:362}
    &\EE\left[(\sum_{j=m-2\ell}^{m+2\ell}|\peak(\ell^{\delta})\cap \{j\}_{\RR}| \mathbbm{1}(\exc_\delta^\ell(m))^2\right]\nonumber\\
    &=  \EE\left[(\sum_{j=m-2\ell}^{m+2\ell}|\peak(n^{\delta})\cap \{j\}_{\RR}| \mathbbm{1}(\exc_\delta^\ell(m)\cap \cC))^2\right]+ (4\ell n)^2 \PP(\cC^c)\nonumber\\                                                                            &\leq (4\ell n^{1200\delta})^2\PP(\exc_\delta^\ell(m)) + 16Cn^4 e^{-cn^{3\delta/4}}\nonumber\\
    &\leq C'\ell^{5/3+2\delta}n^{2400\delta}\leq C'\ell^{5/3}n^{2500\delta}.
  \end{align}
\end{proof}
Note that the exponents $2/3$ and $5/3$ in the above result are important for us; indeed, as we shall see soon, the $5/3$ exponent in Theorem \ref{prop:30} is the same as the $5/3$ above. Also note that the $n^{1300\delta}$ and $n^{2500\delta}$ terms appearing above are not important and thus have not been carefully optimised-- indeed, these can be replaced by a more optimal $\ell^{o(1)}$ term but for convenience, we make do with the above result.
Now, in our estimates, we will also need to handle the case of small values of $\ell$, that is, $\ell\leq n^\delta$ and for this, we use the following crude result.%
\begin{lemma}
  \label{lem:93}
Fix $\beta\in (0,1/2)$ and $\delta\in (0,1/40)$. There exists a constant $C'$ such that for all $\ell\leq n^\delta$, all $m\in [\![\beta n,(1-\beta)n]\!]$ and all $n$ large enough, we have
  \begin{equation}
    \label{eq:354}
        \EE[|\pivot_\delta^\ell(m)|^2]\leq C'n^{2500\delta}.
  \end{equation}
\end{lemma}
\begin{proof}
  Recall the event $\cC$ from the proof of Lemma \ref{lem:29}. We have
  \begin{align}
    \label{eq:617}
    \EE\left[(\sum_{j=m-2\ell}^{m+2\ell}|\peak(n^{\delta^2})\cap \{j\}_{\RR}|)^2\right]%
                                                                                         &\leq (4\ell n^{1200\delta})^2\PP(\cC) + (4\ell n)^2\PP(\cC^c)\nonumber\\
                                        &\leq C'\ell^2n^{2400\delta}\nonumber\\
                                                                                                                        &\leq C' n^{2500\delta},
  \end{align}
  where in the last line, we use $\ell\leq n^{\delta}$.
\end{proof}

\subsection{The proof of Proposition \ref{lem:49}}
\label{sec:geodswitch-scale}

We are now ready to provide the proof of Proposition \ref{lem:49}. %
\begin{proof}[Proof of Proposition \ref{lem:49}]
  We first use Lemma \ref{prop:48} to transform the question to one about static BLPP as opposed to dynamical BLPP. Let $T$ be a static BLPP and let $\{W_n\}_{n\in \ZZ}$ be the associated Brownian motions. Let $(\fraki,\frakm)$ be chosen uniformly over the set $\cM_{\0}^{\n}$ independently of the BLPP $T$. We use tildes to denote quantities with respect the BLPP where just the Brownian motion $X_{\frak i,\frakm}\colon [0,1]\rightarrow \RR$ defined by $X_{\frak i,\frakm}(x)=W_{\frakm}(\fraki+x)-W_{\frakm}(\fraki)$ has been resampled to a fresh independent sample $\widetilde X_{\fraki,\frakm}$. Now, by Lemma \ref{prop:48}, we have
  \begin{align}
    \label{eq:143}
    \EE[\switch_{\0}^{\n,[s,t]}(\ell,m)]=&\EE|\cT_\0^{\n,[s,t]}| \times \EE [|\coarse(\widetilde{\Gamma}_\0^\n)\setminus \coarse(\Gamma_\0^\n)| \ind(\loc^{\ell,2\ell,m}( \widetilde{\Gamma}_{\0}^{\n}\setminus \Gamma_{\0}^{\n}))]\nonumber\\
    &=(t-s)|\cM_\0^\n|\times \EE [|\coarse(\widetilde{\Gamma}_\0^\n)\setminus \coarse(\Gamma_\0^\n)| \ind(\loc^{\ell,2\ell,m}( \widetilde{\Gamma}_{\0}^{\n}\setminus \Gamma_{\0}^{\n}))],
  \end{align}
  where to obtain the $(t-s)|\cM_\0^\n|$ term in the second line above, we have use that $|\cT_\0^{\n,[s,t]}|\sim \mathrm{Poi}((t-s)|\cM_\0^\n|)$. Let $\bigchange_\delta^\ell$ be defined by
  \begin{equation}
    \label{eq:144}
  \bigchange_\delta^\ell=\{|\widetilde{T}_{\0}^{\n}-T_{\0}^{\n}|\geq (\ell\vee n^\delta)^{\delta}\}
  \end{equation}
  and we note that, by Lemma \ref{lem:30}, for some constants $C,c$,
  \begin{equation}
    \label{eq:638}
    \PP(\bigchange_\delta^\ell)\leq Ce^{-cn^{2\delta^2}}.
  \end{equation}
  Now, we claim that on the event $(\bigchange_\delta^\ell)^c\cap \loc^{\ell,2\ell,m}( \widetilde{\Gamma}_{\0}^{\n}\setminus \Gamma_{\0}^{\n})$, we must have
  \begin{equation}
    \label{eq:145}
    \coarse(\widetilde{\Gamma}_\0^\n)\setminus \coarse(\Gamma_\0^\n)\subseteq \pivot_\delta^\ell(m).
  \end{equation}
  Indeed, to see the above, first note that on the event $(\bigchange_\delta^\ell)^c$, we must have $\coarse(\widetilde\Gamma_{\0}^{\n})\subseteq \peak((\ell \vee n^\delta)^\delta)$ and therefore, $\coarse(\widetilde{\Gamma}_\0^\n)\setminus \coarse(\Gamma_\0^\n)\subseteq \peak((\ell \vee n^\delta)^\delta)$ as well. Further, on the event $\loc^{\ell,2\ell,m}( \widetilde{\Gamma}_{\0}^{\n}\setminus \Gamma_{\0}^{\n})$, we necessarily have $\widetilde{\Gamma}_{\0}^{\n}\setminus \Gamma_{\0}^{\n}\subseteq [m-2\ell,m+2\ell]_{\RR}$. Finally, by using Lemma \ref{lem:130} along with the definition of the event $\exc_\delta^\ell(m)$, if $\ell\geq n^\delta$, we also have $\exc_\delta^\ell(m)\supseteq (\bigchange_\delta^\ell)^c\cap \loc^{\ell,2\ell,m}( \widetilde{\Gamma}_{\0}^{\n}\setminus \Gamma_{\0}^{\n})$, and this establishes \eqref{eq:145}.

  Also, as we now explain, on the event $(\bigchange_\delta^\ell)^c\cap \loc^{\ell,2\ell,m}( \widetilde{\Gamma}_{\0}^{\n}\setminus \Gamma_{\0}^{\n})$, we must have
  \begin{equation}
    \label{eq:637}
    (\fraki,\frakm)\in \pivot_\delta^\ell(m).
  \end{equation}
  Indeed, on $\loc^{\ell,2\ell,m}( \widetilde{\Gamma}_{\0}^{\n}\setminus \Gamma_{\0}^{\n})$, we must have $(\fraki,\frakm)\in \coarse(\widetilde{\Gamma}_{\0}^{\n}\cup \Gamma_{\0}^{\n})\cap [\![m-2\ell,m+2\ell]\!]_\RR$ and further, if we are additionally working on $(\bigchange_\delta^\ell)^c$, then we must have $\coarse(\widetilde{\Gamma}_{\0}^{\n}\cup \Gamma_{\0}^{\n})\subseteq \peak((\ell \vee n^\delta)^\delta)$. We emphasize that \eqref{eq:637} is crucial and will be very useful shortly.

  Finally, we note for some absolute constant $C'$, we have the easy worst case bound $|\coarse(\widetilde{\Gamma}_\0^\n)\setminus \coarse(\Gamma_\0^\n)|\leq |\coarse(\widetilde{\Gamma}_\0^\n)|\leq C'n$. Thus, we can write
  \begin{align}
    \label{eq:146}
    &\EE |\coarse(\widetilde{\Gamma}_\0^\n)\setminus \coarse(\Gamma_\0^\n)| \ind(\loc^{\ell,2\ell,m}( \widetilde{\Gamma}_{\0}^{\n}\setminus \Gamma_{\0}^{\n})) \nonumber\\
    &\leq \EE [|\coarse(\widetilde{\Gamma}_\0^\n)\setminus \coarse(\Gamma_\0^\n)| \ind(\loc^{\ell,2\ell,m}( \widetilde{\Gamma}_{\0}^{\n}\setminus \Gamma_{\0}^{\n})\cap (\bigchange_\delta^\ell)^c)]+ C'n\PP(\bigchange_\delta^\ell)\nonumber\\
    &\leq \EE[|\pivot_\delta^\ell(m)|\ind((\fraki,\frakm)\in \pivot_\delta^\ell(m))]+ Cne^{-cn^{2\delta^2}}\nonumber\\
    &=|\cM_\0^{\n}|^{-1}\EE[|\pivot_\delta^\ell(m)|^2]+Cne^{-cn^{2\delta^2}}.%
  \end{align}
  The second term in the second line has been obtained by using \eqref{eq:144} along with the worst case bound $|\coarse(\widetilde{\Gamma}_\0^\n)\setminus \coarse(\Gamma_\0^\n)|\leq C'n$ mentioned above. To obtain the third line, we have used \eqref{eq:145} along with \eqref{eq:637}. The last line follows by simply recalling that $(\fraki,\frakm)$ is chosen uniformly from the set $\cM_{\0}^{\n}$ independently of the BLPP $T$. Now, on combining \eqref{eq:146} along with \eqref{eq:143}, we obtain
  \begin{equation}
    \label{eq:639}
    \EE[\switch_{\0}^{\n,[s,t]}(\ell,m)]\leq (t-s)(\EE[|\pivot_\delta^\ell(m)|^2] + Cn^3e^{-cn^{2\delta^2}}),
  \end{equation}
  where the second term above is obtained by using that $|\cM_{\0}^{\n}|=O(n^2)$. Now, in the case when $\ell\geq n^\delta$, we simply invoke Lemma \ref{lem:29} and this yields the desired expression \eqref{eq:147}. In the case $\ell\leq n^\delta$, we invoke Lemma \ref{lem:93} and this yields \eqref{eq:158}. This completes the proof.
\end{proof}

\subsection{A version of Theorem \ref{prop:30} for general points}
\label{sec:vers}
Though we have chosen to state Theorem \ref{prop:30} for geodesic switches between the points $\0$ and $\n$, the argument directly generalises to yield a corresponding bound for the expectation of geodesic switches between any two points $p,q$ for which $\slope(p,q)$ is bounded away from $0$ and $\infty$. Indeed, we have the following result.
\begin{proposition}
  \label{prop:911}
  Fix $\beta\in (0,1/2)$, $\mu\in(0,1)$ and $\varepsilon>0$. For any $p\in \{0\}_{\RR}$ and $q\in \{n\}_{\RR}$ with $\slope(p,q)\in (\mu,\mu^{-1})$, and for all $n$ large enough and all $[s,t]\subseteq \RR$, we have
  \begin{equation}
    \label{eq:32-general}
    \EE[\switch_{p}^{q,[s,t]}([\![\beta n, (1-\beta)n]\!]_\RR)]\leq n^{5/3+\varepsilon}(t-s).
    \end{equation}
  \end{proposition}
  We shall make heavy use of the above proposition in the next section.

\section{Covering geodesics between on-scale regions by geodesics between typical points}
\label{sec:sprinkling}
While the point-to-point estimate Theorem \ref{prop:30} can directly be used to control $\EE |\hitset_{-\n}^{\n}( [\![-(1-\gamma)n, (1-\gamma)n]\!]_\RR)|$, what we actually wish to prove is Theorem \ref{thm:6} which considers the hitset between the on-scale segments $L_{-n},L_{n}$ around $-\n$ and $\n$ respectively. In order to remedy this, in this section, we shall develop a result (Proposition \ref{lem:36}) which will, with stretched exponentially high probability, allow us to simultaneously access all $\Gamma_{p}^{q,t}$ for $p\in L_{-n}, q\in L_{n}, t\in [0,1]$ by just considering geodesics between certain independently sprinkled typical points. In order to state and prove this result, we shall need a strong tail estimate on the volume of the ``basin of attraction'' around a geodesic, which we now introduce.
\subsection{A tail estimate on the basin of attraction around a geodesic}
\label{sec:proof-basin}
We now introduce the set whose volume we shall lower bound; throughout this section, we shall work with a parameter $\gamma\in (0,1)$ which we shall hold fixed. For $\delta>0$ and points $p=(x,s)\in [\![-(1+\gamma/4)n,-(1-\gamma/4)n]\!]_\RR, q=(y,t)\in [\![(1-\gamma/4)n,(1+\gamma/4)n]\!]_\RR$ and a geodesic $\Gamma_p^q$, we define $\basin^\delta_n(\Gamma_p^q)\subseteq (\ZZ_\RR)^2$ to be the set of points $p'\in [\![s,-(1-\gamma/2)n]\!]_\RR, q'\in  [\![(1-\gamma/2)n,t]\!]_\RR$ such that for some geodesic $\Gamma_{p'}^{q'}$, we have $\Gamma_{p'}^{q'}\setminus \Gamma_{p}^{q}\subseteq  (-(1-\gamma) n,(1-\gamma) n)_\RR^c$ and additionally
\begin{equation}
  \label{eq:714}
  p',q'\in B_{n^{2/3-4\delta/11}}(\Gamma_p^q).
\end{equation}
Now, note that for any measurable sets $A,B\subseteq \ZZ_\RR$, we can define $|A\times B|_{\hor}=|A|_{\hor}|B|_{\hor}$, and this naturally allows us to define $|R|_{\hor}$ for any measurable set $R\subseteq (\ZZ_{\RR})^2$. The goal now is to obtain the following result providing a lower bound on $|\basin^\delta_n(\Gamma_p^q)|_{\hor}$.
      \begin{proposition}
  \label{prop:12}
  Fix $\mu\in (0,1)$.  There exist positive constants $c,C,\delta_0$ such that for any fixed $0<\delta<\delta_0$, points $p\in [\![-(1+\gamma/4)n,-(1-\gamma/4)n]\!]_\RR, q\in [\![(1-\gamma/4)n,(1+\gamma/4)n]\!]_\RR$ with $\slope(p,q)\in (\mu,\mu^{-1})$, we have for all $n$ large enough,
  \begin{equation}
    \label{eq:71}
    \PP(|\basin^\delta_n(\Gamma_p^q)|_{\hor}\leq n^{10/3-2\delta})\leq Ce^{-cn^{3\delta/11}}.
  \end{equation}
\end{proposition}

The key ingredient in the proof of the above is a one-sided ``volume accumulation'' result proved in \cite{BB23}. While the setting in \cite{BB23} is that of semi-infinite and finite geodesics in exponential LPP (see Proposition \ref{prop:45}), the same proof technique yields an analogous result for finite geodesics in BLPP, and we shall now state this. For fixed points $p=(x,s)\in [\![-(1+\gamma/2)n,-(1-\gamma/2)n]\!]_{\RR}, q=(y,t)\in [\![(1-\gamma/2)n,(1+\gamma/2)n]\!]_\RR$, we shall use $\underline{V}_n(p,q)$ to denote the set of points $z\in [\![s, s+\gamma n/4]\!]_\RR$ which are to the ``right'' of $\Gamma_{p}^{q}$ and satisfy $\Gamma_{z}^q\setminus \Gamma_{p}^{q}\subseteq [s,-(1-\gamma) n]_\RR$ for some geodesic $\Gamma_{z}^q$. Analogously, %
we define $\overline{V}_n(p,q)$ as the set of points $z\in [\![t-\gamma n/4, t]\!]_\RR$ which are to the ``right'' of $\Gamma_p^q$ and satisfy $\Gamma_{p}^z\setminus \Gamma_{p}^{q}\subseteq [(1-\gamma)n,  t]_\RR$ for some geodesic $\Gamma_{p}^z$. We now have the following result.

\begin{proposition}[{\cite{BB23}}]
  \label{prop:11}
  Fix $\mu\in (0,1)$. There exist positive constants $c,C,K,\delta_0,\beta$ such that for any $p=(x,s) \in [\![-(1+\gamma/2)n,-(1-\gamma/2)n]\!]_\RR, q\in [\![(1-\gamma/2)n,(1+\gamma/2)n]\!]_\RR$ additionally satisfying $\slope(p,q)\in (\mu,\mu^{-1})$, and for all $\varepsilon, n$ satisfying $\varepsilon^{1/\delta_0} n\geq K$, we have
  \begin{align}
    \label{eq:69}
    &\PP(|\underline{V}_n(p,q)\cap B_{\beta\varepsilon^{5/11}n^{2/3}}(\Gamma_{p}^{q})|_{\hor}\leq \varepsilon n^{5/3})\leq Ce^{-c\varepsilon^{-3/11}},\nonumber\\
    &\PP(|\overline{V}_n(p,q)\cap B_{\beta\varepsilon^{5/11}n^{2/3}}(\Gamma_{p}^{q})|_{\hor}\leq \varepsilon n^{5/3})\leq Ce^{-c\varepsilon^{-3/11}}.
  \end{align}
\end{proposition}

We reiterate that the setting in \cite{BB23} is different-- in an appendix (Section \ref{sec:app-vol}), we discuss the slight modifications needed in the argument therein to obtain the above result. Now, we use Proposition \ref{prop:11} to prove Proposition \ref{prop:12}.

\begin{proof}[Proof of Proposition \ref{prop:12}]
The broad proof strategy is to first use Proposition \ref{prop:11} with $\varepsilon =n^{-\delta}$ for a small value of $\delta$, to obtain that $|\underline{V}_n(p,q)|_{\hor}\geq n^{5/3-\delta}$ with high probability. Thereafter, we use Proposition \ref{prop:11} along with an appropriate union bound argument to argue that for a large collection of points $u\in \underline{V}_n(p,q)$, we also have $|\overline{V}_n(u,q)|_{\hor} \geq n^{5/3-\delta}$. As we shall see, with a few extra conditions, we can ensure that such points $u$ also satisfy $\{u\}\times \overline{V}_n(u,q)\subseteq \basin^\delta_n(\Gamma_p^q)$, and this will allow us to get the desired probability bound. We now begin with the formal proof; it might be helpful for the reader to concurrently refer to Figure \ref{fig:Vset}.
  \begin{figure}
    \centering
    \includegraphics[width=0.7\linewidth]{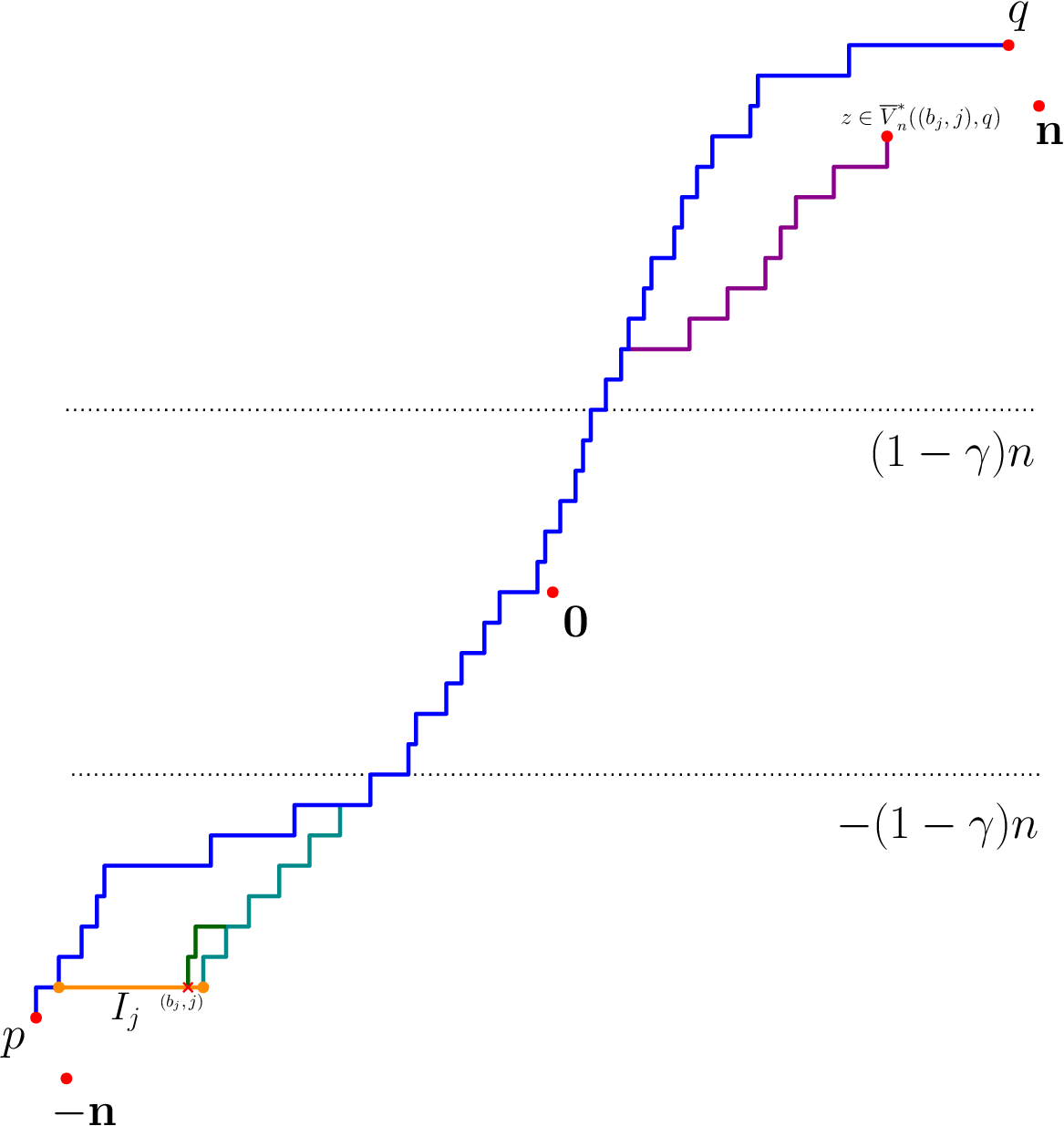}
    \caption{\textit{Proof of Proposition \ref{prop:12}}: Here, the orange interval $I_j$ is defined by $I_j=\underline{V}_n^*(p,q)\cap \{j\}_{\RR}$ and $b_j=\max\{x\in n^{-1}\ZZ:(x,j)\in I_j\}$. The crucial point is that for any $u$ lying strictly between $\Gamma_{p}^{q}(j)$ and $(b_j,j)$, by planarity, we must have $\{u\}\times \overline{V}_n^*((b_j,j),q)\subseteq \basin^\delta_n(\Gamma_p^q)$.}
    \label{fig:Vset}
  \end{figure}

  Consider the event $\hightf$ defined by the condition $\Gamma_{p}^{q}\not\subseteq B_{n^{2/3+2\delta/11}}(\LL_p^q)$. For convenience, we locally define
  \begin{equation}
    \label{eq:153}
    \underline{V}_n^*(p,q)=\underline{V}_n(p,q)\cap B_{n^{2/3-4\delta/11}}(\Gamma_p^q), \overline{V}_n^*(p,q)=\overline{V}_n(p,q)\cap B_{n^{2/3-4\delta/11}}(\Gamma_p^q).
  \end{equation}
  By a transversal fluctuation estimate (Proposition \ref{prop:38}), we have $\PP(\hightf)\leq Ce^{-cn^{3\delta/11}}$. Now, we work on the event $\cE=\hightf^c\cap \{|\underline{V}^*_n(p,q)|_{\hor}\geq n^{5/3-\delta}\}$ and apply Proposition \ref{prop:11} with $\varepsilon=n^{-\delta}$, to obtain that for any fixed $\delta< \delta_0$, for all $n$ large enough, we have
  \begin{equation}
    \label{eq:74}
    \PP(\cE)\geq 1-2Ce^{-cn^{3\delta /11}}.
  \end{equation}
  Note that on this event, every $z\in \underline{V}^*_n(p,q)$ satisfies $z\in B_{2n^{2/3+2\delta/11}}(\LL_p^q)$ as long as $n$ is large enough. Writing $p=(x_0,s_0)$, for $j\in [\![s_0,s_0+\gamma n/4]\!]$, define $I_j=\underline{V}^*_n(p,q)\cap \{j\}_{\RR}$. Now, note that due to planarity, if $(x,s)\in \underline{V}^*_n(p,q)$, then necessarily $(y,s)\in \underline{V}^*_n(p,q)$ for all $y\in [\Gamma_p^q(s),x]$. Due to this property, note that for each $j$, $I_j$ must be a (possibly degenerate) interval with its left endpoint being on the geodesic $\Gamma_p^q$. Now, by the definition of the event $\cE$, we have
  \begin{equation}
    \label{eq:81}
    \sum_{j=s_0}^{s_0+ \gamma n/4}|I_j|_{\hor}\geq n^{5/3-\delta}.
  \end{equation}
  Now, for any $j$ with $|I_j|_{\hor}\geq n^{-1}$, we define $b_j=\max\{x\in n^{-1}\ZZ:(x,j)\in I_j\}$. Let $\disvol$ be the event
  \begin{equation}
    \label{eq:82}
    \disvol=\bigcap_{z\in B_{2n^{2/3+2\delta/11}}(\LL_p^q)\cap ( [\![s_0,s_0+\gamma n/4]\!]_{n^{-1}\ZZ})}\{|\overline{V}_n^*(z,q)|_{\hor}\geq 2n^{5/3-\delta}\},
  \end{equation}
  Note that in the above, since the point $q$ is fixed and since the points $z$ are rational, the geodesic $\Gamma_z^q$ is a.s.\ unique. Further, note that since $p\in [\![-(1+\gamma/4)n, -(1-\gamma/4)n]\!]_\RR$, any point $z$ as above must satisfy $z\in [\![-(1+\gamma/2)n, -(1-\gamma/2)n]\!]_\RR$, and as a result, all the sets $\overline{V}_n^*(z,q)$ are well-defined. By using Proposition \ref{prop:11} along with a union bound, we have
  \begin{equation}
    \label{eq:3}
    \PP(\disvol)\geq 1- C_2e^{-cn^{3\delta/11}}.
  \end{equation}

  Now, on the event $\cE\cap \disvol$, and for any $j$ with $|I_j|_{\hor}\geq n^{-1}$, we claim that
  \begin{equation}
    \label{eq:83}
   \{(b_j,j)\}\times \overline{V}_n^*((b_j,j),q)\subseteq \basin^\delta_n(\Gamma_p^q).%
  \end{equation}
  Indeed, since $(b_j,j)\in \underline{V}_n^*(p,q)$, we have
  \begin{equation}
    \label{eq:640}
    \Gamma_{(b_j,j)}^q\setminus \Gamma_p^q \subseteq [s_0,-(1-\gamma)n]_\RR.
  \end{equation}
  Also, writing $q=(y_0,t_0)$, we know that for any point $u\in \overline{V}_n^*((b_j,j),q)$, we have $u\in B_{n^{2/3-4\delta/11}}(\Gamma_{(b_j,j)}^q)$ and there is a geodesic $\Gamma_{(b_j,j)}^u$ such that $\Gamma_{(b_j,j)}^u\setminus \Gamma_{(b_j,j)}^q\subseteq [(1-\gamma)n,t_0]_\RR$. As a result of this and \eqref{eq:640}, we obtain that $\Gamma_{(b_j,j)}^u\cap (-(1-\gamma) n, (1-\gamma)n)_{\RR}= \Gamma_{p}^q\cap (-(1-\gamma) n, (1-\gamma)n)_{\RR}$, thereby establishing \eqref{eq:83}. In fact, by planarity, since the points $(b_j,j)$ are to the right of $\Gamma_p^q$, the following upgraded version of \eqref{eq:83} holds-- on the event $\cE\cap \disvol$, we have
  \begin{equation}
    \label{eq:84}
    \bigcup_{j: |I_j|_{\hor}\geq n^{-1}}  \{j\}_{[\Gamma_{p}^{q}(j),b_j]}\times \overline{V}^*_n((b_j,j),q)\subseteq \basin^\delta_n(\Gamma_p^q).%
  \end{equation}
  On the event $\cE\cap \disvol$, since $|\overline{V}^*_n((b_j,j),q)|_{\hor}\geq 2n^{5/3-\delta}$ for all $b_j$ as above, it suffices to show that on this event, we also have
  \begin{equation}
    \label{eq:85}
    \sum_{j:|I_j|_{\hor}\geq n^{-1}}(b_j-\Gamma_p^q(j))>n^{5/3-\delta}/2.
  \end{equation}
  However, this is easy to obtain by using \eqref{eq:81}. Indeed, since $b_j=\max\{x\in n^{-1}\ZZ:(x,j)\in I_j\}$ and since we are working with $j\in [\![s_0,s_0+\gamma n/4]\!]$, we have
  \begin{equation}
    \label{eq:86}
    \sum_{j:|I_j|_{\hor}\geq n^{-1}}(b_j-\Gamma_p^q(j))\geq \sum_{j=s_0}^{s_0+\gamma n/4}(|I_j|_{\hor}-n^{-1})\geq \sum_{j=s_0}^{s_0+\gamma n/4}|I_j|_{\hor}-(\gamma n/4+1) n^{-1} = \sum_{j=s_0}^{s_0+\gamma n/4}|I_j|_{\hor}-\gamma/4-n^{-1},
  \end{equation}
and the proof is now completed by applying \eqref{eq:81} and noting that $n^{5/3-\delta}-\gamma/4-n^{-1}>n^{5/3-\delta}/2$ for all large enough $n$.
\end{proof}

\subsection{Capturing all geodesics via geodesics between Poissonian points}
\label{sec:capt-all-geod}
Though Theorem \ref{thm:6} is stated for the $\Theta(n^{2/3})$ length line segments $L_{-n}, L_{n}$ around $-\n,\n$ respectively, in order to prove it, we shall have to consider slightly larger regions around the above points. Indeed, with $\cK^\delta_n$ defined by
\begin{equation}
  \label{eq:726}
  \cK^\delta_n=B_{n^{2/3+\delta}}(\LL_{-\gamma \n/8}^{\gamma \n/8})\cap \ZZ_\RR,
\end{equation}
we shall often work with the regions $-\n+\cK^\delta_n, \n+\cK^\delta_n$ for a small value of $\delta$. The goal now is to argue that all geodesics $\Gamma_{p}^{q,t}$ for $p\in -\n+\cK^\delta_n, q\in \n + \cK^\delta_n$ and $t\in[0,1]$ can be captured by the corresponding geodesics $\Gamma_{\tilde p}^{\tilde q,t}$ where $\tilde p$ and $\tilde q$ are now restricted amongst a Poissonian cloud of points sprinkled independently of the dynamical BLPP. Indeed, the goal of this section is to prove the following result which formalises the above. %
  \begin{proposition}
    \label{lem:36}
    Fix $\nu>0$ and let $\cQ_{n,\nu}$ be a Poisson point process on $(\ZZ_\RR)^2$ with intensity $n^{-10/3+2\nu}$ sampled independently of the dynamical LPP $\{T^t\}_{t\in \RR}$. Let $\cover_n^{\delta,\nu}$ denote the event that for all $t\in [0,1]$, all $p\in -\n+\cK^\delta_n$, all $q\in \n+\cK^\delta_n$, and any geodesic $\Gamma_p^{q,t}$, there exist $(\tilde p, \tilde q)\in \cQ_{n,\nu}\cap \basin^\delta_{n}(\Gamma_p^{q,t})$ which additionally satisfy $\tilde p\in B_{3n^{2/3+\delta}}(\LL_{-2\n}^{2\n})\cap [\![-(1+\gamma/2)n, -(1-\gamma/2)n]\!]_{\RR}$, $\tilde q\in B_{3n^{2/3+\delta}}(\LL_{-2\n}^{2\n})\cap [\![(1-\gamma/2)n,(1+\gamma/2)n]\!]_{\RR}$.

    Then there exists a $\delta_0>0$, independent of $\nu$, such that for any fixed $0<\delta<\delta_0$ there are positive constants $C,c$ for which, for all $n$ large enough, we have%
    \begin{equation}
      \label{eq:91}
      \PP(\cover_n^{\delta,\nu})\geq 1-Ce^{-cn^{3\delta/11}}-Ce^{-cn^{2\nu-2\delta}}.
    \end{equation}
  \end{proposition}
Note that in the above, $\basin^\delta_n(\Gamma_p^{q,t})$ refers to the basin of the path $\Gamma_p^{q,t}$ with respect to the BLPP $T^t$. In reference to \eqref{eq:91}, we shall take $\delta$ and $\nu$ to be both small but such that $2\nu-2\delta>0$-- this shall ensure that both the terms above decay stretched exponentially in $n$. We now start preparing for the proof of Proposition \ref{lem:36}-- the broad reasoning is to use Proposition \ref{prop:12} to obtain that with high probability, $\basin^\delta_n(\Gamma_p^{q,t})$ cannot be too small simultaneously for all $p\in -\n+\cK^\delta_n, q\in \n+\cK^\delta_n$, $t\in[0,1]$ and then use basic properties of Poisson processes to argue that it is very likely that $\cQ_{n,\nu}$ simultaneously intersects all the above basins. In order to use the above strategy, we first need a result which allows us to simultaneously control all the above sets $\basin^\delta_n(\Gamma_p^{q,t})$ by looking at basins only corresponding to $p,q$ lying in a fine mesh of spacing $n^{-1}$. Indeed, we have the following result in the setting of static BLPP.

For the dynamical arguments in this subsection, we use the enlarged set of update blocks and its ring times
\[
 \widehat{\cM}_n=\cM_{-4\n}^{4\n},\qquad
 \widehat{\cT}_n=\cT_{-4\n}^{4\n,[0,1]}.
\]
For sufficiently small fixed $\delta$ and large $n$, the original and mesh endpoint pairs lie in $[-2n,2n]^2$, as do their geodesics by monotonicity. Their basin endpoints lie in $[-3n,3n]^2$ because the basin radius is at most $n^{2/3}$. Thus all geodesics determining these basins are controlled by $\widehat{\cM}_n$, and both the geodesics and their basins are constant between successive times of $\widehat{\cT}_n$.

  \begin{lemma}
    \label{lem:38}
    There exists a $\delta_0>0$ such that for any fixed $\delta<\delta_0$, the event $\latapx_n$ defined by
        \begin{equation}
      \label{eq:164}
      \latapx_n=\{\Gamma_p^q\subseteq B_{2n^{2/3+\delta}}(\LL_{-2\n}^{2\n}) \textrm{ for all } p\in -\n+\cK^\delta_n, q\in \n+\cK^\delta_n, \textrm{ all geodesics } \Gamma_p^q\}
    \end{equation}
    satisfies $\PP(\latapx_n)\geq 1-Ce^{-cn^{3\delta}}$ for some constants $C,c$ and all $n$. Further, for all $n$ large enough, on $\latapx_n$, for every $p\in -\n+\cK^\delta_n, q\in \n+\cK^\delta_n$ and geodesic $\Gamma_p^q$, there exists a $p'\in B_{2n^{2/3+\delta}}(\LL_{-2\n}^{2\n})\cap [\![-(1+\gamma/4) n,-(1-\gamma/4)n]\!]_{n^{-1}\ZZ}$ and a $q'\in  B_{2n^{2/3+\delta}}(\LL_{-2\n}^{2\n})\cap [\![(1-\gamma/4)n,(1+\gamma/4)n]\!]_{n^{-1}\ZZ}$ such that we have
    \begin{align}
      \label{eq:713}
      \basin^\delta_n(\Gamma_{p'}^{q'})&\subseteq \basin^\delta_n(\Gamma_p^q)\nonumber\\
                                &\subseteq (B_{3n^{2/3+\delta}}(\LL_{-2\n}^{2\n}))^2\cap ([\![-(1+\gamma/2)n, -(1-\gamma/2)n]\!]_{\RR})\times [\![(1-\gamma/2)n,(1+\gamma/2)n]\!]_{\RR}). 
    \end{align}
  \end{lemma}
  \begin{proof}
    That $\latapx_n$ satisfies the desired probability estimate is an easy consequence of Proposition \ref{prop:38} along with planarity and a union bound. Further, the second inclusion in \eqref{eq:713} automatically holds on $\latapx_n$ by the condition $\eqref{eq:714}$ in the definition of $\basin^\delta_n(\Gamma_p^q)$. Thus, to complete the proof, we need only show that the first inclusion in \eqref{eq:713} always holds on the event $\latapx_n$.
    
    To begin, we show that on the event $\latapx_n$, for any $p\in -\n+\cK_n^\delta, q\in \n+\cK_n^\delta$, we must have
    \begin{equation}
      \label{eq:165}
      \Gamma_p^q\cap [\![-(1+\gamma/4)n,-(1-\gamma/4)n]\!]_{n^{-1}\ZZ}\neq \emptyset,  \Gamma_p^q\cap [\![(1-\gamma/4)n,(1+\gamma/4)n]\!]_{n^{-1}\ZZ}\neq \emptyset.
    \end{equation}
    We just show the former, and the latter will follow similarly. For $j\in [\![-(1+\gamma/4)n, -(1-\gamma/4)n]\!]$, let $I_j=\Gamma_p^q\cap \{j\}_{\RR}$. Now, with the goal of eventually obtaining a contradiction, we assume that
    \begin{equation}
      \label{eq:662}
      \Gamma_p^q\cap [\![-(1+\gamma/4)n, -(1-\gamma/4)n]\!]_{n^{-1}\ZZ}= \emptyset.
    \end{equation}
 In particular, this implies that we have $|I_j|_{\hor}\leq n^{-1}$ for all $j\in [\![-(1+\gamma/4)n, -(1-\gamma/4)n]\!]$. Writing $p=(x,s)$, this implies that we must have
 \begin{equation}
   \label{eq:663}
   \Gamma_p^q(-(1-\gamma/4)n)-x\leq n^{-1}(-(1-\gamma/4)n-s)\leq 3\gamma/8,
 \end{equation}
 where the last line uses that since $p\in -n+\cK^\delta_n$, we have $-(1-\gamma/8)n\geq s\geq -(1+\gamma/8)n$. Now, since $p\in -\n+\cK^\delta_n$, we have $|x-s|\leq n^{2/3+\delta}$, and this implies that we must have
 \begin{equation}
   \label{eq:664}
   \Gamma_p^q(-(1-\gamma/4)n)\leq s+ n^{2/3+\delta}+3\gamma/8\leq -(1-\gamma/8)n + n^{2/3+\delta}+ 3\gamma/8.
 \end{equation}
 However, this is a contradiction since on the event $\latapx_n$, we must have
 \begin{equation}
   \label{eq:665}
   \Gamma_p^q(-(1-\gamma/4)n)\in [-(1-\gamma/4)n-2n^{2/3+\delta}, -(1-\gamma/4)n + 2n^{2/3+\delta}],
 \end{equation}
an interval which is disjoint with $(-\infty,-(1-\gamma/8)n + n^{2/3+\delta}+ 3\gamma/8]$ for all $n$ large enough as long as $\delta_0$ is chosen to be small enough.

    We have now established that \eqref{eq:165} holds on $\latapx_n$. Now, for any $p,q$, we can simply choose $p'\in \Gamma_p^q\cap [\![-(1+\gamma/4)n, -(1-\gamma/4)n]\!]_{n^{-1}\ZZ}, q'\in \Gamma_p^q\cap [\![(1-\gamma/4)n,(1+\gamma/4)n]\!]_{n^{-1}\ZZ}$. Since $p',q'$ are rational points, it is immediate that the portion of $\Gamma_p^q$ between $p',q'$ is precisely the unique geodesic $\Gamma_{p'}^{q'}$. As a result, the inclusion $\basin^\delta_n(\Gamma_{p'}^{q'})\subseteq \basin^\delta_n(\Gamma_p^q)$ holds trivially.
  \end{proof}

Now, we consider the event $\smallbasin_n$ defined as the event on which there exists a $p'\in B_{2n^{2/3+\delta}}(\LL_{-2\n}^{2\n})\cap [\![-(1+\gamma/4)n,-(1-\gamma/4)n]\!]_{n^{-1}\ZZ}$, a $q'\in B_{2n^{2/3+\delta}}(\LL_{-2\n}^{2\n})\cap [\![(1-\gamma/4)n,(1+\gamma/4)n]\!]_{n^{-1}\ZZ}$ and a $t\in \{0\}\cup \widehat{\cT}_n$ for which we have $|\basin^\delta_n(\Gamma_{p'}^{q',t})|_{\hor}\leq n^{10/3-2\delta}$. For the above event, we have the following lemma.
  \begin{lemma}
    \label{lem:102}
    There exists $\delta_0>0$ and constants $C,c$ such that for any $\delta<\delta_0$ and all $n$ large enough, we have
    \begin{equation}
     \label{eq:99}
     \PP(\smallbasin_n\lvert \widehat{\cT}_n)\leq C(|\widehat{\cT}_n|+1) e^{-cn^{3\delta/11}}.
   \end{equation}
 \end{lemma}
  \begin{proof}
    First observe that there are at most $(\gamma n/2+2)^2(4n^{2/3+\delta}\times n+2)^2$ pairs $(p',q')$ such that $p'\in B_{2n^{2/3+\delta}}(\LL_{-2\n}^{2\n})\cap [\![-(1+\gamma/4)n, -(1-\gamma/4)n]\!]_{n^{-1}\ZZ}$ and $q'\in B_{2n^{2/3+\delta}}(\LL_{-2\n}^{2\n})\cap [\![(1-\gamma/4)n,(1+\gamma/4)n]\!]_{n^{-1}\ZZ}$. Thus, %
for some constants $C,c,C_1,c_1$, we have
    \begin{align}
      \label{eq:480}
      \PP(\smallbasin_n\lvert \widehat{\cT}_n) %
      &=\PP\left(\bigcup_{t\in \{0\}\cup\widehat{\cT}_n,p',q'}\{|\basin^\delta_n(\Gamma_{p'}^{q',t})|_{\hor}\leq n^{10/3-2\delta}\}\lvert \widehat{\cT}_n\right)\nonumber\\
                                                            &\leq (|\widehat{\cT}_n|+1)\sum_{p',q'}\PP(|\basin^\delta_n(\Gamma_{p'}^{q'})|_{\hor}\leq n^{10/3-2\delta})\nonumber\\
                                                            &\leq (\gamma n/2+2)^2(4n^{2/3+\delta}\times n+2)^2(|\widehat{\cT}_n|+1)\times (C_1e^{-c_1n^{3\delta/11}})\nonumber\\
      &\leq C(|\widehat{\cT}_n|+1)e^{-cn^{3\delta/11}}.
    \end{align}
To obtain the second line, we used Lemma \ref{prop:48} and to obtain the third line, we invoked Proposition \ref{prop:12}. This completes the proof.

\end{proof}
From now on, we shall use $\latapx_n^t$ denote the occurrence of the event $\latapx_n$ from Lemma \ref{lem:38} but now for the LPP $T^t$. Using this notation, we now have the following result.
\begin{lemma}
  \label{lem:103}
  There exists $\delta_0>0$ and constants $C,c$ such that for any fixed $\delta<\delta_0$ and all $n$ large enough,
     \begin{equation}
      \label{eq:473}
       \PP(\bigcap_{t\in \{0\}\cup\widehat{\cT}_n}\latapx_{n}^t\lvert \widehat{\cT}_n)\geq 1-C(|\widehat{\cT}_n|+1)e^{-cn^{3\delta}}.
   \end{equation}
 \end{lemma}
 \begin{proof}
   By the same reasoning as in the proof of Lemma \ref{lem:102}, the term on the left hand side of \eqref{eq:473} is lower bounded by $1-(|\widehat{\cT}_n|+1)\PP((\latapx_n)^c)$. Applying Lemma \ref{lem:38} now completes the proof.
 \end{proof}
Note that Lemma \ref{lem:102} and Lemma \ref{lem:103} both involve the term $|\widehat{\cT}_n|$. The following simple tail estimate for this cardinality shall be useful for us.
  \begin{lemma}
    \label{lem:35}
There exists $\delta_0>0$ and constants $C',c'$ such that for any $\delta<\delta_0$, we have
    \begin{equation}
      \label{eq:476}
      \PP(|\widehat{\cT}_n|> C' n^{2})\leq e^{-c'n^{2}}.
    \end{equation}   
  \end{lemma}
  \begin{proof}
    First, consider the set $\widehat{\cM}_n$-- it is easy to see that we have $|\widehat{\cM}_n|=\Theta(n^{2})$. Also, by the definition of the dynamics, we know that $|\widehat{\cT}_n|$ is distributed as a Poisson variable of parameter $|\widehat{\cM}_n|$. Now, we recall the following simple bound for a Poisson variable with parameter $\lambda$-- for all $x>0$, we have
    \begin{equation}
      \label{eq:479}
      \PP(\mathrm{Poi}(\lambda)\geq \lambda+ x)\leq e^{-\frac{x^2}{2(\lambda+x)}}.
    \end{equation}
    Thus, by using the above with $x=\lambda= |\widehat{\cM}_n|$, we immediately obtain the needed result.
  \end{proof}
  In view of the above lemma, we define the event $\bddflips_n=\{|\widehat{\cT}_n|\leq  C' n^{2}\}$ and note that $\PP(\bddflips_n^c)\leq e^{-c'n^2}$. We are now ready to complete the proof of Proposition \ref{lem:36}. %

  \begin{proof}[Proof of Proposition \ref{lem:36}]

    We begin by noting that while the defining condition of the event $\cover_n^{\delta,\nu}$ includes all $t\in [0,1]$, it suffices to only prove that the condition holds for $t\in \{0\}\cup \widehat{\cT}_n$. Indeed, it is easy to see that a.s.\ for any $t\in [0,1]$, there must exist a corresponding $t'\in \{0\}\cup\widehat{\cT}_n$ such that all the geodesics from points in the set $-\n+\cK^\delta_n$ to $\n+\cK^\delta_n$, and their basins, are the same for the LPPs $T^t$ and $T^{t'}$. 
   Consider the event $\cE_n$ defined by
   \begin{equation}
     \label{eq:100}
     \cE_n=\left(\bigcap_{t\in \{0\}\cup\widehat{\cT}_n}\latapx_{n}^t\right)\cap \smallbasin_n^c\cap \bddflips_n.
   \end{equation}
   By Lemma \ref{lem:102}, Lemma \ref{lem:103} and Lemma \ref{lem:35}, it follows that for any fixed $\delta<\delta_0$ and for all $n$ large enough, we have for some constants $C,c$, 
   \begin{equation}
     \label{eq:101}
     \PP(\cE_n)\geq 1-Ce^{-cn^{3\delta/11}}.
   \end{equation}
The utility of the above event $\cE_n$ is that the following estimate holds for any fixed $p'\in B_{2n^{2/3+\delta}}(\LL_{-2\n}^{2\n})\cap [\![-(1+\gamma/4)n, -(1-\gamma/4)n]\!]_{n^{-1}\ZZ}$ and $q'\in B_{2n^{2/3+\delta}}(\LL_{-2\n}^{2\n})\cap [\![(1-\gamma/4)n,(1+\gamma/4)n]\!]_{n^{-1}\ZZ}$:
   \begin{align}
     \label{eq:102}
     &\PP\left(\bigcup_{t\in \{0\}\cup\widehat{\cT}_n}\{\cQ_{n,\nu}\cap \basin^\delta_n(\Gamma_{p'}^{q',t})=\emptyset\} \Big\lvert \cE_n\right)\nonumber\\
     &\leq \EE\left[\sum_{t\in \{0\}\cup\widehat{\cT}_n}\PP(\cQ_{n,\nu}\cap \basin^\delta_n(\Gamma_{p'}^{q',t})=\emptyset \lvert  \{T^s\}_{s\in \RR})\Big\lvert \cE_n\right]\nonumber\\
     &=\EE\left[\sum_{t\in \{0\}\cup\widehat{\cT}_n}\PP(\mathrm{Poi}(n^{-10/3+2\nu}\times |\basin^\delta_n(\Gamma_{p'}^{q',t})|_{\hor})=0)\Big\lvert \cE_n\right]\nonumber\\
     &\leq \EE[ (|\widehat{\cT}_n|+1)\PP(\mathrm{Poi}(n^{-10/3+2\nu}\times n^{10/3-2\delta})=0)\lvert \cE_n]\nonumber\\
&\leq(C'n^{2}+1)e^{-cn^{2\nu-2\delta}}.
   \end{align}
   To obtain the third line above, we have used the definition of $\cQ_{n,\nu}$. Indeed, conditional on the entire dynamics $\{T^s\}_{s\in \RR}$, for any $t\in \{0\}\cup \widehat{\cT}_n$, the cardinality of the set $\cQ_{n,\nu}\cap \basin^\delta_n(\Gamma_{p'}^{q',t})$ is simply a Poisson random variable with rate $n^{-10/3+2\nu}\times |\basin^\delta_n(\Gamma_{p'}^{q',t})|_{\hor}$-- this is because $\cQ_{n,\nu}$ is a Poisson process of rate $n^{-10/3+2\nu}$ on the space $(\ZZ_\RR)^2$ which is independent of the dynamical BLPP $\{T^s\}_{s\in \RR}$. To obtain the fourth line, we have used that on $\cE_n$, we have $|\basin^\delta_n(\Gamma_{p'}^{q',t})|_{\hor}\geq n^{10/3-2\delta}$ for all $t\in \{0\}\cup \widehat{\cT}_n$. Finally, to obtain the last line, we use that on $\cE_n$, we have $|\widehat{\cT}_n|\leq C'n^{2}$.
   
   Now, we consider the event $\cA_n$ defined by the requirement that simultaneously for all $t\in \{0\}\cup\widehat{\cT}_n$, all $p'\in B_{2n^{2/3+\delta}}(\LL_{-2\n}^{2\n})\cap [\![-(1+\gamma/4)n, -(1-\gamma/4)n]\!]_{n^{-1}\ZZ}$ and all $q'\in B_{2n^{2/3+\delta}}(\LL_{-2\n}^{2\n})\cap [\![(1-\gamma/4)n,(1+\gamma/4)n]\!]_{n^{-1}\ZZ}$, we have
   \begin{equation}
     \label{eq:96}
    \cQ_{n,\nu} \cap \basin^\delta_n(\Gamma_{p'}^{q',t})\neq \emptyset.
   \end{equation}
 By using \eqref{eq:102} and a union bound over $p',q'$, we immediately obtain
   \begin{equation}
     \label{eq:103}
     \PP(\cA_n\lvert \cE_n)\geq 1- (\gamma n/2+2)^2(4n^{2/3+\delta}\times n+2)^2\times (C'n^{2}+1)e^{-cn^{2\nu-2\delta}},
   \end{equation}
   where we note that the term $(\gamma n/2+2)^2(4n^{2/3+\delta}\times n+2)^2$ bounds the number of pairs $p'\in B_{2n^{2/3+\delta}}(\LL_{-2\n}^{2\n})\cap [\![-(1+\gamma/4)n, -(1-\gamma/4)n]\!]_{n^{-1}\ZZ}$ and $q'\in B_{2n^{2/3+\delta}}(\LL_{-2\n}^{2\n})\cap [\![(1-\gamma/4)n,(1+\gamma/4)n]\!]_{n^{-1}\ZZ}$.
   
   Finally, we note the inclusion
   \begin{equation}
     \label{eq:98}
     \cA_n\cap \cE_n\subseteq \cover_n^{\delta,\nu}.
   \end{equation}
   Indeed, since $\cE_n$ was defined to satisfy $\cE_n\subseteq \bigcap_{t\in \{0\}\cup\widehat{\cT}_n}\latapx_{n}^t$, we must have, for any $p,q,t,\Gamma_p^{q,t}$ as in the definition of the event $\cover_n^{\delta,\nu}$, a $p'\in B_{2n^{2/3+\delta}}(\LL_{-2\n}^{2\n})\cap [\![-(1+\gamma/4)n, -(1-\gamma/4)n]\!]_{n^{-1}\ZZ}$ and $q'\in B_{2n^{2/3+\delta}}(\LL_{-2\n}^{2\n})\cap [\![(1-\gamma/4)n,(1+\gamma/4)n]\!]_{n^{-1}\ZZ}$ satisfying
   \begin{align}
     \label{eq:715}
     \basin^\delta_n(\Gamma_{p'}^{q',t})&\subseteq \basin^\delta_n(\Gamma_p^{q,t})\nonumber\\
     &\subseteq (B_{3n^{2/3+\delta}}(\LL_{-2\n}^{2\n}))^2\cap ([\![-(1+\gamma/2)n, -(1-\gamma/2)n]\!]_{\RR})\times [\![(1-\gamma/2)n,(1+\gamma/2)n]\!]_{\RR}).
   \end{align}
   Further, by the definition of $\cA_n$ above, on the event $\cA_n\cap \cE_n$, for the above choice of $p',q'$, we also have $\cQ_{n,\nu}\cap \basin^\delta_n(\Gamma_{p'}^{q',t})\neq \emptyset$ and as a result, $\cQ_{n,\nu}\cap  \basin^\delta_n(\Gamma_{p}^{q,t})\neq \emptyset$. This justifies the inclusion \eqref{eq:98}. Thus, by using \eqref{eq:98}, \eqref{eq:103} and \eqref{eq:101}, we can write
   \begin{equation}
     \label{eq:104}
     \PP(\cover_n^{\delta,\nu})\geq \PP(\cA_n\cap \cE_n)= \PP(\cE_n)\PP(\cA_n\lvert \cE_n)\geq (1-Ce^{-cn^{3\delta/11}})(1-Ce^{-cn^{2\nu-2\delta}}),
   \end{equation}
   and this completes the proof.
  \end{proof}

  \section{Upper bounds on the Hausdorff dimension of exceptional times}
 \label{sec:dim-ub}
 The first goal of this section is to prove Theorem \ref{thm:6}, and then we shall subsequently use this to prove Theorem \ref{thm:3} and Theorem \ref{thm:5}.
 
\subsection{Proof of Theorem \ref{thm:6}}
\label{sec:proof-theor-refthm:6}
In fact, we shall prove the following stronger version of Theorem \ref{thm:5}, which considers the hitset corresponding to parallelograms (as opposed to segments) around $-\n$ and $\n$.%
 \begin{proposition}
    \label{prop:14}
Fix $\gamma\in (0,1)$. There exists a constant $\delta_0>0$ such that for all fixed $0<\delta<\delta_0$, and all $n$ large enough, we have
    \begin{equation}
      \label{eq:114}
      \EE\left[|\hitset_{-\n+\cK_n^\delta}^{\n+\cK_n^\delta,[s,t]}([\![-(1-\gamma)n,(1-\gamma)n]\!]_{\RR})|\right]\leq n^{1+8\delta} + n^{5/3+8\delta}(t-s).
    \end{equation}
  \end{proposition}
  In order to prove the prove the above result, we shall heavily rely on Proposition \ref{lem:36} and shall frequently use the set $\cS^\delta_n$ consisting of $(p,q)\in \cQ_{n,2\delta}$ which additionally satisfy $p\in [\![-(1+\gamma/2)n, -(1-\gamma/2)n]\!]_\RR\cap B_{3n^{2/3+\delta}}(\LL_{-2\n}^{2\n})$ and $q\in [\![(1-\gamma/2)n,(1+\gamma/2)n]\!]_\RR\cap B_{3n^{2/3+\delta}}(\LL_{-2\n}^{2\n})$. The following result is an immediate consequence of Lemma \ref{lem:38}.
  \begin{lemma}
    \label{lem:141}
   On the event $\cover_n^{\delta,2\delta}$, we have
    \begin{equation}
      \label{eq:705}
    \hitset_{-\n+\cK_n^\delta}^{\n+\cK_n^\delta,[s,t]}((-(1-\gamma)n,(1-\gamma)n)_{\RR})\subseteq  \bigcup_{(p,q)\in \cS_{n}^\delta} \hitset_{p}^{q,[s,t]}((-(1-\gamma)n,(1-\gamma)n)_{\RR}).
    \end{equation}     
  \end{lemma}
  The following lemma shall be proved by combining the above along with an easy worst case estimate.
  \begin{lemma}
    \label{lem:142}
 There exists $\delta_0>0$ and constants $C,c$ such that for any fixed $\delta<\delta_0$ and all $n$ large enough, we have
    \begin{equation}
      \label{eq:707}
      \EE\left[|\hitset_{-\n+\cK_n^\delta}^{\n+\cK_n^\delta,[s,t]}([\![-(1-\gamma)n,(1-\gamma)n]\!]_{\RR})|\right]\leq \EE[|\cS_n^\delta|]\sup_{p,q}\EE[|\hitset_{p}^{q,[s,t]}([\![-(1-\gamma)n,(1-\gamma)n]\!]_{\RR})|]+Cn+Ce^{-cn^{3\delta/11}},
    \end{equation}
    where the supremum above is over all $p\in [\![-(1+\gamma/2)n, -(1-\gamma/2)n]\!]_\RR\cap B_{3n^{2/3+\delta}}(\LL_{-2\n}^{2\n})$ and $q\in [\![(1-\gamma/2)n,(1+\gamma/2)n]\!]_\RR\cap B_{3n^{2/3+\delta}}(\LL_{-2\n}^{2\n})$.
  \end{lemma}
  \begin{proof}
    The closed slab differs from the open slab in Lemma \ref{lem:141} by at most two integer rows. For sufficiently small $\delta$ and large $n$, all staircases between the endpoint regions $-\n+\cK_n^\delta$ and $\n+\cK_n^\delta$ lie in $[-2n,2n]^2$. Hence these boundary rows contribute at most $Cn$ coarse sites, uniformly over all endpoint pairs and dynamical times. We also have the worst case estimate
          \begin{equation}
    \label{eq:484}
    |\hitset_{-\n+\cK_n^\delta}^{\n+\cK_n^\delta,[0,1]}([\![-(1-\gamma)n,(1-\gamma)n]\!]_\RR)|\leq |\cM_{-\n+\cK_n^\delta}^{\n+\cK_n^\delta}|.
  \end{equation}
Now, for some constants $C,C_1,c_1$, we can write
    \begin{align}
      \label{eq:708}
      &\EE\left[|\hitset_{-\n+\cK_n^\delta}^{\n+\cK_n^\delta,[s,t]}([\![-(1-\gamma)n,(1-\gamma)n]\!]_{\RR})|\right]\nonumber\\
      &\leq \EE[\sum_{(p,q)\in\cS_n^\delta}|\hitset_{p}^{q,[s,t]}([\![-(1-\gamma)n,(1-\gamma)n]\!]_{\RR})|]+Cn+ \EE[\ind( (\cover_n^{\delta,2\delta})^c)|\cM_{-\n+\cK_n^\delta}^{\n+\cK_n^\delta}|]\nonumber\\
      &\leq \EE[|\cS_n^\delta|]\sup_{p,q}\EE[|\hitset_{p}^{q,[s,t]}([\![-(1-\gamma)n,(1-\gamma)n]\!]_{\RR})|]+Cn+ Cn^2\PP(\cover_n^{\delta,2\delta})^c)\nonumber\\
      &\leq \EE[|\cS_n^\delta|]\sup_{p,q}\EE[|\hitset_{p}^{q,[s,t]}([\![-(1-\gamma)n,(1-\gamma)n]\!]_{\RR})|]+Cn+C_1e^{-c_1n^{3\delta/11}}.
    \end{align}
    To obtain the second line, we use Lemma \ref{lem:141}, the boundary-row bound above and \eqref{eq:484}. To obtain the first term in the third line, we use that the Poisson process $\cQ_{n}^{\delta,2\delta}$ is independent of the dynamical BLPP, and to obtain the last term therein, we simply use that there is a constant $C$ for which we have $|\cM_{-\n+\cK_n^\delta}^{\n+\cK_n^\delta}|\leq Cn^2$. Finally to obtain the last line, we used \eqref{eq:91} from Proposition \ref{lem:36}.
  \end{proof}
  In order to control the sum appearing on the right hand side of Lemma \ref{lem:142}, we shall use the following elementary but very useful fact.
  \begin{lemma}
    \label{lem:143}
    Almost surely, for all $p\leq q\in \ZZ_{\RR}$, all $s<t$ and any $K\subseteq \RR^2$, we have
    \begin{equation}
      \label{eq:709}
      |\hitset_p^{q,[s,t]}(K)|\leq |\hitset_p^{q,\{s\}}(K)|+ \switch_p^{q,[s,t]}(K).
    \end{equation}
  \end{lemma}
  \begin{proof}
    Suppose $(i,m)\in \hitset_{p}^{q,[s,t]}(K)$. The first case is that $(i,m)\in \hitset_{p}^{q\{s\}}(K)$ as well, in which case it is accounted for in the first term above. If not, then consider the time $r_*\in (s,t]$ defined by
    \begin{equation}
      \label{eq:116}
      r_*=\inf\{r: (i,m)\in \hitset_{p}^{q[s,r]}(K)\}.
    \end{equation}
    Thus, with this definition, we have
    \begin{equation}
      \label{eq:117}
     (i,m)\in \coarse(K\cap \Gamma_{p}^{q,r_*})\setminus \coarse(K\cap \Gamma_{p}^{q,r_*^-})
    \end{equation}
    and as result, it is accounted for in the second term on the right hand side of \eqref{eq:709}.
  \end{proof}
  We are now ready to complete the proof of Proposition \ref{prop:14} and thereby of Theorem \ref{thm:6} as well.
  \begin{proof}[Proof of Proposition \ref{prop:14}]
    It is easy to see that there is a deterministic constant $C'$ for which we always have $|\hitset_p^{q,\{s\}}(K)|\leq C'n$ for all $p\in [\![-(1+\gamma/2)n, -(1-\gamma/2)n]\!]_\RR\cap B_{3n^{2/3+\delta}}(\LL_{-2\n}^{2\n})$ and $q\in [\![(1-\gamma/2)n,(1+\gamma/2)n]\!]_\RR\cap B_{3n^{2/3+\delta}}(\LL_{-2\n}^{2\n})$. Further, for any $p,q$ as above, by Proposition \ref{prop:911}, we have
    \begin{equation}
      \label{eq:711}
      \EE[\switch_p^{q,[s,t]}([\![-(1-\gamma)n,(1-\gamma)n]\!]_{\RR})]\leq n^{5/3+\delta}(t-s).
    \end{equation}
    As a result of this and Lemma \ref{lem:142}, we obtain
    \begin{equation}
      \label{eq:712}
     \EE\left[|\hitset_{-\n+\cK_n^\delta}^{\n+\cK_n^\delta,[s,t]}([\![-(1-\gamma)n,(1-\gamma)n]\!]_{\RR})|\right]\leq \EE[|\cS_n^\delta|](C'n+ n^{5/3+\delta}(t-s))+Cn+Ce^{-cn^{3\delta/11}}.
   \end{equation}
   Finally, note that since $\cQ_{n,2\delta}$ is a Poisson process of rate $n^{-10/3+4\delta}$, for some constant $C$, we have $\EE|\cS_n^\delta| \leq C(\gamma n \times 6n^{2/3+\delta})^2n^{-10/3+4\delta}=36C\gamma^2n^{6\delta}$, where $C(\gamma n \times 6n^{2/3+\delta})^2$ is simply an upper bound for
   \begin{equation}
     \label{eq:716}
     |(B_{3n^{2/3+\delta}}(\LL_{-2\n}^{2\n}))^2\cap ([\![-(1+\gamma/2)n, -(1-\gamma/2)n]\!]_{\RR})\times [\![(1-\gamma/2)n,(1+\gamma/2)n]\!]_{\RR})|_{\hor}.
   \end{equation}
    This completes the proof.
  \end{proof}
  Before moving on, we note that while Proposition \ref{prop:14} was stated and proved for regions around the points $-\n,\n$, the same arguments can be used to obtain a general version corresponding to points $q,p$ such that $\slope(q,p)$ is bounded away from $0$ or $\infty$. Shortly, we shall frequently use such a result for the case when $p=-q$ and we now provide a statement without proof.
  \begin{lemma}
    \label{lem:144}
There exists $\delta_0>0$ such that the following holds. Fix $\gamma\in (0,1)$, $\mu\in (0,1)$ and $\delta<\delta_0$. Then with $\cK^\delta_p\coloneqq B_{n^{2/3+\delta}}(\LL_{-\gamma p/8}^{\gamma p/8})\cap \ZZ_\RR$, for all $p\in \{n\}_{\RR}$ satisfying $\slope(\0,p)\in (\mu,\mu^{-1})$, all $[s,t]\subseteq \RR$ and all $n$ large enough, we have
    \begin{equation}
      \label{eq:717}
      \EE[|\hitset_{-p+\cK_p^\delta}^{p+\cK_p^\delta,[s,t]}([\![-(1-\gamma)n,(1-\gamma)n]\!]_{\RR})|]\leq n^{5/3+8\delta}+ n^{1+8\delta}(t-s).
    \end{equation}
  \end{lemma}

\subsection{Proof of Theorem \ref{thm:5}}
\label{sec:proof-theor-refthm:5}

From now onwards, we shall work with $\gamma=1/2$, and as a result, we simply have $\cK_p^\delta=B_{n^{2/3+\delta}}(\LL_{-p/16}^{p/16})\cap \ZZ_\RR$. %
We shall also frequently work with the set
\begin{equation}
  \label{eq:488}
  \nbd_n= \coarse(B_{n^{2/3}}(\LL_{-p/64}^{p/64})).
\end{equation}
To simplify notation later, from now on, we write $\mathscr{I}^\delta_n=B_{n^{2/3+\delta}/2}(\{\0\})$. The reason why we define $\nbd_p$ and $\mathscr{I}^\delta_n$ as above is because of the following trivial observation.
\begin{lemma}
  \label{lem:41}
There exists $\delta_0>0$ such that for any fixed $\delta<\delta_0$, for all $n$ large enough, all $p\in \{n\}_{\RR}$ and all $q,q'\in \nbd_p$, we have $[\![-n/4,n/4]\!]_\RR\subseteq [\![-n/2,n/2]\!]_\RR + q'-q$ and $\mathscr{I}^\delta_n\subseteq \cK_p^\delta +q'-q$. %
\end{lemma}
We now make another definition that will be in play for the next few results. We let $(\fraki_p,\frakm_p)$ be uniformly chosen from $\nbd_p$ independently of the dynamical LPP. Now, with this definition, we have the following result.
\begin{lemma}
  \label{lem:40}
 There exists $\delta_0>0$ such that the following holds. For any fixed $\mu\in (0,1)$ and $\delta<\delta_0$, all $n$ large enough, and all $p\in \{n\}_{\RR}$ satisfying $\slope(\0,p)\in (\mu,\mu^{-1})$, we have, for some constant $C$,    %
  \begin{equation}
    \label{eq:121}
    \PP
    \left(
      (\fraki_p,\frakm_p)\in \hitset_{-p+\cK_p^\delta}^{p+\cK_p^\delta,[0,n^{-2/3}]}([\![-n/2,n/2]\!]_\RR)\right)\leq Cn^{-2/3+8\delta}.
  \end{equation}
\end{lemma}

\begin{proof}
  Since $(\fraki_p,\frakm_p)$ is independent of the dynamical BLPP, for some constants $C_1,C_2$ and all $n$ large enough, we have 
  \begin{align}
    \label{eq:489}
    \PP
    \left(
    (\fraki_p,\frakm_p)\in \hitset_{-p+\cK_p^\delta}^{p+\cK_p^\delta,[0,n^{-2/3}]}([\![-n/2,n/2]\!]_\RR)\right)&\leq (\nbd_p)^{-1}\EE [|\hitset_{-p+\cK_p^\delta}^{p+\cK_p^\delta,[0,n^{-2/3}]}([\![-n/2,n/2]\!]_\RR)|]\nonumber\\
                                                                                                &\leq C_1n^{-5/3}(n^{1+8\delta}+n^{5/3+8\delta}\times n^{-2/3})\nonumber\\
                                                                                                &\leq C_2n^{-2/3+8\delta},
  \end{align}
  where to obtain the second line above, we have used that $|\nbd_n|\geq Cn^{5/3}$ for some positive constant $C$; also, we have used Proposition \ref{prop:14} with $[s,t]=[0,n^{-2/3}]$. %
\end{proof}
With the help of Lemma \ref{lem:41}, we can now obtain a version of the above result where $(\fraki_p,\frakm_p)$ is replaced by a fixed point.
\begin{lemma}
  \label{lem:42}
   There exists $\delta_0>0$ such that the following holds. For any fixed $\mu\in (0,1)$ and $\delta<\delta_0$, all $n$ large enough, all $p\in \{n\}_{\RR}$ satisfying $\slope(\0,p)\in (\mu,\mu^{-1})$, and all points $q\in \nbd_p$, we have, for some constant $C$,
  \begin{equation}
    \label{eq:490}
    \PP\left(
      q\in \hitset_{-p+\mathscr{I}^\delta_n}^{p+\mathscr{I}^\delta_n,[0,n^{-2/3}]}([\![-n/4,n/4]\!]_\RR)
    \right)\leq  Cn^{-2/3+8\delta}.
  \end{equation}
\end{lemma}

\begin{proof}
By Lemma \ref{lem:41}, we almost surely have $[\![-n/4,n/4]\!]_\RR-q\subseteq [\![-n/2,n/2]\!]_\RR-(\fraki_p,\frakm_p)$, $-p+\mathscr{I}^\delta_n-q\subseteq -p+\cK_p^\delta-(\fraki_p,\frakm_p)$ and $p+\mathscr{I}^\delta_n-q\subseteq p+\cK_n^\delta-(\fraki_p,\frakm_p)$. As a result of this, for some constant $C$ and all $n$ large enough, we have
  \begin{align}
    \label{eq:124}
    &\PP\left(
      q\in \hitset_{-p+\mathscr{I}^\delta_n}^{p+\mathscr{I}^\delta_n,[0,n^{-2/3}]}([\![-n/4,n/4]\!]_\RR)\right)\nonumber\\
    &=\PP\left(
      \0\in \hitset_{-p+\mathscr{I}^\delta_n-q}^{p+\mathscr{I}^\delta_n-q,[0,n^{-2/3}]}([\![-n/4,n/4]\!]_\RR-q)\right)
    \nonumber\\
    &\leq\PP
    \left(
      \0\in \hitset_{-p+\cK_p^\delta-(\fraki_p,\frakm_p)}^{p+\cK_p^\delta-(\fraki_p,\frakm_p),[0,n^{-2/3}]}([\![-n/2,n/2]\!]_\RR-(\fraki_p,\frakm_p))
      \right)\nonumber\\
    &=\PP
    \left(
      (\fraki_p,\frakm_p)\in \hitset_{-p+\cK_p^\delta}^{p+\cK_p^\delta,[0,n^{-2/3}]}([\![-n/2,n/2]\!]_\RR
      )
      \right)\leq Cn^{-2/3+8\delta}.
  \end{align} 
The second inequality above follows by using the translational invariance of BLPP along with the fact that $(\fraki_p,\frakm_p)$ is independent of the dynamical LPP. Finally, the last inequality follows by Lemma \ref{lem:40}.
\end{proof}

For the next result, we shall use the intervals $I_{n,i}=[in^{-2/3},(i+1)n^{-2/3}]$, and for any point $p\in \{n\}_{\RR}$, we shall use $N_p$ to denote the number of $i\in [\![0,n^{2/3}-1]\!]$ such that we have
  \begin{equation}
    \label{eq:537}
    \0\in \hitset_{-p+\mathscr{I}^\delta_n}^{p+\mathscr{I}^\delta_n,I_{n,i}}([\![-n/4,n/4]\!]_\RR).
  \end{equation}
By invoking Lemma \ref{lem:42} with $q=0$ and using the stationarity of dynamical BLPP, we immediately have the following result.
\begin{lemma}
  \label{lem:114}
  There exists $\delta_0>0$ such that for any fixed $\mu\in (0,1)$ and $\delta<\delta_0$, there exists a constant $C$ such that the following holds. For any point $p\in \{n\}_\RR$ with $\slope(\0,p)\in (\mu,\mu^{-1})$, and for all large enough $n$, we have $\EE N_p\leq Cn^{8\delta}$.%
\end{lemma}

We are now ready to complete the proof of Theorem \ref{thm:5}.
\begin{proof}[Proof of Theorem \ref{thm:5}]

  First, by a simple countable union argument, it suffices to work with $\scT_\0^\theta$ defined as the set of times $t\in [0,1]$ at which there exists a bigeodesic $\Gamma^t$ additionally satisfying $\0\in \coarse(\Gamma^t)$. %
  The goal now is to establish that we a.s.\ have $\dim \scT_\0^\theta=0$ almost surely.
  
  The first ingredient in the proof is Proposition \ref{prop:22} which implies that for any fixed $\delta>0$, the following holds almost surely-- for all $t\in \scT_\0^\theta$ and the corresponding bigeodesic $\Gamma^t$, we have
  \begin{equation}
    \label{eq:491}
    \Gamma^t\cap [\![-n,n]\!]_\RR\subseteq B_{n^{2/3+\delta}/2}(\LL_{-(\theta n,n)}^{(\theta n,n)})
  \end{equation}
  for all $n$ large enough. %

  Thus, for every $t\in \scT_\0^\theta$, the event
  \begin{equation}
    \label{eq:492}
   \liminf_{n\rightarrow \infty}\left\{ \0\in \hitset_{-(\theta n,n)+ \mathscr{I}_n^\delta}^{(\theta n,n)+ \mathscr{I}_n^\delta,\{t\}}\right\}
 \end{equation}
  must hold. The goal now is to show that the set of $t$ satisfying the above must have Hausdorff dimension at most $12\delta$. Since $\delta$ is arbitrary, this would complete the proof.

  By a countable union argument, it suffices to fix an $m\in \NN$ and show that the set of $t\in [0,1]$ for which we additionally have
  \begin{equation}
    \label{eq:493}
   \bigcap_{n\geq m} \left\{ \0\in \hitset_{-(\theta n,n)+ \mathscr{I}_n^\delta}^{(\theta n,n)+ \mathscr{I}_n^\delta,\{t\}}\right\}
 \end{equation}
 almost surely has Hausdorff dimension at most $12\delta$. We locally use $\scT^{\theta,\delta}_{\0,m}$ to denote the above-mentioned set of exceptional times.

 Now, consider the intervals $I_{n,i}=[in^{-2/3},(i+1)n^{-2/3}]$ and let $\cI_n$ denote the set of $i\in [\![0,n^{2/3}-1]\!]$ such that the event
 \begin{equation}
   \label{eq:495}
   \left\{ \0\in \hitset_{-(\theta n,n)+ \mathscr{I}_n^\delta}^{(\theta n,n)+ \mathscr{I}_n^\delta,I_{n,i}}\right\}
 \end{equation}
 occurs.
 Then by the definition of $\scT^{\theta,\delta}_{\0,m}$ from \eqref{eq:493}, for all $n\geq m$, we almost surely have
 \begin{equation}
   \label{eq:494}
   \scT^{\theta,\delta}_{\0,m}\subseteq \bigcup_{i\in \cI_n} I_{n,i}. 
 \end{equation}
 The above equation yields a covering of the exceptional set $\scT^{\theta,\delta}_{\0,m}$ by intervals of length $n^{-2/3}$. It now remains to compute the expected number of intervals in such a covering, and for this, we use Lemma \ref{lem:114}. Indeed, by Lemma \ref{lem:114}, there is a constant $C$ such that for all $n$, we have
 \begin{equation}
   \label{eq:496}
   \EE |\cI_n|\leq Cn^{8\delta}=C(n^{-2/3})^{-12\delta}.
 \end{equation}
 As a result of the above estimate, we a.s.\ have $\dim \scT^{\theta,\delta}_{\0,m}\leq 12\delta$. This completes the proof. %
\end{proof}

\subsection{Proof of Theorem \ref{thm:3}}
\label{sec:proof-thm:3}
As in the proof of Theorem \ref{thm:5}, we can reduce to proving the following simpler statement.
\begin{proposition}
  \label{prop:43}
  Fix $0<\theta_1<\theta_2<\infty$ and let $\scT_\0^{(\theta_1,\theta_2)}$ denote the set of times $t\in [0,1]$ such that there exists a $\theta$-directed bigeodesic $\Gamma^t$ for the BLPP $T^t$ for some $\theta\in (\theta_1,\theta_2)$ additionally satisfying $\0\in \coarse(\Gamma^t)$. Then we almost surely have $\dim \scT_\0^{(\theta_1,\theta_2)}\leq 1/2$.
\end{proposition}

\begin{proof}[Proof of Theorem \ref{thm:3} assuming Proposition \ref{prop:43}]
  First, since the set $\coarse(\RR^2)$ is countable, it suffices to restrict to bigeodesics $\Gamma^t$ satisfying $\0\in \coarse(\Gamma^t)$. Further, since $\RR=\bigcup_{i\in \NN}[i,i+1]$ and since we are working with a stationary dynamics, it again suffices to work with $t\in [0,1]$. Also, recall that by Proposition \ref{prop:22}, almost surely, for any $t\in \RR$, any bigeodesic must be $\theta$-directed for some $\theta\in (0,\infty)$. In view of the above discussion, to complete the proof, we need only show that $\dim \scT_\0^{(0,\infty)}\leq 1/2$ almost surely, where the set $\scT_\0^{(0,\infty)}$ is defined as in the statement of Proposition \ref{prop:43}. However, we now note that the set of possible angles $(0,\infty)$ can be written as a countable union of $(\theta_1,\theta_2)$ where $\theta_1,\theta_2$ are restricted to be rational. Thus, by again using the stability of Hausdorff dimension over countable unions, the proof is complete.
\end{proof}
We now provide the proof of Proposition \ref{prop:43}.

 \begin{proof}[Proof of Proposition \ref{prop:43}]
  Fix $\delta>0$. By using the definition of $\scT_{\0}^{(\theta_1,\theta_2)}$, we know that almost surely, for every $t\in \scT_{\0}^{(\theta_1,\theta_2)}$, there exists a random $\theta\in (\theta_1,\theta_2)$ such that the geodesic $\Gamma^t$ is $\theta$-directed, satisfies $\0\in \coarse(\Gamma^t)$ and for all $n$ large enough, satisfies
   \begin{equation}
     \label{eq:540}
     \Gamma^t\cap [\![-n,n]\!]_\RR\subseteq B_{n^{2/3+\delta}/4}(\LL_{-(\theta n,n)}^{(\theta n,n)}).
   \end{equation}
   Now, for $j\in [\![2\theta_1 n^{1/3-\delta}-1,2\theta_2 n^{1/3-\delta}]\!]$, consider the overlapping intervals $J^\delta_{n,j}$ defined by $J^\delta_{n,j}= [j n^{2/3+\delta}/2,(j/2+1) n^{2/3+\delta}]$. Now, for convenience, we define $\cJ^\delta_{n,j}=\{n\}_{J^\delta_{n,j}}$.

   Since any sub-interval of length $n^{2/3+\delta}/2$ of $[\theta_1 n,\theta_2 n]$ must lie in at least one of the above intervals $J^\delta_{n,j}$, we obtain the following. Almost surely, for every $t\in \scT_\0^{(\theta_1,\theta_2)}$, the event
   \begin{equation}
     \label{eq:541}
    \liminf_{n\rightarrow \infty} \bigcup_j\{ \Gamma^t\cap (-\cJ^\delta_{n,j})\neq \emptyset, \Gamma^t\cap \cJ^\delta_{n,j}\neq \emptyset\}
   \end{equation}
   occurs and thereby, since $\0\in \coarse(\Gamma^t)$ as well, the event
   \begin{equation}
     \label{eq:542}
    \liminf_{n\rightarrow \infty} \bigcup_j \{\0\in \hitset_{-\cJ^\delta_{n,j}}^{\cJ^\delta_{n,j},\{t\}}\}
   \end{equation}
   also occurs almost surely.
   
   Now, for $m\in \NN$, we define $\scT_{\0,m}^{(\theta_1,\theta_2),\delta}\subseteq [0,1]$ as the set of times for which the event
   \begin{equation}
     \label{eq:543}
    \bigcap_{n\geq m} \bigcup_j \{\0\in \hitset_{-\cJ^\delta_{n,j}}^{\cJ^\delta_{n,j},\{t\}}\}
   \end{equation}
   occurs, where recall that $j\in [\![2\theta_1 n^{1/3-\delta}-1,2\theta_2 n^{1/3-\delta}]\!]$ in the above.
 Now, by \eqref{eq:542} and the stability of Hausdorff dimension under countable unions and the fact that $\delta$ is arbitrary, it suffices to show that for any fixed $m$, we a.s.\ have $\dim \scT_{\0,m}^{(\theta_1,\theta_2),\delta}\leq 1/2+21\delta/2$.

 Now, for $i\in [\![0, n^{2/3}-1]\!]$, we consider the intervals $I_{n,i}=[in^{-2/3},(i+1)n^{-2/3}]$ and let $\cI_n$ denote the set of $i$ such that the event $\bigcup_j \{\0\in \hitset_{-\cJ^\delta_{n,j}}^{\cJ^\delta_{n,j},I_{n,i}}\}$ occurs. By the definition of $\scT_{\0,m}^{(\theta_1,\theta_2),\delta}$, for all $n\geq m$, we have
 \begin{equation}
   \label{eq:507}
   \scT_{\0,m}^{(\theta_1,\theta_2),\delta}\subseteq \bigcup_{i\in \cI_n}I_{n,i}.
 \end{equation}
 Thus, the above equation yields a cover for the set $\scT_{\0,m}^{(\theta_1,\theta_2),\delta}$ by intervals of length $n^{-2/3}$. Further, by using Lemma \ref{lem:114} along with a union bound over $j\in [\![2\theta_1 n^{1/3-\delta}-1,2\theta_2 n^{1/3-\delta}]\!]$, we have the following bound on the size of the cover for some constant $C$ and all $n$,
 \begin{equation}
   \label{eq:508}
   \EE|\cI_n|\leq Cn^{8\delta}\times (2(\theta_2-\theta_1)n^{1/3-\delta}+1)=O(n^{1/3+7\delta})=O((n^{-2/3})^{-1/2-21\delta/2}). 
 \end{equation}
 This immediately yields that $\dim \scT_{\0,m}^{(\theta_1,\theta_2),\delta}\leq 1/2+21\delta/2$ almost surely, thereby completing the proof. %
\end{proof}

\section{Appendix 1: Directedness of infinite geodesics in dynamical BLPP}
\label{sec:direction}
The goal of this section is to discuss the proofs of Propositions \ref{prop:21}, \ref{prop:22}. For Proposition \ref{prop:21}, we shall follow the classical argument used by Newman \cite{New95} and Howard-Newman \cite{HN01} for proving the corresponding result about semi-infinite geodesics in static first passage percolation. A version of this argument was implemented for static exponential LPP in the work \cite{FP05}. In our setting, the primary difference is that we are working with a dynamical model (dynamical BLPP) instead of a static one and thus require control on geodesics simultaneously for all times. As we shall see, it turns out that the above difficulty is merely superficial, and the same argument works, albeit with minor modifications.

\subsection{Proof sketch of Proposition \ref{prop:21}}
\label{sec:proof-prop-refpr}
Before beginning, we introduce some notation. For a point $p\in \RR^2\setminus \{0\}$ and a $\theta>0$, let $\cC(p,\theta)$ denote the cone of angle $\theta$ around $p$, that is, with $\langle\cdot,\cdot\rangle$ being the usual inner product in $\RR^2$,
\begin{equation}
  \label{eq:647}
  \cC(p,\theta)= \{q\in \RR^2: \langle p,q\rangle \geq |p||q|\cos\theta\}.
\end{equation}
Now, for the proof sketch of Proposition \ref{prop:21}, it will be useful for the reader to concurrently refer to the proof of \cite[Proposition 7]{FP05}. Indeed, by the argument therein, Proposition \ref{prop:21} can be reduced to proving the following statement.
\begin{lemma}
  \label{lem:68}
  Fix $\delta>0$. Almost surely, there is a random compact set $\cK\subseteq \RR^2$ such that the following holds. For all $t\in [0,1]$ and any semi-infinite geodesic $\Gamma^t$ emanating from $\0$, for all points $p\in \Gamma^{t}\cap \cK^c$ and all points $q\in \Gamma^t$ satisfying $p\leq q$, we have $q\in \cC(p,|p|^{-1/3+\delta})$. %
\end{lemma}
Further, as discussed in the proof of \cite[Proposition 7]{FP05}, the above result in turn follows from the following transversal fluctuation estimate; note that for $A\subseteq \RR^2$, we shall use $B^{\euc}_r(A)=\{z\in \RR^2: d(z,A)\leq r\}$, where $d$ is the Euclidean metric on $\RR^2$.
\begin{lemma}
  \label{lem:65}
Fix $\delta>0$. Almost surely, there is only a bounded set of points $p\in \ZZ_\RR$ with $\0\leq p$ for which there exists a $t\in [0,1]$ and a geodesic $\Gamma_{\0}^{p,t}$ with $\Gamma_{\0}^{p,t}\not\subseteq B_{|p|^{2/3+\delta}}^{\euc}(\LL_{\0}^{p})$.
\end{lemma}

The goal now is to discuss the proof of the above statement. Since the above concerns geodesic structure in BLPP simultaneously as dynamical time is varied, we shall need a few transversal fluctuation estimates that hold uniformly as the dynamics evolves, and the following is a result in this direction.

\begin{lemma}
  \label{lem:63}
  There exist constants $C,c$ such that for all $n\in \NN$, all $\beta>0$ and $r\leq n^{1/10}$ , we have
  \begin{equation}
    \label{eq:645}
       \PP(\Gamma_{\0}^{(\beta n,n),t}\not\subseteq B_{\beta r n^{2/3}}(\LL_{\0}^{(\beta n, n)}) \textrm{ for some }t\in [0,1])\leq (1+C\beta n^2)e^{-cr^3}.
  \end{equation}
\end{lemma}
\begin{proof}
  Recall that since $\cT_{\0}^{(\beta n,n),[0,1]}\sim \textrm{Poi}(|\cM_{\0}^{(\beta n,n)}|)$, for some constant $C$, we have
  \begin{equation}
    \label{eq:691}
    \EE|\cT_{\0}^{(\beta n,n),[0,1]}|\leq C\beta n^2
  \end{equation}
 Now, by Lemma \ref{prop:48}, conditional on the set $\cT_{\0}^{(\beta n,n),[0,1]}$, $T^t$ is a BLPP for each $t\in \cT_{\0}^{(\beta n,n),[0,1]}$. As a result, we can write
  \begin{align}
    \label{eq:202}
    &\PP(\Gamma_{\0}^{(\beta n,n),t}\not\subseteq B_{\beta rn^{2/3}}(\LL_{\0}^{(\beta n,n)}) \textrm{ for some } t\in [0,1])\nonumber\\
    &=\PP(\Gamma_{\0}^{(\beta n,n),t}\not\subseteq B_{\beta rn^{2/3}}(\LL_{\0}^{(\beta n,n)}) \textrm{ for some } t\in \{0\}\cup \cT_{\0}^{(\beta n,n),[0,1]})\nonumber\\
    &\leq \EE\left[\sum_{t\in \{0\}\cup \cT_{\0}^{(\beta n,n),[0,1]}}\PP(\Gamma_{\0}^{(\beta n,n),t}\not\subseteq B_{\beta rn^{2/3}}(\LL_{\0}^{(\beta n,n)})\lvert\cT_{\0}^{(\beta n,n),[0,1]})\right]\nonumber\\
    &=(1+\EE|\cT_{\0}^{(\beta n,n),[0,1]}|)\PP(\Gamma_{\0}^{(\beta n,n)}\not\subseteq B_{\beta rn^{2/3}}(\LL_{\0}^{(\beta n,n)}))\nonumber\\
    &= (1+\EE|\cT_{\0}^{(\beta n,n),[0,1]}|)\PP(\Gamma_{\0}^{\n}\not\subseteq B_{rn^{2/3}}(\LL_{\0}^{\n}))\nonumber\\
    &\leq (1+C\beta n^2)e^{-cr^3}.
  \end{align}
The second last line above is obtained by using Brownian scaling and the last line is a consequence of \eqref{eq:691} and Proposition \ref{prop:53}. This completes the proof.
\end{proof}

We now use the above to prove Lemma \ref{lem:65}.
\begin{proof}[Proof of Lemma \ref{lem:65}]
  First, by basic trigonometry, for any $\beta>0,x>0$, we know that
  \begin{equation}
    \label{eq:646}
    B_{\beta x}(\LL_{\0}^{(\beta n,n)})=B_{\beta x/\sqrt{1+\beta^2}}^{\euc}(\LL_{\0}^{(\beta n, n)})\subseteq B_{x}^{\euc}(\LL_{\0}^{(\beta n, n)}).
  \end{equation}
 As a result, by using Lemma \ref{lem:63}, we know that for some constants $C,c>0$, and for all $\beta>0$ and $r\leq n^{1/10}$, we have
  \begin{equation}
    \label{eq:187.1}
    \PP(\Gamma_{\0}^{(\beta n,n),t}\subseteq B^{\euc}_{r n^{2/3}}(\LL_{\0}^{ (\beta n, n )}) \textrm{ for all } t\in [0,1])\geq 1-(1+C\beta n^2)e^{-cr^3}.
  \end{equation}
  Note that deterministically, we have the inequality $\Gamma_{\0}^{(\beta n,n),t}\subseteq B^{\euc}_{n}(\LL_{\0}^{ (\beta n, n )})$ for all $t\in \RR$. As a result, for some constant $\nu>0$, all $\beta>0$ and all $r>0$, we have
  \begin{equation}
    \label{eq:649}
    \PP(\Gamma_{\0}^{(\beta n,n),t}\subseteq B^{\euc}_{r n^{2/3}}(\LL_{\0}^{ (\beta n, n )}) \textrm{ for all } t\in [0,1])\geq 1-(1+C\beta n^2)e^{-cr^\nu}.
  \end{equation}
We now choose $r$ such that $rn^{2/3}=|(\beta n,n)|^{2/3+\delta}/2$ and as a result, we obtain that for some constant $\nu',C',c'>0$ and all $\0\leq p\in \ZZ_\RR$,
  \begin{equation}
    \label{eq:187}
    \PP(\Gamma_{\0}^{p,t}\subseteq B^{\euc}_{|p|^{2/3+\delta}/2}(\LL_{\0}^{ p}) \textrm{ for all } t\in [0,1])\geq 1-C'e^{-c'|p|^{\nu'}}.
  \end{equation}
  Thus, by the Borel-Cantelli lemma, we obtain that almost surely, for all except a finite set of $p\in \NN^2$, we have
  \begin{equation}
    \label{eq:690}
    \Gamma_{\0}^{p,t} \subseteq B^{\euc}_{|p|^{2/3+\delta}/2}(\LL_{\0}^{p})
  \end{equation}
  for all $t\in [0,1]$. By geodesic ordering, this in fact implies that almost surely, for all except a bounded set of $p$, we have $\Gamma_{\0}^{p,t} \subseteq B^{\euc}_{|p|^{2/3+\delta}}(\LL_{\0}^{p})$ for all $t\in [0,1]$ and geodesics $\Gamma_{\0}^{p,t}$. This completes the proof.
\end{proof}

Before moving on, we record a consequence of Lemma \ref{lem:68} that will shortly be very useful.
\begin{lemma}
  \label{lem:69}
  Fix $\delta>0$. The following holds almost surely. For all $t\in \RR$ and $\theta\in [0,\infty]$ admitting a $\theta$-directed semi-infinite geodesic $\Gamma^{t}$, and for all $m$ large enough, we have %
  \begin{equation}
    \label{eq:212}
    \Gamma^t\cap \{m\}_{\RR}\subseteq [\theta m-m^{2/3+\delta}, \theta m+m^{2/3+\delta}].
  \end{equation}
\end{lemma}
\begin{proof}
  First, note that it suffices to work only with non-trivial geodesics $\Gamma^t$ as the estimate is obvious in the trivial case. Now, since the above geodesic is assumed to be non-trivial, it is easy to see that it must pass via a rational point in $\ZZ_\RR$. As a result, by translation invariance, it suffices to prove the lemma where we additionally assume that $t\in [0,1]$ and that $\Gamma^t$ is a geodesic emanating from $\0$. Now, with the aim of eventually obtaining a contradiction, we assume that there exists a $t\in [0,1]$ and an infinite sequence of points $p_i=(x_i,m_i)\in \Gamma^{t}$ such that $m_i\rightarrow \infty$ and
  \begin{equation}
    \label{eq:650}
    x_i\notin [\theta m_i-m_i^{2/3+\delta}, \theta m_i + m_i^{2/3+\delta}]
  \end{equation}
  for all $i$. As a consequence there must exist a positive constant $C$ such that for all $i$, we have
  \begin{equation}
    \label{eq:727}
    \cC( (\theta m_i, m_i), Cm_i^{-1/3+\delta})\cap \cC(p_i, C m_i^{-1/3+\delta})=\emptyset.
  \end{equation}%
 However, by using Lemma \ref{lem:68} with $\delta$ replaced by $\delta/2$, we know that there is a positive constant $C$ such that for all $i$ large enough, and all $r\geq m_i$, we have
    \begin{equation}
    \label{eq:213}
    (\Gamma^{t}(r),r)\in \cC(p_i, C m_i^{-1/3+\delta}). 
  \end{equation}
Finally, since $\Gamma^{t}$ is $\theta$-directed, we must also have $ (\Gamma^t(r),r)\in \cC( (\theta m_i, m_i), Cm_i^{-1/3+\delta})$ for any fixed $i$ and all $r$ large enough. This contradicts \eqref{eq:727} and \eqref{eq:213}, thereby completing the proof.
\end{proof}
\subsection{Ruling out non-trivial axial semi-infinite geodesics}
\label{sec:proof-prop-refpr-1}
The goal of this section is to prove the following result.
\begin{proposition}
  \label{lem:135}
  Almost surely, for all $t\in \RR$, there does not exist any non-trivial $0$-directed or $\infty$-directed semi-infinite geodesic $\Gamma^t$.
\end{proposition}
 A version of this result was proved for static exponential LPP in the work \cite[Section 5]{BHS22}. The difference is that we are working with a dynamical model and want the result to hold uniformly for all times. We follow the same proof strategy but with some minor modifications to achieve the above uniformity. %
\subsubsection{\textbf{A mesoscopic transversal fluctuation estimate}}
\label{sec:mesosc-transv-fluct}

We begin with a version of the mesoscopic transversal fluctuation estimate \cite[Lemma 2.4]{BBBK25} (Proposition \ref{prop:53} herein) which now holds uniformly in time. Regarding notation, for a point $\0\leq q\in \ZZ_\RR$, a staircase $\xi\colon \0\rightarrow q$ and an $\ell\in \NN$, we define
\begin{equation}
  \label{eq:670}
  \TF_\ell(\xi)=\inf\{ r: \xi\cap \{\ell\}_{\RR}\subseteq B_r(\LL_{\0}^{q})\},
\end{equation}
where we note that the ball $B_r(\LL_{\0}^{q})$ is not the same as $B_r^{\euc}(\LL_{\0}^{q})$. We now have the following result.
\begin{lemma} 
  \label{lem:57}
  Fix $\delta\in (0,1/10]$. Then there exist positive constants $n_0,\ell_0$ such that for all $n\geq n_0,\ell\geq \ell_0, 0<\beta<n$, we have
  \begin{equation}
    \label{eq:189}
        \PP(\TF_\ell(\Gamma_{\0}^{(\beta n,n),t})\leq \beta\ell^{2/3+\delta}\textrm{ for all } t\in [0,1])\geq 1-Ce^{-c\ell^{3\delta}}.
  \end{equation}
\end{lemma}
In order to discuss the proof of the above, we first state an immediate consequence of Lemma \ref{lem:63}.
\begin{lemma}
  \label{lem:134}
  There exist constants $C,c$ such that for any $\delta\in (0,1/10],0<\beta<n$ and all $n\in \NN$, we have
  \begin{equation}
    \label{eq:648}
    \PP(\Gamma_{\0}^{(\beta n,n),t}\not\subseteq B_{\beta n^{2/3+\delta}}(\LL_{\0}^{(\beta n,n)})\textrm{ for some } t\in [0,1])\leq Ce^{-cn^{3\delta}}.
  \end{equation}
\end{lemma}
\begin{proof}[Proof sketch of Lemma \ref{lem:57}]
  Broadly, the proof of \cite[Lemma 2.4]{BBBK25} proceeds by showing that in static BLPP, if the transversal fluctuation is too high at a mesoscopic scale, then one can find two auxiliary points such that the geodesic between them has macroscopically large transversal fluctuation (see \cite[Figure 4]{BBBK25}), and this probability is controlled by Proposition \ref{prop:38}. Now, in order to obtain Lemma \ref{lem:57}, we precisely follow the proof of \cite[Lemma 2.4]{BBBK25} but now simply use the dynamical BLPP transversal fluctuation estimate (Lemma \ref{lem:134}) at all the instances when the corresponding static BLPP transversal estimate (Proposition \ref{prop:38}) is used.

\end{proof}

\subsubsection{\textbf{Ruling out $\0$-directed non-trivial semi-infinite geodesics}}
\label{sec:ruling-out-0}
In the following string of lemmas, we shall now use the above result to rule out non-trivial semi-infinite geodesics which are $0$-directed, that is, which are directed ``vertically upward''. For a $0$-directed semi-infinite staircase $\xi$, we say that $\xi$ has infinite width if we have $\lim_{n\rightarrow \infty}\xi(n)=\infty$, and following \cite[Section 5]{BHS22}, we shall analyse finite width and infinite width semi-infinite geodesics separately. We now have the following result which is a direct analogue of \cite[Lemma 5.2]{BHS22}; while going through the proof, it might be helpful for the reader to refer to Figure \ref{fig:Vset1}.
  \begin{figure}
    \centering
    \includegraphics[width=0.45\linewidth]{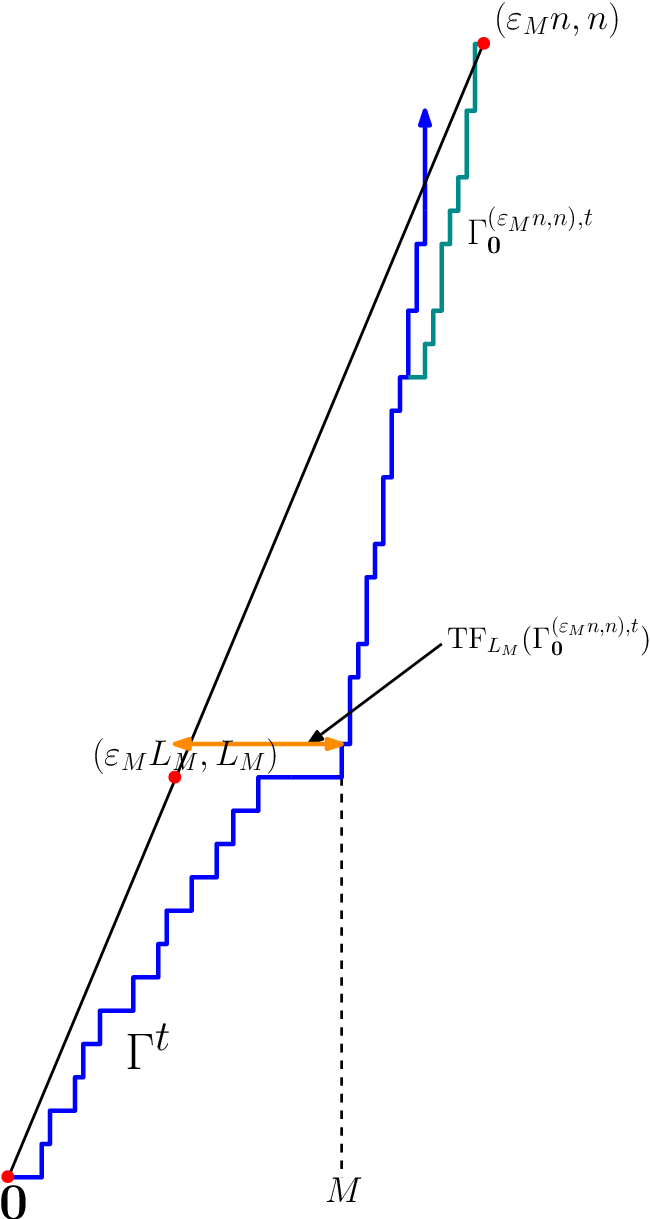}
    \caption{\textit{Proof of Proposition \ref{prop:19}}: If the geodesic $\Gamma^t$ has infinite width, then for all $M>0$, there must be an $L_M\in \NN$ for which $\Gamma^t$ intersects $[\![1,L_M]\!]_{\{M\}}$. Now, for any $\varepsilon_M>0$, since $\Gamma^t$ is $0$-directed, the point $(\varepsilon_Mn,n)$ must lie to the right of $\Gamma^t$ for all large $n$. By choosing $\varepsilon_M\leq M/(2L_M)$ and by planarity, for all large $n$, we must have $\TF_{L_M}(\Gamma_{\0}^{(\varepsilon_M n,n),t})\geq M/2$, but this has at most probability $Ce^{-cM^{3/10}}$ by transversal fluctuation estimates.}
    \label{fig:Vset1}
  \end{figure}
\begin{proposition}
  \label{prop:19}
  Almost surely, there does not exist any $t\in [0,1]$ and a point $p\in \ZZ_{\RR}$ with a $0$-directed infinite width semi-infinite geodesic $\Gamma^t$ emanating from $p$. %
\end{proposition}

\begin{proof}
 By translation invariance and the fact that any geodesic $\Gamma^t$ as above must pass via a rational point in $\ZZ_\RR$, it suffices to show that a.s.\ there does not exist any $t\in [0,1]$ with a $0$-directed infinite width semi-infinite geodesic $\Gamma^t$ emanating from $\0$. With the aim of eventually obtaining a contradiction, we assume the contrary. That is, if we let $\cA$ be the random set of times $t\in [0,1]$ and geodesics $\Gamma^t$ as above, then we assume that $\PP(\cA\neq \emptyset)=\delta>0$.

  Now, for $M,L$, let $\cA_{M,L}$ denote the set of $t\in \cA$ such that there exists a geodesic $\Gamma^t$ as above which additionally satisfies $\Gamma^t\cap [\![1,L]\!]_{\{M\}}\neq \emptyset$. Now, for each $M>0$, there must exist an $L_M$ such that $\PP(\cA_{M,L_M}\neq \emptyset)\geq \delta/2$. %

Now, given $M$, we fix $\varepsilon_M\in (0,1)$ small enough so as to satisfy $M\geq 2\varepsilon_ML_M$. Now, since we are looking at $0$-directed semi-infinite geodesics, for each $t\in \cA_{M,L_M}$, for all $n$ sufficiently large, the point $(\varepsilon_M n,n)$ must be to the right of $\Gamma^t$. Thus, if we define $\cA_{M,L_M}^{n}$ be the set of $t\in \cA_{M,L_M}$ for which the point $(\varepsilon_M n,n)$ lies to the right of $\Gamma^t$, then for all $M$ large and for $n\geq n_M$, we must have
  \begin{equation}
    \label{eq:190}
    \PP(\cA_{M,L_M}^{n}\neq\emptyset)\geq \delta/4.
  \end{equation}
  Note that by the ordering of geodesics, for any $t\in \cA_{M,L_M}^{n}$, we have
  \begin{equation}
    \label{eq:191}
    \Gamma_{\0}^{(\varepsilon_M n, n),t}\cap \{L_M\}_{\RR}\subseteq \{L_M\}_{[M,\infty)}.
  \end{equation}
  Also, note that $M-\varepsilon_ML_M\geq M/2$. Now, Lemma \ref{lem:57} yields that for some constants $C,c$, all $M$ and all $n\geq n_M$, we have
  \begin{equation}
    \label{eq:192}
    \PP(\TF_{L_M}(\Gamma_{\0}^{(\varepsilon_M n,n),t})\leq M/2\textrm{ for all } t\in [0,1])\geq 1-Ce^{-c(M/\varepsilon_M)^{3/10}}\geq 1-Ce^{-cM^{3/10}},
  \end{equation}
  where we have used $\varepsilon_M\in (0,1)$ to obtain the last inequality. However, the above along with \eqref{eq:191} implies that
  \begin{equation}
    \label{eq:193}
    \PP(\cA^n_{M,L_M}\neq\emptyset)\leq Ce^{-cM^{3/10}}
  \end{equation}
  for all $M$ large enough. This is in contradiction with \eqref{eq:190}, and the proof is complete. %
\end{proof}
It remains to consider $0$-directed semi-infinite geodesics which are of finite width. To do so, we shall first need the following easy lemma.
\begin{lemma}
  \label{lem:58}
  For any fixed $x_0>0, \ell\in \NN,\varepsilon>0$, we a.s.\ have $\Gamma_{\0}^{(x_0,n),t}\cap \{\ell\}_{\RR}\subseteq \{\ell\}_{[0,\varepsilon]}$ for all $n$ large enough.
\end{lemma}
\begin{proof}
  By using Lemma \ref{lem:134} with $\beta =x_0n^{-1}$, it follows that the probability
  \begin{equation}
    \label{eq:692}
       \PP(\Gamma_{\0}^{(x_0,n),t}\cap \{\ell\}_{[\varepsilon,\infty)}\neq\emptyset \textrm{ for some } t\in [0,1])
  \end{equation}
converges to $0$ superpolynomially fast as $n\rightarrow \infty$. The desired result now follows by applying the Borel-Cantelli lemma.
\end{proof}
We are now ready to handle $0$-directed finite height semi-infinite geodesics.
\begin{lemma}
  \label{lem:59}
  Almost surely, there does not exist any $t\in [0,1]$ and a point $p\in \ZZ_\RR$ with a non-trivial $0$-directed finite width semi-infinite geodesic emanating from $p$.
\end{lemma}
\begin{proof}
  By standard arguments, it suffices to show that there a.s.\ does not exist any $t\in [0,1]$ with a non-trivial $0$-directed finite width semi-infinite geodesic emanating from $\0$. Assume the contrary. If so, then there exists a set $\cA\subseteq [0,1]$ satisfying $\PP(\cA\neq\emptyset)=\delta>0$ such that for all $t\in \cA$, we have a geodesic $\Gamma^t$ satisfying the above. Now, there must exist an $M$ large enough such that if we define $\cA_M\subseteq \cA$ to be the set of $t\in \cA_M$ for which $\Gamma^t\subseteq \RR_{[0,M]}$, then $\PP(\cA_{M}\neq \emptyset)\geq \delta/2$. Now, we fix $M'\in \NN$ with $M'\geq M$. By the ordering of geodesics, we note that that for any $t\in \cA_M$, the geodesic $\Gamma_{\0}^{(M',n),t}$ must stay to the right of the geodesic $\Gamma^t$. However, we also know that for any fixed $\ell,\varepsilon$, we have $\Gamma_{\0}^{(M',n),t}\cap \{\ell\}_{\RR}\subseteq \{\ell\}_{[0,\varepsilon]}$ for all large enough $n$. Thus, for every $\varepsilon>0$, and $t\in \cA_M$, we must in fact have $\Gamma^t\subseteq \RR_{[0,\varepsilon]}$. Since $\varepsilon$ is arbitrary, this contradicts the non-triviality of $\Gamma^t$ and completes the proof.
\end{proof}

\subsubsection{\textbf{Ruling out $\infty$-directed non-trivial semi-infinite geodesics}}
\label{sec:ruling-out-infty}
Having ruled out non-trivial $0$-directed semi-infinite geodesics, the task now is to rule out non-trivial $\infty$-directed semi-infinite geodesics. The arguments for doing this are entirely analogous; however unlike exponential LPP, the $0$-directed case and the $\infty$-directed case are not exactly equivalent by symmetry. Thus, we now repeat the arguments from the previous section adapted to the current setting of ``entirely horizontal'' semi-infinite geodesics. For a $\infty$-directed semi-infinite staircase $\xi$, we say that $\xi$ has finite height if $\xi\cap \{n\}_{\RR}=\emptyset$ for all $n$ large enough, and otherwise, we say that $\xi$ has infinite height. We now have the following lemma which is an analogue of Proposition \ref{prop:19}.
\begin{proposition}
  \label{prop:20}
Almost surely, there does not exist any $t\in [0,1]$ and a point $p\in \ZZ_\RR$ with an $\infty$-directed infinite height semi-infinite geodesic $\Gamma^t$ emanating from $p$. %
\end{proposition}

\begin{proof}
  Again, by translation invariance, it suffices to work with $p=\0$. Let $\cA$ denote the random set of times $t\in [0,1]$ admitting a geodesic $\Gamma^t$ as in the statement of the lemma. With the aim of eventually obtaining a contradiction, we assume that $\PP(\cA\neq \emptyset)=\delta>0$.

  For $M,L$, let $\cA_{M,L}$ denote the set of $t\in \cA$ for which there exists a geodesic $\Gamma^t$ as above which satisfies $\Gamma^t\cap \{M\}_{[0,L]}\neq \emptyset$. By the requirement of infinite height in the statement of the lemma, for each $M$, there must exist an $L_M$ such that we have $\PP(\cA_{M,L_M}\neq \emptyset)\geq \delta/2$.

  Now, given $M$, we fix $\chi_M\geq 1$ large enough so as to satisfy $\chi_MM\geq 2L_M$. Since we are working with $\infty$-directed geodesics $\Gamma^t$, for each $t\in \cA_{M,K}$ and all $n$ sufficiently large, the point $(\chi_Mn,n)$ must be to the left of $\Gamma^t$. As a result, if we define $\cA_{M,L_M}^n$ as the subset of $\cA_{M,L_M}$ where the point $(\chi_M n,n)$ lies to the left of $\Gamma^t$, then for all $M$ large and for all $n\geq n_M$, we must have
  \begin{equation}
    \label{eq:196}
    \PP(\cA_{M,L_M}^n)\geq \delta/4.
  \end{equation}
  By the ordering of geodesics, for all $t\in \cA_{M,L_M}^n$, we must have
  \begin{equation}
    \label{eq:197}
    \Gamma_{\0}^{(\chi_M n, n),t}\cap \{M\}_{\RR}\subseteq \{M\}_{[0,L_M]}.
  \end{equation}
  Also, note that $\chi_M M-L_M\geq \chi_M M/2$. Now, by Lemma \ref{lem:57}, we know that for all $M$ large and $n$ large enough depending on $M$,
  \begin{equation}
    \label{eq:201}
    \PP(\TF_{M}(\Gamma_{\0}^{(\chi_M n,n),t})\leq \chi_M M/2 \textrm{ for all } t\in [0,1])\geq 1-Ce^{-cM^{3/10}}.
  \end{equation}
  When combined with \eqref{eq:197}, this implies that $\PP(\cA_{M,L_M}^n)\leq Ce^{-cM^{3/10}}$ for all $M$ large enough and $n\geq n_M$. This contradicts \eqref{eq:196} and completes the proof.
\end{proof}

It now remains to rule out non-trivial $\infty$-directed semi-infinite geodesics. The following result shall serve as a substitute of Lemma \ref{lem:58} from the previous section.

\begin{lemma}
  \label{lem:62}
Almost surely, for all $t\in [0,1]$, we have $\lim_{x\rightarrow \infty}\Gamma_{\0}^{(x,1),t}(0)=\infty$.
\end{lemma}
\begin{proof}
  It is easy to check that, almost surely, for all $t\in [0,1]$ and rational $x>0$,
  \begin{equation}
    \label{eq:693}
    \Gamma_{\0}^{(x,1),t}(0)=\argmax_{y\in [0,x]}(W_0^t(y)-W_1^t(y)).
  \end{equation}
  As a result, the condition $\lim_{x\rightarrow \infty}\Gamma_{\0}^{(x,1),t}(0)<\infty$ is equivalent to $\argmax_{y>0}(W_0^t(y)-W_1^t(y))<\infty$. Thus, it suffices to establish that almost surely, for all $t\in [0,1]$, $\sup_{y>0}(W_0^t(y)-W_1^t(y))=\infty$. Since we have the distributional equality $(W_0^t,W_1^t)_{t\in \RR}\stackrel{d}{=}(-W_0^t,-W_1^t)_{t\in \RR}$, we can reduce to simply showing that almost surely, for all $t\in [0,1]$,
  \begin{equation}
    \label{eq:698}
  \sup_{y>0}|W_0^t(y)-W_1^t(y)|=\infty.  
  \end{equation}
   However, this is not difficult to show-- indeed, first, by a small ball estimate (see \cite[Theorem 2]{Chu48}), we know that there is a constant $c$ such that the probability
  \begin{equation}
    \label{eq:695}
    \PP(\sup_{y\in [0,n]}|W_0(y)-W_1(y)|< cn^{1/2}(\log n)^{-1/2})
  \end{equation}
  decays superpolynomially in $n$. As a result, for the same constant $c$, we have
  \begin{align}
    \label{eq:694}
    &\PP(\exists t\in [0,1]: \sup_{y\in [0,n]}|W_0^t(y)-W_1^t(y)|< cn^{1/2}(\log n)^{-1/2})\nonumber\\
    &=\PP(\exists t\in \{0\}\cup \cT_{\0}^{(n,1),[0,1]}: \sup_{y\in [0,n]}|W_0^t(y)-W_1^t(y)|< cn^{1/2}(\log n)^{-1/2})\nonumber\\
                                                                          &\leq \EE[\sum_{t\in \{0\}\cup \cT_{\0}^{(n,1),[0,1]}}\PP(\sup_{y\in [0,n]}|W_0^t(y)-W_1^t(y)|< cn^{1/2}(\log n)^{-1/2}\lvert \cT_{\0}^{(n,1),[0,1]})]\nonumber\\
    &\leq (1+\EE|\cT_{\0}^{(n,1),[0,1]}|)\PP(\sup_{y\in [0,n]}|W_0(y)-W_1(y)|< cn^{1/2}(\log n)^{-1/2}).
  \end{align}
  Since $\EE|\cT_{\0}^{(n,1),[0,1]}|\leq Cn$ for some constant $C$ and since \eqref{eq:695} decays superpolynomially, the final expression in \eqref{eq:694} must also decay superpolynomially as $n\rightarrow \infty$. Now, by using the Borel-Cantelli lemma, we obtain that almost surely, for all $n$ large enough, and for all $t\in [0,1]$, we have
  \begin{equation}
    \label{eq:696}
    \sup_{y\in [0,n]}|W_0^t(y)-W_1^t(y)|\geq cn^{1/2}(\log n)^{-1/2}.
  \end{equation}
  In particular, this establishes \eqref{eq:698} and completes the proof.
\end{proof}

We are now ready to rule out non-trivial $\infty$-directed finite height semi-infinite geodesics.
\begin{lemma}
  \label{lem:60}
  Almost surely, there does not exist any $t\in \RR$ and a point $p\in \ZZ_\RR$ with a non-trivial $\infty$-directed finite height semi-infinite geodesic emanating from $p$
\end{lemma}
\begin{proof}
  It suffices to show that there a.s.\ does not exist any $t\in [0,1]$ with an $\infty$-directed finite height semi-infinite geodesic emanating from $\0$. Assume the contrary. If so, then there exists a set $\cA\subseteq [0,1]$ satisfying $\PP(\cA\neq \emptyset)>0$ such that for all $t\in \cA$, we have a geodesic $\Gamma^t$ satisfying the above. Now, there must exist $M\in \NN\in \{0\}$ such that if we define $\cA_{M}\subseteq \cA$ by the requirement that for $t\in \cA_M$, for all $x$ large enough, we have $\Gamma^t\cap \RR_{[x,\infty)}=\{M\}_{[x,\infty)}$ along with $\PP(\cA_{M}\neq \emptyset)>0$. Note that if $M=0$ in the above, then this would contradict the assumption that $\Gamma^t$ is non-trivial, and as a result, we can assume that $M\geq 1$.

  However, the above is in contradiction with Lemma \ref{lem:62}. Indeed, suppose that for some $t\in \cA_M$, we have $\Gamma^t\cap \RR_{[x,\infty)}=\{M\}_{[x,\infty)}$ for all $x\geq x_0$. Thus, for this $t$, we must have $\Gamma_{(0,M-1)}^{(x,M),t}(M-1)\leq x_0$ for all $x\geq x_0$. However, this contradicts Lemma \ref{lem:62}; thus, our initial assumption that $\PP(\cA_M\neq \emptyset)>0$ must be incorrect, and this completes the proof.
\end{proof}
\begin{proof}[Proof of Proposition \ref{lem:135}]
  The result follows by combining Proposition \ref{prop:19}, Lemma \ref{lem:59}, Proposition \ref{prop:20} and Lemma \ref{lem:60}.
\end{proof}

\subsection{Proof of Proposition \ref{prop:22}}
\label{sec:proof-prop-refpr-2}
We are now ready to complete the proof of Proposition \ref{prop:22}.

\begin{proof}[Proof of Proposition \ref{prop:22}]
  First, by a simple countable union argument, it suffices to only work with $t\in [0,1]$. We now begin by noting that for a bigeodesic $\Gamma^t$ for $t\in \RR$ and for any point $p\in \Gamma^t\cap \ZZ_\RR$, we can write $\Gamma^t=\gamma_{p,\uparrow}^{t}\cup \gamma_{p,\downarrow}^{t}$, where the former is a semi-infinite geodesic emanating from $p$ and the latter is a ``down-left'' semi-infinite geodesic emanating from $p$. Indeed, until now, we have only considered semi-infinite geodesics which traverse in an up-right fashion, but one can also similarly consider down-left semi-infinite geodesics. We note that by symmetry, all results which are true for usual up-right geodesics also have a counterpart for the down-left case.

  Now, observe that almost surely, for any $t\in [0,1]$, any non-trivial bigeodesic $\Gamma^t$ and $p\in \Gamma^t\cap \ZZ_\RR$, both $\gamma_{p,\uparrow}^{t},\gamma_{p,\downarrow}^{t}$ must be non-trivial semi-infinite geodesics. Indeed, suppose that for some $p$ as above, $\gamma_{p,\uparrow}^{t}$ is a trivial semi-infinite geodesic. Now, either there exists a $q\in \Gamma^t\cap \ZZ_\RR$ such that $\gamma_{q,\uparrow}^t$ is a non-trivial axially directed semi-infinite geodesic or the geodesic $\Gamma^t$ must be trivial itself. The former case is ruled out by Proposition \ref{lem:135}, while the latter case is ruled out since we are assuming $\Gamma^t$ to be non-trivial.

  Now, as a consequence of Proposition \ref{prop:21}, for all $p\in \Gamma^t\cap \ZZ_\RR$, there must exist $\theta^t_{p,\uparrow}, \theta^t_{p,\downarrow}\in [0,\infty]$ such that $\gamma_{p,\downarrow}^t$ is $\theta^t_{p,\downarrow}$-directed and $\gamma_{p,\uparrow}^t$ is $\theta_{p,\uparrow}^t$ directed. Since both $\gamma_{p,\uparrow}^t, \gamma_{p,\downarrow}^t$ are non-trivial as we argued in the previous paragraph, by Proposition \ref{lem:135}, we must in fact have $\theta^t_{p,\uparrow}, \theta^t_{p,\downarrow}\in (0,\infty)$. For the rest of the proof, we just define $z^t= (\Gamma^t(0),0)\in \Gamma^t\cap \ZZ_\RR$ and simply write $\gamma_{\uparrow}^t=\gamma_{z^t,\uparrow}^t$, $\gamma_{\downarrow}^t=\gamma_{z^t,\downarrow}^t$ and $\theta^t_{\uparrow}= \theta^t_{z^t,\uparrow}, \theta^t_{\downarrow}=\theta^t_{z^t,\downarrow}$.

  In view of Lemma \ref{lem:69}, to complete the proof, the goal now is to simply establish that almost surely, for any $t\in [0,1]$ admitting a non-trivial bigeodesic $\Gamma^t$, we must have $\theta_\uparrow^t=\theta_\downarrow^t$.

  In fact, by a basic countable union argument, it suffices to fix a $\mu\in (0,1)$ and establish that almost surely, for any $t\in [0,1]$ admitting a bigeodesic $\Gamma^t$ satisfying $\theta_\uparrow^t,\theta_\downarrow^t\in [\mu,\mu^{-1}]$, we have $\theta_\uparrow^t=\theta_\downarrow^t$. The goal of the remainder of the proof is to establish this.

  Fix $\varepsilon \in (0,1/10]$. As a consequence of Lemma \ref{lem:63}, we know that for any $(x,-n)$ and $(y,n)$ with $(-x),y\in [\mu n/ 2, 2\mu^{-1} n]$, for some constants $C,c$, we have
  \begin{equation}
    \label{eq:217}
    \PP(\Gamma_{(x,n)}^{(y,n),t}\subseteq B_{n^{2/3+\varepsilon}}(\LL_{(x,n)}^{(y,n)})\textrm{ for all }t\in [0,1])\geq 1-Ce^{-cn^{3\varepsilon}}.
  \end{equation}
  We now define the event $\cE_n$ by
  \begin{equation}
    \label{eq:219}
   \cE_n= \{\Gamma_{(x,n)}^{(y,n),t}\subseteq B_{n^{2/3+\varepsilon}}(\LL_{(x,n)}^{(y,n)})\textrm{ for all } (-x),y\in [\mu n/ 2, 2\mu^{-1} n], \textrm{ all } t\in [0,1], \textrm{ all geodesics } \Gamma_{(x,n)}^{(y,n),t}\}.
  \end{equation}
  By using \eqref{eq:217} and taking a union bound over a mesh of $x,y$ and using the ordering of geodesics, it can be obtained that for some constants $C',c'$,
  \begin{equation}
    \label{eq:218}
     \PP(\cE_n)\geq 1-C'e^{-c'n^{3\varepsilon}}.
   \end{equation}
   By applying the Borel-Cantelli lemma, we immediately obtain that a.s.\ the events $\cE_n^c$ occur only for finitely many $n$. Since $[\mu,\mu^{-1}]\subseteq (\mu/2,2\mu^{-1})$, the above immediately implies that almost surely, for every $t\in [0,1]$ admitting a bigeodesic $\Gamma^t$, and the sequence $p_n^t=(\Gamma^t(-n-1),-n)$ and $q_n^t=(\Gamma^t(n),n)$, we must have
   \begin{equation}
     \label{eq:220}
     \Gamma^t\cap [-n,n]_\RR\subseteq B_{n^{2/3+\varepsilon}}(\LL_{p_n^t}^{q_n^t})
   \end{equation}
   for all $n$ large enough. Now, if it were true that $\theta_\uparrow^t\neq \theta_\downarrow^t$, then with $z^t=(\Gamma^t(0),0)$, we would necessarily have %
   \begin{equation}
     \label{eq:221}
     d(z^t,\LL_{p_n^t}^{q_n^t})\geq Cn
   \end{equation}
 for all $n$ large enough, with $d(\cdot,\cdot)$ denoting the Euclidean distance and $C>0$ being a constant depending on $\theta_\uparrow^t,\theta_\downarrow^t$. This is in contradiction with \eqref{eq:220}. This completes the proof.
 \end{proof}

 \section{Appendix 2: Brownian regularity estimates for BLPP line ensembles}
\label{sec:brownian-regularity}
Recall the line ensemble $\cP$ associated to BLPP from Section \ref{sec:ensemble}-- we recall that the dependence of $\cP$ on the parameter $n$ is suppressed for notational ease. As discussed earlier in Section \ref{sec:brown}, a very useful property of $\cP$ is that when viewed locally, each of the individual lines $\cP_k$ are absolutely continuous to a Brownian motion of diffusivity $2$-- many versions of such statements have been developed in the past decade and have led to a various applications. Recently, the work \cite{Dau23} obtained a fine Brownianity result for the Airy line ensemble, the appropriate distributional scaling limit of $\cP$ as $n\rightarrow \infty$. In our work, we require an analogue of the above-mentioned result for the BLPP line ensemble $\cP$. Broadly, the proofs from \cite{Dau23} do adapt to yield the BLPP statements that we require, and the goal of this section is to give an outline the proofs in this case, placing emphasis on the adaptations needed. Throughout this appendix, to improve readability, we shall try to use the same notation as in \cite{Dau23} and explicitly mention the analogous result from \cite{Dau23} for each result stated here-- in case, the proof is the same with only minor differences, we omit it. We now state the main result, which is a finite-ensemble analogue of \cite[Theorem 1.8]{Dau23}.

\begin{proposition}
  \label{prop:24}

{
Fix $k\in \NN,t\geq 1$, and let $\mathbf{a}\in \RR$ satisfy $(\mathbf{a},\mathbf{a}+t)\subseteq [-n^{1/10},n^{1/10}]$. Define $U(\mathbf{a})=[\![1,k]\!]\times (\mathbf{a},\mathbf{a}+t)$. Then, as long as $n$ is sufficiently large depending on $k$, there exists a random sequence of continuous functions $\fL^{\mathbf{a}}=\fL^{t,k,\mathbf{a}}$ such that the following hold.\footnote{Strict inequalities between successive lines refer to indices $1\leq i\leq n$ and horizontal coordinates $r>-n^{1/3}/2$; the common left endpoint is excluded.} 
\par}

  \begin{enumerate}
  \item Almost surely, $\fL^{\mathbf{a}}$ satisfies $\fL_i^{\mathbf{a}}(r)>\fL_{i+1}^{\mathbf{a}}(r)$ for all $(i,r)\notin U(\mathbf{a})$.
  \item

{
The line ensemble $\fL^{\mathbf{a}}$ satisfies the following Gibbs property. For $1\leq m\leq n+1$ and $-n^{1/3}/2<a<b<\infty$, set $S=[\![1,m]\!]\times[a,b]$. Then conditional on the data given by $\fL_i^{\mathbf{a}}(r)$ for $(i,r)\notin S$, the law of $\fL_i^{\mathbf{a}}(r)$ for $(i,r)\in S$ is given by $m$ independent Brownian bridges $B_1,\dots,B_m$ of diffusivity $2$ from $(a,\fL_i^{\mathbf{a}}(a))$ to $(b,\fL_i^{\mathbf{a}}(b))$ for $i\in[\![1,m]\!]$, additionally conditioned on having $B_i(r)>B_{i+1}(r)$ whenever $1\leq i\leq n$, $r\in[a,b]$ and $(i,r)\notin U(\mathbf{a})$. Here $B_i(r)=\fL_i^{\mathbf{a}}(r)$ for $i>m$.
\par}
  \item There exists a constant $c_k$ for which we have
    \begin{equation}
      \label{eq:276}
      \PP(\fL_1^{\mathbf{a}}>\fL_2^{\mathbf{a}}>\dots>\fL_{n+1}^{\mathbf{a}})\geq e^{-c_kt^3}.
    \end{equation}
  \item

{
If we condition $\fL^{\mathbf a}$ on the event $\{\fL_1^{\mathbf{a}}>\fL_2^{\mathbf{a}}>\dots>\fL_{n+1}^{\mathbf{a}}\}$, then the resulting ensemble has the same law as $\cP$.
\par}

  \end{enumerate}
\end{proposition}
For our application, we will also need some one point tail bounds on $\fL^{\mathbf{a}}$. While \cite{Dau23} proves stronger and more general bounds in the Airy line ensemble setting, in the interest of brevity, we only prove the result the following specific result that we shall require for our application.
\begin{proposition}
  \label{prop:25}
For $t\geq 1$ and $(\mathbf{a},\mathbf{a}+t)\subseteq [-n^{1/10},n^{1/10}]$, consider $\fL^{\mathbf{a}}=\fL^{t,1,\mathbf{a}}$. Then there exist constants $c,c'$ such that for all $x\in \{\mathbf{a},\mathbf{a}+t\}$, and all $r>0$, we have %
  \begin{equation}
    \label{eq:293}
    \PP( |\fL^{\mathbf{a}}_1(x)+x^2|\geq r)\leq e^{c't^3}e^{-cr^{3/2}}.
  \end{equation}
\end{proposition}

\paragraph{\textbf{Notation from \cite{Dau23} used in this appendix}}
Before starting with the proofs of Propositions \ref{prop:24}, \ref{prop:25}, we import some notation from \cite{Dau23}. We shall work with continuous functions throughout, and for an interval $I\subseteq \RR$, we let $\mathscr{C}^k(I)$ denote the sequence of tuples $f=(f_1,\dots,f_k)$ such that each $f_i\colon I\rightarrow \RR$ is continuous; thus, for instance, we have $\cP\in \mathscr{C}^{n+1}([-n^{1/3}/2,\infty))$ almost surely. For such a tuple $f$, we use $f^\ell$ to denote the first $\ell$ coordinates of $f$, that is, $f=(f_1,\dots,f_\ell)$; for example, for $\ell\leq n$, we have $\cP^\ell=(\cP_1,\dots,\cP_\ell)$. Now for a fixed $k\in \NN$, an interval $I\subseteq \RR$, a set $J\subseteq \RR$, and a function $g$ defined on $I$, we consider the non-intersecting collection
\begin{equation}
  \label{eq:577}
  \NI(g,J)=\{f\in \mathscr{C}^k(I): f_1(s)> f_2(s)>\dots>f_k(s)> g(s) \textrm{ for all } s\in J\}.
\end{equation}
Note that in the above, we also allow $g=-\infty$, the function which is identically equal to $-\infty$. Further, in the above setting, for $J\subseteq I=[s,t]$ and $\mathbf{x},\mathbf{y}\in \RR^k$, we use $\PP_{s,t}(\mathbf{x},\mathbf{y}, g, J)$ to denote the probability that a sequence of independent Brownian motions $B=(B_1,\dots, B_k)$ satisfying $B_i(s)=\mathbf{x}_i$ and $B_i(t)=\mathbf{y}_i$ for all $i\in [\![1,k]\!]$ satisfy the event $\{B\in \NI(g,J)\}$. In case $J=I=[s,t]$, it is omitted from the notation and we simply write $\PP_{s,t}(\mathbf{x},\mathbf{y}, g)$.

\subsection{Proof outline for Proposition \ref{prop:24}}
\label{sec:proof-outl-prop}
With regard to Proposition \ref{prop:24}, we shall primarily discuss the special case $\mathbf{a}=0$. The same proofs works for general values of $\mathbf{a}$ satisfying $(\mathbf{a},\mathbf{a}+t)\subseteq [-n^{1/10},n^{1/10}]$ as well, and we shall give a short discussion of this later. We now set up the relevant objects used in the proofs later.

Now, consider the intervals $[0,t]$ and $[-s,s]$ for $s>t$. Define the line ensemble $\fB$ by $\fB_i(r)=\cP_{i}(r)$ for $(i,r)\notin [\![1,k]\!]\times [0,t]$ and let $\fB^k\lvert_{[0,t]}$ be given by $k$ independent Brownian bridges from $(0, \cP^k(0))$ to $(t,\cP^k(t))$. Use $\tilde{\eta}$ to denote the law of $\fB$; note that, just as $\cP$, $\fB$ and $\tilde{\eta}$ depend on $n$ as well. Let $\eta$ be the measure which is absolutely continuous to $\tilde{\eta}$ with density given by
\begin{equation}
  \label{eq:222}
  \frac{d\eta}{d\tilde{\eta}}(f)= \frac{1}{\PP_{0,t}(f^k(0),f^k(t),f_{k+1})}.
\end{equation}

Fix $s>t$ and define a measure $\eta_s$ by the following procedure. Define a line ensemble $\fB_s$ by $\fB_{s,i}(r)=\cP_i(r)$ for $(i,r)\notin [\![1,k]\!]\times [-s,s]$. Now, given $\cP$ outside the region $[\![1,k]\!]\times [-s,s]$, let $\fB_s^k\lvert_{[-s,s]}$ consist of $k$ independent Brownian bridges from $(-s,\cP^k(-s))$ to $(s,\cP^k(s))$ conditioned on the event $\NI(\cP_{k+1},[-s,0]\cup [t,s])$. We use $\tilde \eta_s$ to denote the law of $\fB_s$. Now, let $\eta_s$ be the measure which is absolutely continuous with respect to $\tilde{\eta}_s$ with density given by
\begin{equation}
  \label{eq:223}
  \frac{d\eta_s}{d\tilde{\eta}_s}(f)=\frac{\PP_{-s,s}(f^k(-s),f^k(s),f_{k+1},[-s,0]\cup [t,s])}{\PP_{-s,s}(f^k(-s),f^k(s),f_{k+1})}.
\end{equation}
Often, we shall use the $\sigma$-algebra $\cF_s$ generated by the values $\cP_i(x)$ for $(i,x)\notin [\![1,k]\!]\times [-s,s]$ and will use $\EE_{\cF_s}$ to denote conditional expectation with respect to it.

{
The objects introduced above are all taken from \cite[Section 3]{Dau23}. We shall define $\fL^{\mathbf a}=\fL^0$ to have the law
\begin{equation}\label{eq:584}
\frac{\eta}{\eta\bigl(\mathscr C^{n+1}([-n^{1/3}/2,\infty))\bigr)}.
\end{equation}
In order to justify this definition, the main task is to prove that $\eta$ is a finite measure, as was done in \cite{Dau23} for the corresponding measure for the Airy line ensemble. We shall simply state lemmas without proof if the proof is the same as the one in \cite{Dau23}, and shall provide more details for the results where there are substantial differences. The following two results rely only on the Brownian Gibbs property, which holds both for $\cP$ and for the Airy line ensemble.
\par}
\begin{lemma}[{\cite[Lemma 3.1]{Dau23}}]
  \label{lem:71}
  For all $n$ and all $t<s<n^{1/3}/2$, we have $\eta=\eta_s$ and further, we have
  \begin{equation}
    \label{eq:224}
    \EE_{\cF_s}\left[\frac{1}{\PP_{0,t}(\cP^k(0),\cP^k(t),\cP_{k+1})}\right]=\frac{1}{\EE_{\cF_s}[\PP_{0,t}(\fB_s^k(0),\fB_s^k(t),\cP_{k+1})]}.
  \end{equation}
\end{lemma}

\begin{lemma}[{\cite[Lemma 3.2]{Dau23}}]
  \label{lem:73}
  Consider the $\cF_s$-measurable random variables
  \begin{align}
    \label{eq:227}
    D&=\sqrt{t}+\max_{r,r'\in [0,t]}|\cP_{k+1}(r)-\cP_{k+1}(r')|\nonumber\\
    M&=\sqrt{s}+\max_{r,r'\in [-s,s]}|\cP_{k+1}(r)-\cP_{k+1}(r')|+ \max_{i\in[\![1,k]\!]}|\cP_i(s)-\cP_i(-s)|.
  \end{align}
  Then for all $2\leq 2t<s<n^{1/3}/2$ and all $n$, we have

{
\begin{equation}
    \label{eq:228}
    \EE_{\cF_s}[\PP_{0,t}(\fB_s^k(0),\fB_s^k(t),\cP_{k+1})]\geq \exp( -ck^3s^{-1}(D^2+MD)-cks).
  \end{equation}
  
\par}

  \end{lemma}

{
In order to prove that $\eta$ is a finite measure, we need to analyse the tails of $D$ and $M$. For the increment estimate below, instead of the step-back argument used in \cite[Lemma 3.4]{Dau23} for the Airy line ensemble, we use strong moment estimates for increments of Dyson Brownian motion as obtained in the recent work \cite[Theorem 1.1]{Wu24}. This avoids the step back argument used in \cite[Lemma 3.4]{Dau23}, where one moves further back as the requested deviation grows, which in our setting could take us outside the domain $[-n^{1/3}/2,\infty)$ since we are working with prelimiting BLPP line ensembles instead of the Airy line ensemble. 
\par}

{
Before stating the increment estimate, define, for $x\geq-n^{1/3}/2$,
\begin{equation}\label{eq:247}
\phi_n(x)=n^{-1/3}\left(2n+2n^{2/3}x-2\sqrt{n(n+2n^{2/3}x)}\right).
\end{equation}
\par}
\begin{lemma}[{\cite[Theorem 1.1]{Wu24}}]
  \label{lem:75}

{
For every $k\in\NN$, there exist positive constants $C_k,c_k$ such that the following holds for $n\geq1$ with $k\leq n+1$. For all $s\in(-n^{1/3}/4,n^{1/3}/4)$, $r\in(0,n^{1/3})$ and $a>0$, we have
\begin{equation}\label{eq:232}
\PP\bigl(|\cP_k(s)+\phi_n(s)-\cP_k(s+r)-\phi_n(s+r)|>a\bigr)
\leq C_k\exp\left(-c_k\frac{a^2}{r}\right).
\end{equation}
\par}
\end{lemma}

\begin{proof}

{
Put $Q_k(x)=\cP_k(x)+\phi_n(x)$. Writing $\tau_x=n+2n^{2/3}x$, the definition of $\cP$ gives
\[
Q_k(x)=n^{-1/3}\bigl(P_{k,n}(\tau_x)-2\sqrt{n\tau_x}\bigr).
\]
The ensemble $(P_{i,n})_{i=1}^{n+1}$ has the law of Dyson Brownian motion with $\beta=2$, started with all coordinates at zero. Let $\widehat Q_k=Q_k-\EE Q_k$. Applying part (i) of \cite[Theorem 1.1]{Wu24} at the times $\tau_s$ and $\tau_{s+r}$, whose difference is $2n^{2/3}r$, yields, for every $p\geq1$,
\begin{equation}\label{eq:wu-centered-increment}
\EE|\widehat Q_k(s+r)-\widehat Q_k(s)|^p
\leq (2r)^{p/2}\EE|G|^p,
\qquad G\sim N(0,1).
\end{equation}
Note that both times $s$ and $s+r$ are positive: throughout the stated range they lie between $n/2$ and $7n/2$.

It remains to account for the mean. For this, note that Brownian scaling yields
\begin{equation}\label{eq:wu-mean}
\EE Q_k(x)=m_{k,n}\sqrt{1+2xn^{-1/3}},
\qquad m_{k,n}=\EE Q_k(0).
\end{equation}
The moderate and large deviation bounds in Proposition \ref{prop:47} imply $|m_{k,n}|\leq C_k$, uniformly in $n$. Differentiating \eqref{eq:wu-mean} on $[s,s+r]$ and using $r\leq n^{1/3}$ gives
\[
|\EE Q_k(s+r)-\EE Q_k(s)|\leq C_krn^{-1/3}\leq C_k\sqrt r.
\]
Combining this with \eqref{eq:wu-centered-increment} and the Gaussian moment bound gives
$\|Q_k(s+r)-Q_k(s)\|_p\leq C_k\sqrt{pr}$ for every $p\geq1$. If $a\geq eC_k\sqrt r$, take $p=a^2/(e^2C_k^2r)\geq1$. Markov's inequality then gives
\[
\begin{aligned}
\PP\bigl(|Q_k(s+r)-Q_k(s)|>a\bigr)
&\leq \frac{\EE|Q_k(s+r)-Q_k(s)|^p}{a^p}\\
&\leq \left(\frac{C_k\sqrt{pr}}{a}\right)^p
=e^{-p}=\exp\left(-\frac{a^2}{e^2C_k^2r}\right).
\end{aligned}
\]
This proves \eqref{eq:232} in this range; increasing the prefactor in that bound covers $0<a<eC_k\sqrt r$.
\par}
\end{proof}

{
Lemma \ref{lem:75}, combined with the chaining argument in \cite[Lemma 2.3]{Dau23}, gives the following oscillation bounds.
\par}
\begin{lemma}[{\cite[Lemma 3.5]{Dau23}}]
  \label{cor:2}

{
For every $k\in\NN$, there exist positive constants $c_k,c_k'$ such that, for all $n\geq k$, all $2\leq2t<s<n^{1/3}/4$ and all $a>0$,
\begin{equation}\label{eq:242}
\PP(D>a\sqrt t)\leq e^{c_kt^3-c_k'a^2},\qquad
\PP(M>a\sqrt s)\leq e^{c_ks^3-c_k'a^2}.
\end{equation}
\par}
\end{lemma}

\begin{proof}

{
Write $Q_i=\cP_i+\phi_n$. Applying \cite[Lemma 2.3]{Dau23} to the increment bound in Lemma \ref{lem:75} gives, for every interval $I\subseteq[-n^{1/3}/4,n^{1/3}/4]$ of length $T>0$,
\begin{equation}\label{eq:468}
\PP\left(\sup_{x,y\in I}|Q_i(x)-Q_i(y)|>b\sqrt T\right)
\leq C_i e^{-c_i b^2},\qquad b>0.
\end{equation}
To pass from $Q_i$ back to $\cP_i$, note that $\phi_n(0)=\phi_n'(0)=0$ and
\begin{equation}\label{eq:248}
0<\phi_n''(x)=2(1+2xn^{-1/3})^{-3/2}\leq6,
\qquad |x|\leq n^{1/3}/4.
\end{equation}
Thus the oscillation of $\phi_n$ on $[0,t]$ is at most $Ct^2$, and on $[-s,s]$ it is at most $Cs^2$. Writing $\operatorname{osc}_I Q=\sup_{x,y\in I}|Q(x)-Q(y)|$, the definitions in \eqref{eq:227} give
\[
\begin{aligned}
D&\leq\sqrt t+Ct^2+\operatorname{osc}_{[0,t]}Q_{k+1},\\
M&\leq\sqrt s+Cs^2+\operatorname{osc}_{[-s,s]}Q_{k+1}
       +\max_{i\in[\![1,k]\!]}|Q_i(s)-Q_i(-s)|.
\end{aligned}
\]
After division by $\sqrt t$ or $\sqrt s$, the deterministic terms above are at most $C_kt^{3/2}$ or $C_ks^{3/2}$, since $t,s\geq1$. Apply \eqref{eq:468} to the random term in the bound for $D$. For $M$, use a union bound over the oscillation term and the $k$ endpoint increments, noting that the interval has length $2s$. For some constants $C_k\geq1$ and $b_k>0$, this gives
\[
\begin{aligned}
\PP(D>a\sqrt t)&\leq C_k\exp\bigl(-b_k(a-C_kt^{3/2})_+^2\bigr),\\
\PP(M>a\sqrt s)&\leq C_k\exp\bigl(-b_k(a-C_ks^{3/2})_+^2\bigr),
\end{aligned}
\]
where $x_+=\max\{x,0\}$; when the positive part is zero, the corresponding bound is trivial. To obtain \eqref{eq:242}, use $(a-v)_+^2\geq a^2/2-v^2$ for $a,v\geq0$. In particular,
\[
C_k\exp\bigl(-b_k(a-C_kt^{3/2})_+^2\bigr)
\leq C_k\exp\left(b_kC_k^2t^3-\frac{b_k}{2}a^2\right),
\]
and the same calculation with $s$ in place of $t$ applies to $M$. Absorbing the prefactor $C_k$ into $e^{c_kt^3}$ or $e^{c_ks^3}$, using $t,s\geq1$, proves \eqref{eq:242}. Thus the cubic terms arise by squaring the normalised deterministic allowances $t^2/\sqrt t=t^{3/2}$ and $s^2/\sqrt s=s^{3/2}$.
\par}
\end{proof}

The above lemma can now be used to provide tail estimates on the quantity from Lemma \ref{lem:71}.
\begin{lemma}[{\cite[Corollary 3.6]{Dau23}}]
  \label{cor:3}

{
There exist constants $d_k,c_k'>0$ such that, for all $2\leq2t<s<n^{1/3}/4$ and all $\varepsilon\in(0,1]$, with $\alpha=\log\varepsilon^{-1}$, we have
\begin{equation}\label{eq:258}
\PP\left(\EE_{\cF_s}[\PP_{0,t}(\fB_s^k(0),\fB_s^k(t),\fB_{s,k+1})]\leq\varepsilon\right)
\leq e^{c_k's^3}e^{-d_k\alpha\sqrt{s/t}}.
\end{equation}
\par}
\end{lemma}
\begin{proof}

{
Without loss of generality, we can assume $\alpha\geq C_ks^{5/2}t^{1/2}$ for a large constant $C_k$, since otherwise the claimed bound is at least one after increasing $c_k'$. Now, by applying Lemma \ref{cor:2}, we obtain that for an appropriately chosen constant $c_k''$, if we define
\[
D'=t^{-1/2}D-c_k''t^{3/2},\qquad
M'=s^{-1/2}M-c_k''s^{3/2},\qquad X=0\vee D'\vee M',
\]
then, for a constant $b_k>0$ and all $a>0$, we have
\par}

{
\begin{equation}\label{eq:469}
\PP(X>a)\leq e^{-b_ka^2}.
\end{equation}
\par}

  Now, by Lemma \ref{lem:73}, for some constant $\gamma_k$, we obtain
  \begin{equation}
    \label{eq:244}
    \EE_{\cF_s}[\PP_{0,t}(\fB_s^k(0),\fB_s^k(t),\fB_{s,k+1})]\geq \exp(-\gamma_k(X^2\sqrt{t/s}+s\sqrt{t}X+st^2))
  \end{equation}

{
Consequently,
\par}

  \begin{align}
    \label{eq:245}
    \PP(\gamma_k(X^2\sqrt{t/s}+s\sqrt{t}X+st^2)\geq \alpha)&\leq \PP(X^2\geq \gamma_{k}^{-1} \sqrt{s/t}\alpha/4)+ \PP(X\geq \gamma_k^{-1}\alpha/(4s\sqrt{t}))\nonumber\\
    &\leq 2\PP(X^2\geq \gamma_k^{-1} \sqrt{s/t}\alpha/4)\nonumber\\
    &\leq 2\exp(-d_k(\alpha \sqrt{s/t})).
  \end{align}

{
For the first inequality, the deterministic term is absorbed using $\alpha\geq C_ks^{5/2}t^{1/2}$. Increasing $C_k$ also ensures that the threshold in the second probability is at least $(\gamma_k^{-1}\alpha\sqrt{s/t}/4)^{1/2}$, giving the second inequality. The last inequality follows from \eqref{eq:469}. Absorbing the factor $2$ in $e^{c_k's^3}$ completes the proof.
\par}

\end{proof}

{
With $d_k\leq1$ being the constant in Lemma \ref{cor:3}, define $s(t)=4d_k^{-2}t$. Then $d_k\sqrt{s(t)/t}=2$. For $t\leq2n^{1/10}$ and sufficiently large $n$ depending on $k$, we have $s(t)<n^{1/3}/4$. This fixed enlargement makes the inverse acceptance probability integrable.
\par}
\begin{lemma}
  \label{lem:78}

{
There exists a constant $c_k$ such that, for all sufficiently large $n$ depending on $k$ and all $1\leq t\leq2n^{1/10}$,
\begin{equation}\label{eq:259}
\begin{aligned}
\eta\bigl(\mathscr C^{n+1}([-n^{1/3}/2,\infty))\bigr)&=\EE\left[\frac{1}{\EE_{\cF_{s(t)}}[\PP_{0,t}(\fB_{s(t)}^k(0),\fB_{s(t)}^k(t),\fB_{s(t),k+1})]}\right]\\
&\leq e^{c_kt^3}.
\end{aligned}
\end{equation}
\par}

\end{lemma}

\begin{proof}

{
The equality in \eqref{eq:259} follows from \eqref{eq:222} and Lemma \ref{lem:71}, initially as an equality of nonnegative integrals. Lemma \ref{cor:3} and the choice of $s(t)$ give
\begin{equation}\label{eq:261}
\PP\left(\EE_{\cF_{s(t)}}[\PP_{0,t}(\fB_{s(t)}^k(0),\fB_{s(t)}^k(t),\fB_{s(t),k+1})]\leq\varepsilon\right)
\leq e^{C_kt^3}\varepsilon^2,
\qquad 0<\varepsilon\leq1.
\end{equation}
The conditional probability on the left is therefore positive almost surely. Since it is at most one, Tonelli's theorem and \eqref{eq:261}, with $\varepsilon=u^{-1}$, give
\[
\EE\left[\frac{1}{\EE_{\cF_{s(t)}}[\PP_{0,t}(\fB_{s(t)}^k(0),\fB_{s(t)}^k(t),\fB_{s(t),k+1})]}\right]
\leq1+e^{C_kt^3}\int_1^\infty u^{-2}\,du
\leq e^{c_kt^3},
\]
after increasing $c_k$, as $t\geq1$.
\par}

\end{proof}

\begin{proof}[Proof of Proposition \ref{prop:24} in the case $\mathbf{a}=0$]

{
Lemma \ref{lem:78} shows that $\eta$ is finite, so \eqref{eq:584} defines $\fL^0$.
\par}

{
Verifying that the above-defined $\fL^{\mathbf{a}}$ satisfies the remaining properties in the statement of Proposition \ref{prop:24} is the same as the corresponding steps in the completion of proof of Theorem 1.8 in \cite{Dau23}, and for this reason, we omit the remaining details.
\par}

\end{proof}
Until now we have discussed Proposition \ref{prop:24} for the special case $\mathbf{a}=0$. We now briefly discuss how the result can be obtained for general values of $\mathbf{a}$-- the proof is the same, albeit with notational changes. For an interval $A=[c,d]$, let $\cF_{A}$ denote the $\sigma$-algebra generated by the set $\{\cP_i(x): (i,x)\notin [\![1,k]\!]\times A\}$. For an interval $B=[a,b]\subseteq [c,d]$, we define a line ensemble $\fB_{B,A}$ in a similar manner to the definition of $\fB_s$ earlier. That is, we define $\fB_{B,A,i}(r)=\cP_i(r)$ for all $(i,r)\notin [\![1,k]\!]\times A$ and conditionally on this, we define $\fB_{B,A}^k\lvert_{A}$ to be a Brownian bridge from $(c,\cP^k(c))$ to $(d,\cP^k(d))$ additionally conditioned on the event $\NI(\cP_{k+1},A\setminus B)$. For a line ensemble $f$, consider the random variable
\begin{equation}
  \label{eq:268}
  Z(f,B)= \PP_{a,b}(f^k(a),f^k(b),f_{k+1}).
\end{equation}
The following is an analogue of Lemma \ref{lem:71}.

\begin{lemma}
  \label{lem:80}
  For any $A,B$ as above, we have
  \begin{equation}
    \label{eq:269}
    \EE_{\cF_A}[Z(\fB_{B,B},B)^{-1}]=[\EE_{\cF_A}Z(\fB_{B,A},B)]^{-1}.
  \end{equation}
\end{lemma}

By the same proof as in the case of $\mathbf{a}=0$, one can obtain the following analogue of Lemma \ref{cor:3}.
\begin{lemma}
  \label{lem:81}

{
With $d_k$ as in Lemma \ref{cor:3}, there exists a constant $c_k>0$ such that, for all $2\leq2t<s$ and $\mathbf a$ satisfying $[\mathbf a-s,\mathbf a+s]\subseteq[-n^{1/3}/4,n^{1/3}/4]$, and all $\varepsilon\in(0,1]$, with $\alpha=\log\varepsilon^{-1}$,
\begin{equation}\label{eq:2581}
\PP\left(\EE_{\cF_{[\mathbf a-s,\mathbf a+s]}}
Z(\fB_{[\mathbf a,\mathbf a+t],[\mathbf a-s,\mathbf a+s]},[\mathbf a,\mathbf a+t])\leq\varepsilon\right)
\leq e^{c_ks^3}e^{-d_k\alpha\sqrt{s/t}}.
\end{equation}
\par}

\end{lemma}

{
\begin{proof}[Proof sketch of Proposition \ref{prop:24} for general $\mathbf a$]
We first explain the shifted estimates used in Lemma \ref{lem:81}. Write $B=[\mathbf a,\mathbf a+t]$ and $A_s=[\mathbf a-s,\mathbf a+s]$, and suppose $A_s\subseteq[-n^{1/3}/4,n^{1/3}/4]$. Add the same affine function to every line, writing
\[
\widehat{\cP}_i(x)=\cP_i(x)+g_{\mathbf a}(x),\qquad
g_{\mathbf a}(x)=\phi_n(\mathbf a)+\phi_n'(\mathbf a)(x-\mathbf a).
\]
The difference between $Q_i=\cP_i+\phi_n$ and $\widehat{\cP}_i$ is $\phi_n-g_{\mathbf a}$. By \eqref{eq:248}, its oscillation is at most $Ct^2$ on $[\mathbf a,\mathbf a+t]$ and $Cs^2$ on $[\mathbf a-s,\mathbf a+s]$. Thus \eqref{eq:468} gives the same tail bounds as in Lemma \ref{cor:2} for the quantities in \eqref{eq:227} formed from $\widehat{\cP}$ on these shifted intervals.

To use these quantities in the non-intersection estimate, observe that adding $g_{\mathbf a}$ to every bridge and to its lower boundary shifts the endpoints and leaves all pairwise differences unchanged. Since $g_{\mathbf a}$ is affine, the shifted paths are still Brownian bridges of the same diffusivity. Thus both the non-intersection probabilities $Z(f,B)$ and the conditioning outside $B$ in the construction of $\fB_{B,A_s}$ are preserved. The bridge estimate \eqref{eq:228}, after translating the intervals by $-\mathbf a$, can therefore be applied using the quantities formed from $\widehat{\cP}$. Repeating the proof of Lemma \ref{cor:3} gives Lemma \ref{lem:81}.

For the parameters in Proposition \ref{prop:24}, $|\mathbf a|\leq n^{1/10}$ and $t\leq2n^{1/10}$. Thus $|\mathbf a|+s(t)\leq(1+8d_k^{-2})n^{1/10}<n^{1/3}/4$ for sufficiently large $n$ depending on $k$, so Lemma \ref{lem:81} applies on this enlarged interval. Integrating its bound as in Lemma \ref{lem:78} shows that the shifted version of $\eta$ has total mass at most $e^{C_kt^3}$. Normalizing it and repeating the ordering argument proves Proposition \ref{prop:24} for general $\mathbf a$.
\end{proof}
}

\subsection{One point tail bounds}
\label{sec:one-point-tail}
The goal now is to outline the proof of Proposition \ref{prop:25}. First, we shall need the following comparison lemma.
\begin{lemma}
  \label{lem:100}

{
Work in the setting of Proposition \ref{prop:24}, with the fixed choice $s=s(t)=4d_k^{-2}t$. Write $B=[\mathbf a,\mathbf a+t]$ and $A_s=[\mathbf a-s,\mathbf a+s]$. For a Borel set $\cA\subseteq\mathscr C^{n+1}([-n^{1/3}/2,\infty))$, let $\beta=\PP(\fB_{B,A_s}\in\cA)$. Then for some constants $C_k,c_k$, we have
\begin{equation}\label{eq:472}
\PP(\fL^{\mathbf a}\in\cA)\leq C_ke^{C_kt^3}\beta^{1/2}.
\end{equation}
\par}
\end{lemma}

\begin{proof}

{
  Begin by noting that for some constant $C_*$, the ensemble $\fL^{\mathbf{a}}$ has the following Radon-Nikodym density with respect to $\fB_{B_,A_s}$:
\begin{equation}\label{eq:291}
X=\frac{C_*}{\EE_{\cF_{A_s}}Z(\fB_{B,A_s},B)}.
\end{equation}
Below, $X$ is evaluated at $\fB_{B,A_s}$, and its tail is taken under this reference law. Since the denominator in \eqref{eq:291} is the conditional expectation of a probability involving brownian bridges, $X\geq C_*$. The normalisation $\EE X=1$ therefore gives $C_*\leq1$.

With $s=4d_k^{-2}t$, Lemma \ref{lem:81} gives
\[
\PP\left(\EE_{\cF_{A_s}}Z(\fB_{B,A_s},B)\leq\varepsilon\right)
\leq K\varepsilon^2,\qquad 0<\varepsilon\leq1,
\qquad K=e^{C_kt^3}\geq1.
\]
Consequently, for $y\geq1$, substituting $\varepsilon=C_*/y\leq1$ yields
\[
\PP(X>y)
\leq\PP\left(\EE_{\cF_{A_s}}Z(\fB_{B,A_s},B)\leq C_*/y\right)
\leq K(C_*/y)^2\leq Ky^{-2}.
\]
Combining this with the trivial bound by one gives
\begin{equation}\label{eq:294}
S(y):=\PP(X>y)\leq\min\{1,Ky^{-2}\},\qquad y>0.
\end{equation}

Recall that $\beta=\PP(\fB_{B,A_s}\in\cA)$ is the probability of $\cA$ under the reference law. If $\beta=0$, absolute continuity proves the assertion. Otherwise, to bound the integral of $X$ over this event, split its tail at $y_0=\sqrt{K/\beta}$, so that $Ky_0^{-2}=\beta$. Below $y_0$ we use the bound $\beta$ on the probability of the event; above $y_0$ we use the density-tail bound \eqref{eq:294}. Tonelli's theorem gives
\begin{align}\label{eq:292}
\PP(\fL^{\mathbf a}\in\cA)
&=\EE\left[X\ind(\fB_{B,A_s}\in\cA)\right]\nonumber\\
&=\int_0^\infty\PP(X>y,\,\fB_{B,A_s}\in\cA)\,dy\nonumber\\
&\leq\int_0^{y_0}\beta\,dy+\int_{y_0}^\infty Ky^{-2}\,dy\nonumber\\
&=\beta y_0+\frac{K}{y_0}=2\sqrt{K\beta}.
\end{align}
This proves the lemma.
\par}
\end{proof}
In the following lemma, we a moderate deviation estimates for $\fB_{[\mathbf{a},\mathbf{a}+t],[\mathbf{a}-s,\mathbf{a}+s]}$ which shall subsequently be used in conjunction with the above comparison result to obtain Proposition \ref{prop:25}.
\begin{lemma}
  \label{lem:85}
Fix $k=1$ and take $\mathbf a,t,s$ as in Lemma \ref{lem:100}. Then for any $x\in \{\mathbf{a},\mathbf{a}+t\}$, we have
  \begin{align}
    \label{eq:475}
    &\PP(\fB_{[\mathbf{a},\mathbf{a}+t],[\mathbf{a}-s,\mathbf{a}+s],1}(x)+x^2\geq r)\leq Ce^{-cr^{3/2}},\\ \label{eq:4751} 
    &\PP(\fB_{[\mathbf{a},\mathbf{a}+t],[\mathbf{a}-s,\mathbf{a}+s],2}(x)+x^2\leq -r)\leq Ce^{-cr^{3/2}}.
  \end{align}
\end{lemma}
\begin{proof}

{
First, by monotonicity argument, it can be shown that $\fB_{[\mathbf{a},\mathbf{a}+t],[\mathbf{a}-s,\mathbf{a}+s]}$ is stochastically dominated by $\cP$. Thus, by using Lemma \ref{lem:74} with $k=1$, we immediately obtain the first inequality.
To obtain the second inequality, we first note that $\fB_{[\mathbf{a},\mathbf{a}+t],[\mathbf{a}-s,\mathbf{a}+s],k}=\cP_k$ for all $k\geq 2$ and then apply Lemma \ref{lem:74} with $k=2$.
\par}
\end{proof}

\begin{proof}[Proof of Proposition \ref{prop:25}]
  Apply Lemmas \ref{lem:100} and \ref{lem:85} with their common fixed choice $s=s(t)$. Now, by using \eqref{eq:475} in Lemma \ref{lem:85} along with Lemma \ref{lem:100}, we immediately obtain that for $x\in \{\mathbf{a},\mathbf{a}+t\}$, for some constant $c'$, we have
  \begin{equation}
    \label{eq:477}
    \PP(\fL^{\mathbf{a}}_1(x)+x^2\geq r)\leq e^{c't^3}e^{-cr^{3/2}}.
  \end{equation}
  Now, to bound $\PP(\fL^{\mathbf{a}}_1(x)+x^2\leq -r)$, we first note that by the above reasoning applied to \eqref{eq:4751}, we have for $x\in \{\mathbf{a},\mathbf{a}+t\}$ and for some constant $c'$,
  \begin{equation}
    \label{eq:478}
    \PP(\fL^{\mathbf{a}}_2(x)+x^2\leq -r)\leq e^{c't^3}e^{-cr^{3/2}}.
  \end{equation}
  However, we a.s.\ also have $\fL^{\mathbf{a}}_1(x)\geq \fL^{\mathbf{a}}_2(x)$ for $x\in \{\mathbf{a},\mathbf{a}+t\}$. This shows that $\PP(\fL^{\mathbf{a}}_1(x)+x^2\leq -r)\leq e^{c't^3}e^{-cr^{3/2}}$ and completes the proof.
\end{proof}

\section{Appendix 3: Tail bound on the number of near maximisers for BLPP weight profiles}
\label{sec:app-peaks}

The aim of this section is to establish Proposition \ref{lem:31}, the result controlling the number of peaks of the routed distance profile $Z_{\0}^{\n}$ on the line $\{m\}_\RR$. Note that we allow $m$ to be much smaller than $n$, and thus will have to ensure that all the estimates also hold for this regime. Our first goal is to establish the following lemma, which will later allow us to only count peaks for $Z_{\0}^{\n}$ in $\{m\}_{[m- (1+m)^{2/3}n^\delta, m+(1+m)^{2/3}n^\delta]}$ instead of the larger space $\{m\}_\RR$.
\begin{lemma}
  \label{lem:121}
There exist constants $g,C,c>0$ such that for all $\delta>0$ and all $0\leq m\leq n$, we have
  \begin{equation}
    \label{eq:601}
    \PP(\sup_{|x|\geq (1+m)^{2/3}n^{\delta}}Z_{\0}^{\n}(m+x,m)-Z_{\0}^{\n}(m,m) \geq -g(1+m)^{1/3}n^{\delta})\leq Ce^{-cn^{3\delta/8}}.
  \end{equation}
\end{lemma}
For the above lemma, by symmetry, we first note that it suffices to assume $m\leq n/2$ and indeed, we shall do this for the remainder for this section. %
For the proof of the above result, we shall first require a few easy preliminary results regarding passage times in BLPP which we now introduce; these estimates are fairly standard (see \cite[Propositions 2.28, 2.30]{Ham22} for similar estimates), and thus we shall not provide detailed proofs.
\subsection{Preliminary BLPP estimates}
\label{sec:prel-blpp-estim}
First, define $H_{x,k}\colon [-k,x]\rightarrow \RR$ by
\begin{equation}
  \label{eq:611}
    H_{x,k}= 2k+x- 2\sqrt{k(k+x)}.
\end{equation}
By a Taylor expansion argument, the following is easy to obtain.
\begin{lemma}
  \label{lem:127}
  For a fixed $k\geq 0$, $H_{x,k}$ is convex, increasing for $x>0$, decreasing for $x<0$. Also, there exists a constant $c\in (0,1)$ such that for all $|x|\leq ck$, we have $H_{x,k}\geq x^2(8(k+1))^{-1}$ and by convexity, for $|x|>ck$, we have $H_{x,k}\geq c|x|/8$.
\end{lemma}
The following is a simple consequence of the above.
\begin{lemma}
  \label{lem:128}
 There is a constant $g'>0$ such that for all $\delta>0$, $0\leq m\leq n$ and $|x|\geq (1+m)^{2/3}n^\delta$, we have $H_{x,m}\geq g'(1+m)^{1/3}n^\delta$.
\end{lemma}
Throughout the argument, we shall often use the following result, which we sketch a proof of. %
\begin{lemma}
  \label{lem:1221}
There exist constants $C,c,C',c'$ such that for all $\delta>0$, all $n$ and for all $m\in [\![0,n]\!]$, we have
  \begin{align}
    \label{eq:608}
    &\PP(\exists x: |x|\geq (1+m)^{2/3}n^\delta, T_{\0}^{(m+x,m)}-x\geq T_{\0}^{(m,m)} - H_{x,m}/2)\nonumber\\
    &\leq C\exp(-c\min_{|x|\geq (1+m)^{2/3}n^\delta}(\frac{H_{x,m}}{m^{-1/6}\sqrt{m+x}})^{3/2})\nonumber\\
    &\leq C'e^{-c'n^{3\delta/4}}.
  \end{align}
\end{lemma}
\begin{proof}[Proof sketch]
  Write the set $[-(1+m)^{2/3}n^\delta, (1+m)^{2/3}n^\delta]^c$ as the union of countably many intervals $[a_i,b_i]$ of length $1$ each. For some constants $C,c$, the probability in question can be bounded above by
  \begin{align}
    \label{eq:610}
    &\PP(T_{\0}^{(m,m)} -2m\leq -H_{(1+m)^{2/3}n^\delta,m}/4)+  \sum_i (\PP(T_{\0}^{(m+b_i,m)}-2m\sqrt{m(m+a_i)}\geq \min(H_{a_i,m}, H_{b_i,m})/4)\nonumber\\
    &\leq Ce^{-cn^{3\delta/2}}+ \sum_iC\exp( -c(H_{a_i,m}/(m^{-1/6}\sqrt{m+a_i}))^{3/2}),
  \end{align}
  where in the last line, the first term is obtained by using Proposition \ref{prop:47} and Lemma \ref{lem:128} while the second term is obtained just by Proposition \ref{prop:47} along with the fact that $b_i-a_i=1$. Since Lemma \ref{lem:127} yields that $H_{a_i,m}\geq \frac{a_i^2}{8(m+1)}\wedge (c|a_i|/8)$ for some constant $c$, it can be checked that, irrespective of the precise value of $m$, the sum in \eqref{eq:610} is finite and is at most $C'e^{-c'n^{3\delta/4}}$ for some constants $C',c'$.
\end{proof}
In order to handle the case of small values of $m$ in Lemma \ref{lem:121}, we shall require a few additional BLPP estimates. The following estimate on the diffusive fluctuations of BLPP weight profiles, which can be obtained by a comparison to Brownian motion (see \cite[Theorem 3.11]{CHH23}), will be useful for us.
\begin{lemma}
  \label{lem:125}
There exist constants $C,c,C',c'$ such that for all $\delta>0$, all $n$ and all $m$ satisfying $m\leq n^{1-3\delta/2}$, we have
  \begin{align}
    \label{eq:612}
        &\PP(\exists x: |x|\in [(1+m)^{2/3}n^\delta, n^{2/3}], |T_{(m+x,m)}^{\n}+x- T_{(m,m)}^{\n}|\geq H_{x,m}/4)\nonumber\\
    &\leq C\exp(-c \min_{|x|\in [(1+m)^{2/3}n^\delta,n^{2/3}]}(\frac{H_{x,m}}{\sqrt{|x|}})^2)\leq C'e^{-c'n^{\delta}}.
  \end{align}
\end{lemma}

Finally, for the case when $m$ is much smaller than $n$, we shall use the following deviation estimate which can be obtained using Proposition \ref{prop:47} similarly to the proof of Lemma \ref{lem:1221}.
\begin{lemma}
  \label{lem:118}
There exist positive constants $C,c,C',c'$ such that for all $\delta>0$, all $n$ and all $0\leq m\leq n/2$, we have
  \begin{align}
    \label{eq:598}
    \PP(\exists x\in [-m,n-m]: |x|\geq n^{2/3}, T_{(m+x,m)}^{\n}-x \geq T_{(m,m)}^{\n}-H_{x,m})&\leq C\exp(-c \min_{|x|\geq n^{2/3}} (\frac{H_{x,m}}{n^{1/3}})^{3/2})\nonumber\\
    &\leq C'e^{-c'(n/(1+m))^{3/2}}.%
  \end{align}
\end{lemma}

\subsection{Proof of Lemma \ref{lem:121}}
\label{sec:proof-lemma-refl}
The first task is to handle the case when $m$ is large, and for this, we have the following result.
\begin{lemma}
  \label{lem:124}
 There exist constants $C,c,g>0$ such that for all $\delta>0$, all $n$ and all $m$ satisfying $ n/2\geq m\geq n^{1-3\delta/4}$, we have
 \begin{equation}
    \label{eq:603}
    \PP(\sup_{|x|\geq (1+m)^{2/3}n^{\delta}}Z_{\0}^{\n}(m+x,m)-Z_{\0}^{\n}(m,m)\geq - g(1+m)^{1/3}n^\delta)\leq Ce^{-cn^{3\delta/8}}.
  \end{equation}  
\end{lemma}
\begin{proof}
Since $m\leq n/2$, we have $n-m\geq n/2$. Now, by two applications of Lemma \ref{lem:1221} and Lemma \ref{lem:128}, for some constant $C,c,g>0$, we have the following inequalities.
  \begin{align}
    \label{eq:602}
    &\PP(\exists x: |x|\geq (1+m)^{2/3}n^\delta, T_{\0}^{(m+x,m)}- x \geq T_{\0}^{(m,m)} - g(1+m)^{1/3}n^{\delta}/2) \leq Ce^{-cn^{3\delta/4}},\nonumber\\
    &\PP(\exists x: |x|\geq n^{2/3+\delta/2}, T_{(m+x,m)}^{\n}+x \geq T_{(m,m)}^{\n} - gn^{1/3+\delta/2}/2) \leq Ce^{-cn^{3\delta/8}}.
  \end{align}
  Now, since  $(1+m)^{2/3}n^\delta\geq n^{2/3+\delta/2}$ and since $m\leq n$, a simple union bound yields the desired statement.  
\end{proof}
Now, it remains to handle small values of $m$. For this, we first have the following result.
\begin{lemma}
  \label{lem:126}
 There exist constants $g,C,c>0$ such that for all $\delta>0$, all $n$ and all $m$ satisfying $0\leq m< n^{1-3\delta/2}$, we have
  \begin{equation}
  \label{eq:609}
      \PP(\exists x: |x|\in [(1+m)^{2/3}n^\delta, n^{2/3}]: Z_{\0}^{\n}(m+x,m)-Z_{\0}^{\n}(m,m)\geq -g (1+m)^{1/3}n^{\delta})\leq Ce^{-cn^{3\delta/4}}.
\end{equation}
\end{lemma}
\begin{proof}
    As a result of Lemma \ref{lem:1221}, we know that
\begin{equation}
  \label{eq:606}
  \PP(\exists x : |x|\geq (1+m)^{2/3}n^{\delta}: T_{\0}^{(x,m)}-x \geq T_{\0}^{(m,m)}-H_{x,m}/2) \leq Ce^{-cn^{3\delta/4}}. 
\end{equation}
By combining this with Lemma \ref{lem:125} and then using Lemma \ref{lem:128}, we immediately obtain the desired result.
\end{proof}
Further, we have the following result.
\begin{lemma}
  \label{lem:123}
There exist constants $g,C,c>0$ such that for all $\delta>0$, all $n$ and all $m$ satisfying $0\leq m< n^{1-3\delta/2}$, we have
    \begin{equation}
    \label{eq:605}
    \PP(\exists x: |x|\geq n^{2/3}: Z_{\0}^{\n}(m+x,m)-Z_{\0}^{\n}(m,m)\geq -g (1+m)^{1/3}n^{\delta})\leq Ce^{-cn^{3\delta/4}}.
  \end{equation}
\end{lemma}
\begin{proof}
  By assumption, we have $n/m\geq n^{3\delta/2}$ and thus $(n/m)^{3/2}\geq n^{9\delta/4}\geq n^{3\delta/4}$. We now combine \eqref{eq:606} and Lemma \ref{lem:118} to obtain the desired result.
\end{proof}

\begin{proof}[Proof of Lemma \ref{lem:121}]
  The proof is completed by combining Lemmas \ref{lem:124}, \ref{lem:126}, \ref{lem:123}.
\end{proof}

\subsection{Proof of Proposition \ref{lem:31}}
\label{sec:proof-prop-refl}
The goal now is to use Lemma \ref{lem:121} along with a Brownianity argument to prove Proposition \ref{lem:31}. For this, we shall need a result on the Brownianity of routed distance profiles at mesoscopic scales, and for this, we introduce some notation. For $m\leq n/2$ and $r>0$, we shall consider the process $\frakf^m_{r}\colon [-r (1+m)^{2/3}, r (1+m)^{2/3}]\rightarrow \RR$ defined by
\begin{equation}
  \label{eq:672}
  \frakf^m_{r}(x)=Z_{\0}^{\n,\bullet}(m+x,m) -Z_{\0}^{\n,\bullet}(m,m),
\end{equation}
and we note that the above is well-defined as long as $m-r (1+m)^{2/3}\geq 0$.
\begin{proposition}[{\cite[Theorem 1.2]{GH20}}]
  \label{prop:52}
  Let $B_{r}^m$ denote a Brownian motion of diffusivity $2$ on $[-r (1+m)^{2/3}, r (1+m)^{2/3}]$. Then there exist positive constants $g_1,g_2,g_3,g_4,g_5$ and $m_0>0$ such that for all $m_0\leq m\leq n/2$, $r$ satisfying $r\leq m^{1/3}/2$, and all measurable sets $A$ satisfying
  \begin{equation}
    \label{eq:613}
   e^{-g_1m^{1/12}}\leq  \PP(B_{r}^m\in A) \leq g_2\wedge e^{-g_3r^{12}},
 \end{equation}
 we have
    \begin{equation}
    \label{eq:615}
    \PP(\frakf^m_{r}\in A)\leq  g_4r^6\PP(B_{r}^m\in A)\exp( g_5r^7 (\log \PP(B_{r}^m\in A)^{-1})^{5/6}).
  \end{equation}
\end{proposition}
We note that in the source \cite{GH20}, the above result is stated for the case when $m=\Theta(n)$, but the same proof generalises to yield the above result. Also, we note that that Proposition \ref{prop:24} can be used to obtain a stronger version of Brownianity of $Z_{\0}^{\n,\bullet}$ than the above result; however, Proposition \ref{prop:52} will be entirely sufficient for our application.

Now, for any set $A\subseteq \RR$ and a real valued function $f$ defined on $I$, let $\nearmax^\alpha(f)$ be the largest possible size of a set $S\subseteq I$ with the following properties:
  \begin{enumerate}
  \item $|x-y|\geq \alpha^2$ for all $x,y\in S$
  \item $\max_{x\in I}f(x) - f(y)\leq \alpha$ for all $y\in S$.
  \end{enumerate}
Given the literature on the Brownianity of BLPP weight profiles, the following result is easy to obtain and we now sketch a proof.

\begin{lemma}
  \label{lem:129}
  Fix $\delta>0$. For all $m\geq n^{600\delta}$, all $\alpha \leq n^{\delta/2}m^{1/3}$, we have the following.
  \begin{equation}
    \label{eq:614}
    \PP(\nearmax^{\alpha}(\frakf_{n^\delta}^m)\geq n^{48\delta})\leq Ce^{-cn^{48\delta}}
  \end{equation}
\end{lemma}
\begin{proof}
  The proof proceeds by using Proposition \ref{prop:52} along with a Brownian computation. First, by \cite[Proposition 2.5]{CHH23}, there exist constants $D,d$ such that for all $k\geq 1$ and all $\alpha\leq \sqrt{r}(1+m)^{1/3}$, we have
  \begin{equation}
    \label{eq:616}
    \PP(\nearmax^{\alpha}(B_{r}^m)\geq k)\leq De^{-dk}.
  \end{equation}
 Now set $r=n^\delta$ and $k=n^{48\delta}$. Choose a constant $a\in(0,d)$ and write $q=e^{-ak}$. By \eqref{eq:616}, the Brownian probability of the near-maximum event is at most $q$ for all sufficiently large $n$. Since Brownian motion law is atomless, we may enlarge this event to a measurable set $A$ with Brownian probability exactly $q$. The probability cutoffs in \eqref{eq:613} hold because
\[
 r^{12}=n^{12\delta}=o(k),\qquad
 k=n^{48\delta}=o(m^{1/12}),\qquad m^{1/12}\geq n^{50\delta}.
\]
Also, $r\leq m^{1/3}/2$ for all sufficiently large $n$. Applying Proposition \ref{prop:52} to $A$, we obtain
\[
\begin{aligned}
 \PP(\nearmax^\alpha(\frakf_{n^\delta}^m)\geq k)
 &\leq g_4r^6q\exp\bigl(g_5r^7(\log q^{-1})^{5/6}\bigr)\\
 &\leq Cn^{6\delta}\exp\bigl(-an^{48\delta}+Cn^{47\delta}\bigr)
 \leq C'e^{-c'n^{48\delta}},
\end{aligned}
\]
which proves the lemma.
\end{proof}

Note that the above result concerns $\frakf_{n^\delta}^m$ which is defined in terms of $Z_{\0}^{\n,\bullet}$. However, our final goal is to prove Proposition \ref{lem:31} which is a statement about the behaviour of $Z_{\0}^{\n}$. We now present a lemma which will allow us to reduce the analysis of the latter to that of the former.
\begin{lemma}
  \label{lem:137}
Fix $\delta>0$. For some constants $C,c$, there is an event $E_n$ with probability at least $1-Ce^{-cn^{3\delta/8}}$ on which for any $m\geq n^{600\delta}$ and for any $|x|\leq (1+m)^{2/3}n^{\delta}+1$ satisfying
  \begin{equation}
    \label{eq:674}
    T_{\0}^{\n}-Z_{\0}^{\n}(m+x,m)\leq n^{\delta/2},
  \end{equation}
  there must exist an $x'$ satisfying $x'-x\in (0, 2n^{\delta})$ with the property that
  \begin{equation}
    \label{eq:675}
    T_{\0}^{\n}-Z_{\0}^{\n,\bullet}(m+x',m)\leq n^{\delta/2}.
  \end{equation}
\end{lemma}
\begin{proof}
   By applying Lemma \ref{lem:121} for the degenerate case $m=0$ along with a Brownian scaling argument, we know that there are constants $C,c$ such that for all $m\geq n^{600\delta}$ and for any $y$ satisfying $-m/2\leq y-n^{\delta}\leq y\leq m/2$,
  \begin{equation}
    \label{eq:677}
    \PP(T_{\0}^{(m+y,m)}-Z_{\0}^{(m+y,m)}(m+y-n^{\delta},m)\geq n^{\delta/2})\geq 1-Ce^{-cn^{3\delta/8}}.
  \end{equation}
 Now, we define the event $E_n$ by
  \begin{equation}
    \label{eq:676}
    E_n=\{ T_{\0}^{(m+y,m)}-Z_{\0}^{(m+y,m)}(m+y-n^{\delta},m)\geq n^{\delta/2} \textrm{ for all } m\geq n^{600\delta}, |y|\leq ((1+m)^{2/3}+3)n^{\delta}, y\in \ZZ\}.
  \end{equation}
 By combining \eqref{eq:677} with a union bound, for constants $C,c$, we obtain
  \begin{equation}
    \label{eq:678}
    \PP(E_n)\geq 1-Ce^{-cn^{3\delta/8}}.
  \end{equation}
  Now, on the event $E_n$, suppose $x,m$ are such that the condition in \eqref{eq:674} in satisfied. This means that there must exist a staircase $\xi\colon \0\rightarrow \n$ such that $(m+x,m)\in \xi$ and additionally,
  \begin{equation}
    \label{eq:679}
    T_{\0}^{\n}-\wgt(\xi)\leq n^{\delta/2}.
  \end{equation}
  We claim that on the event $E_n$, we must in fact have
  \begin{equation}
    \label{eq:680}
    \xi(m)\in [x,x+2n^{\delta}).
  \end{equation}
  Note that this would complete the proof since we could simply take $x'=\xi(m)-m$ since we do know that $Z_{\0}^{\n,\bullet}(\xi(m),m)\geq \wgt(\xi) \geq T_{\0}^{\n}-n^{\delta/2}$. Now, let the point $\widetilde x\in \ZZ$ be the largest point additionally satisfying $x< \widetilde x < x+ 2n^{\delta}$.

  Note that since $x\in \xi$, we always have $\xi(m)\geq x$. Thus, with the aim of obtaining a contradiction to \eqref{eq:679}, suppose that on the event $E_n$, we have $\xi(m)> x+ 2n^{\delta}$. Since $(m+\widetilde x,m)\in \xi$ and since we have assumed \eqref{eq:679}, we must have $Z_{\0}^{\n}( m+\widetilde x,m)\geq T_{\0}^{\n}-n^{\delta/2}$. We now let $\widetilde\xi\subseteq \xi$ be the staircase satisfying $\widetilde \xi\colon \0\rightarrow (m+\widetilde x,m)$. Then we must have
  \begin{equation}
    \label{eq:683}
    \wgt(\widetilde \xi)\geq T_{\0}^{(m+\widetilde x,m)}- n^{\delta/2}.
  \end{equation}
  However, by the definition of $\widetilde x$, we have $\widetilde x - n^{\delta}\geq x+n^{\delta}-1\geq x$ and thus since $(m+x,m)\in \widetilde \xi$, we also have $(m+\widetilde x - n^{\delta},m)\in \widetilde \xi$. As a consequence of this and \eqref{eq:683}, we must have
  \begin{equation}
    \label{eq:684}
    Z_{\0}^{(m+\widetilde x,m)}(m+\widetilde x-n^{\delta},m)\geq T_{\0}^{(m+\widetilde x,m)}-n^{\delta/2},
  \end{equation}
  but this contradicts the definition of the event $E_n$ from \eqref{eq:676} since $\widetilde x \in \ZZ$ and $|\widetilde x|< |x| +2n^{\delta}< ((1+m)^{2/3}+3)n^{\delta}$. Thus, our assumption that $\xi(m)> x+2n^{\delta}$ must have been incorrect. This completes the proof.
\end{proof}

We are now ready to complete the proof of Proposition \ref{lem:31}.
\begin{proof}[Proof of Proposition \ref{lem:31}]
By Lemma \ref{lem:121}, for all large enough $n$, there is an event $\cE_{m,n}$ satisfying for some constants $C,c$,
  \begin{equation}
    \label{eq:622}
    \PP(\cE_{m,n})\geq 1-Ce^{-cn^{3\delta/8}},
  \end{equation}
  on which we have
  \begin{equation}
    \label{eq:619}
    \sup_{|x|\geq (1+m)^{2/3}n^{\delta}/2}Z_{\0}^{\n}(m+x,m)-Z_{\0}^{\n}(m,m)< - n^{\delta/2}.
  \end{equation}
  As a result of the above, on the event $\cE_{m,n}$, for all $n$ large enough, we have
  \begin{equation}
    \label{eq:686}
    |\peak(n^{\delta/2})\cap \{m\}_{\RR}|= |\peak(n^{\delta/2})\cap \{m\}_{[m-(1+m)^{2/3}n^\delta/2, m+(1+m)^{2/3}n^\delta/2]}|.
  \end{equation}
  We now consider two separate cases-- first, we handle the case when $m$ is small, that is $m< n^{600 \delta}$. Here, we make do with a crude bound. Indeed, for any $m$, on the event $\cE_{m,n}$, by using \eqref{eq:686}, %
  we immediately have the deterministic bound
  \begin{equation}
    \label{eq:627}
    |\peak(n^{\delta/2})\cap \{m\}_{\RR}|\leq  n^{\delta}(1+m)^{2/3}+1.
  \end{equation}
  Note that for $m\leq n^{600\delta}$ and for all large enough $n$, we have $n^{\delta}(1+m)^{2/3}+1\leq n^{600\delta}$, and as a result, for all $m\leq n^{600\delta}$, we obtain
  \begin{equation}
    \label{eq:628}
    \PP(|\peak(n^{\delta/2})\cap \{m\}_{\RR}|\geq n^{600\delta})\leq \PP(\cE_{m,n}^c)\leq Ce^{-cn^{3\delta/8}}.
  \end{equation}

  Thus, it now remains to consider values of $m$ satisfying $m\geq n^{600\delta}$ and for this, we shall need to work with the profile $Z_{\0}^{\n,\bullet}$. To begin, for $\alpha\in \RR$, we define $\peak^{\bullet}(\alpha)$ by replacing the routed distance profile $Z_{\0}^{\n}$ in the definition of $\peak(\alpha)$ by the profile $Z_{\0}^{\n,\bullet}$. The utility of this is that on the event $E_n$ from Lemma \ref{lem:137}, for all $m\geq n^{600\delta}$, we must have
  \begin{align}
    \label{eq:687}
    &|\peak(n^{\delta/2})\cap \{m\}_{[m-(1+m)^{2/3}n^\delta/2, m+(1+m)^{2/3}n^\delta/2]}|\nonumber\\
    &\leq (2n^{\delta}+1)|\peak^{\bullet}(n^{\delta/2})\cap \{m\}_{[m-(1+m)^{2/3}n^\delta, m+(1+m)^{2/3}n^\delta]}|.
  \end{align}
  Indeed, Lemma \ref{lem:137} implies that on the event $E_n$, for any $(i,m)$ which is an element of the set on the left hand side above, there must exist an $i'$ such that $i\leq i'\leq 2n^{\delta}+1$ such that $(i',m)$ is an element of the set on the right hand side above.

  We now analyse the set $\peak^{\bullet}(n^{\delta/2})$. First, note that deterministically, for any $\alpha, m$, we have the inequality %
  \begin{equation}
    \label{eq:358}
    |\peak^{\bullet}(\alpha)\cap \{m\}_{[m-(1+m)^{2/3}n^\delta, m+(1+m)^{2/3}n^\delta]}|\leq (2\alpha^2)\nearmax^{\alpha}(\frakf_{n^\delta}^m),
  \end{equation}
  where, in the above, we used the profile $\frakf_{n^\delta}^m$ as defined in \eqref{eq:672}.
  We shall use the above with $\alpha = n^{\delta/2}$.
 Now, by applying Lemma \ref{lem:129}, we know that for all such $m$,
  \begin{equation}
    \label{eq:618}
    \PP(\nearmax^{n^{\delta/2}}(\frakf_{n^\delta}^m) \geq n^{48\delta}) \leq Ce^{-cn^{48\delta}}.
  \end{equation}
  As a result, for all $n$ large enough and all $m\geq n^{600\delta}$, for some constants $C,c,C',c'$, we have
  \begin{align}
    \label{eq:621}
    &\PP(|\peak(n^{\delta/2})\cap \{m\}_{\RR}\geq n^{51\delta}) \nonumber\\
    &\leq \PP(\cE_{m,n}^c)+ \PP(E_n^c)+ \PP((2n^{\delta}+1)(2n^\delta)\nearmax^{n^{\delta/2}}(\frakf_{n^\delta}^m)\geq n^{51\delta})\nonumber\\
                                                             &\leq Ce^{-cn^{3\delta/8}}+ \PP(\nearmax^{n^{\delta/2}}(\frakf_{n^\delta}^m) \geq n^{48\delta})\nonumber\\
    &\leq C'e^{-cn^{3\delta/8}}.
  \end{align}
  The second term in the first inequality above is obtained by using \eqref{eq:687} and \eqref{eq:358}. To obtain the second inequality, we have used \eqref{eq:622} and Lemma \ref{lem:137}. Finally, the last inequality is obtained by using \eqref{eq:618}.

The proof is now completed by replacing $\delta$ by $2\delta$ and using \eqref{eq:621}, \eqref{eq:628} along with a union bound over all possible values of $m$.

\end{proof}

\section{Appendix 4: A twin peaks estimate for BLPP routed weight profiles}
\label{sec:twin-peaks}
The goal of this section is to prove the twin peaks result-- Proposition \ref{prop:26}.
{
We shall instead first prove the following parameterised version, in which the deficit and the spatial separation can be chosen independently, and then deduce Proposition \ref{prop:26}.
\begin{proposition}[Parameterised twin peaks]\label{prop:tp-parameters}
Fix $\beta'\in(0,1/2)$, $D>0$ and $\varepsilon\in(0,1)$. There exist constants $C,c>0$ and $n_0$, depending only on these parameters, such that the following holds for $n\geq n_0$ and $m\in[\![\beta'n,(1-\beta')n]\!]$. For $A,w>0$, set $\sigma=A/\sqrt w$ and suppose that
\begin{equation}\label{eq:tp-parameter-range}
 n^{-D}\leq \frac{w}{n^{2/3}}\leq1,
 \qquad n^{-D}\leq\sigma\leq1.
\end{equation}
Define
\begin{equation}\label{eq:tp-parameter-event}
 \TP_{A,w,m}=\{\exists x:\ |x-\Gamma_{\0}^{\n}(m)|\geq w,
 \ |T_{\0}^{\n}-Z_{\0}^{\n,\bullet}(x,m)|\leq A\}.
\end{equation}
Then
\begin{equation}\label{eq:tp-parameter-bound}
 \PP(\TP_{A,w,m})\leq C\sigma^{1-\varepsilon}+Ce^{-cn^{1/12}}.
\end{equation}
\end{proposition}
\par}
We shall first prove certain preliminary results and then combine them at the end to obtain the desired result. In the setting of Proposition \ref{prop:tp-parameters}, we note that by symmetry, after decreasing $\beta'$ if necessary, it suffices to work with $\beta'n\leq m\leq n/2$, and we assume this throughout this section. We shall work with $f(\sigma)=\lceil C\log^{1/3}(e/\sigma)\rceil$, with $C$ sufficiently large depending on $\beta'$, and will often consider the transversal fluctuation event $\TF_{\sigma,m}$ defined by

{
\begin{equation}
  \label{eq:340}
  \TF_{\sigma,m}=\{|\Gamma_{\0}^{\n}(m)- m|\geq 2f(\sigma)m^{2/3}\}.
\end{equation}
\par}

By the transversal fluctuation estimate for BLPP (see Proposition \ref{prop:53}), we have the following result. Here $f(\sigma)=O_D((\log n)^{1/3})$ by \eqref{eq:tp-parameter-range}, so the transversal estimate applies uniformly for all sufficiently large $n$.
\begin{lemma}
  \label{lem:89}
  There exists a constant $C$ such that for all $m,n,A,w$ as before, we have $\PP(\TF_{\sigma,m})\leq C\sigma^2$.
\end{lemma}

Now, with $k\in[\![1,2f(\sigma)]\!]$, we divide the interval $[-f(\sigma),f(\sigma)]$ into $2f(\sigma)$ many  intervals $J_k$ of length $1$ each. Indeed, we shall often work with the intervals

{
\begin{align}
  \label{eq:343}
  J_k=&[-f(\sigma)+ (k-1), -f(\sigma)+k],\nonumber\\
  \underline{J}_k=&[-f(\sigma)+ (k-2), -f(\sigma)+(k+1)].
\end{align}
\par}

and we note that $J_k$ is the middle interval if we divide $\underline{J}_k$ into three intervals of equal length.
 Define the event $\TP_{A,w,m}^k$ to be

{

{
\begin{align}
    \label{eq:303}
 &\{\exists x\in \underline{J}_k:
|x-\argmax_{y\in \underline{J}_k}Z_{\0}^{\n,\bullet}(m+2m^{2/3}y,m)|\geq w/(2m^{2/3}),\nonumber\\
    &\qquad |\max_{y\in \underline{J}_k}Z_{\0}^{\n,\bullet}(m+2m^{2/3}y,m)-Z_{\0}^{\n,\bullet}(m+2m^{2/3}x,m)|\leq
    A\}\nonumber\\
    &\bigcap \{\argmax_{y\in \underline{J}_k}Z_{\0}^{\n,\bullet}(m+2m^{2/3}y,m)\in J_k\}.
\end{align}
\par}

\par}

With the above definition, we immediately have the following result.
\begin{lemma}
  \label{lem:90}
 For all $m,n,A,w$ as before, we have the inclusion

{

{
\begin{equation}
  \label{eq:304}
  \TP_{A,w,m}\cap (\TF_{\sigma,m})^c\subseteq (\bigcup_{k=1}^{2f(\sigma)}\TP_{A,w,m}^k) \cup \bigcup_{k=1}^{2f(\sigma)} \{|\max_{y\in J_k}Z_{\0}^{\n,\bullet}(m+2m^{2/3}y,m)-\max_{y\in J_k^c}Z_{\0}^{\n,\bullet}(m+2m^{2/3}y,m)|\leq A\}.
\end{equation}
\par}

\par}

\end{lemma}
\begin{proof}
Recall that $T_{\0}^{\n}=\max_{x}Z_{\0}^{\n,\bullet}(x,m)$ and that $\Gamma_{\0}^{\n}(m)=\argmax_{x}Z_{\0}^{\n,\bullet}(x,m)$. As a result, on the event $(\TF_{\sigma,m})^c$, there must exist a $k^*\in [\![1,2f(\sigma)]\!]$ for which $\Gamma_{\0}^{\n}(m)=\argmax_{x}Z_{\0}^{\n,\bullet}(x,m)\in m+2m^{2/3}J_{k^*}$, and thus $k^*$ must satisfy
  
{
\begin{equation}
    \label{eq:466}
    \frac{\Gamma_{\0}^{\n}(m)-m}{2m^{2/3}}=\argmax_{y\in J_{k^*}}Z_{\0}^{\n,\bullet}(m+2m^{2/3}y,m)= \argmax_{y\in \underline{J}_{k^*}}Z_{\0}^{\n,\bullet}(m+2m^{2/3}y,m).
  \end{equation}
}
  Now, on the event $\TP_{A,w,m}\cap (\TF_{\sigma,m})^c$, there must exist an $x^*$ such that $|x^*-\Gamma_{\0}^{\n}(m)|\geq w, |T_{\0}^{\n}-Z_{\0}^{\n,\bullet}(x^*,m)|\leq
 A$. Now, there are two cases-- if we have $x^*\in m+2m^{2/3}\underline{J}_{k^*}$, then the event $\TP_{A,w,m}^{k^*}$ must have occurred. If instead, we have $x^*\notin m+2m^{2/3}\underline{J}_{k^*}$, then the event

{
\begin{equation}
   \label{eq:464}
\{|\max_{y\in J_{k^*}}Z_{\0}^{\n,\bullet}(m+2m^{2/3}y,m)-\max_{y\in J_{k^*}^c}Z_{\0}^{\n,\bullet}(m+2m^{2/3}y,m)|\leq A\}
\end{equation}
\par}

must have occurred instead. This completes the proof.
\end{proof}

Now, we present a lemma in which we obtain an estimate on the second term on the right hand side of \eqref{eq:304}.
  \begin{lemma}
    \label{lem:92}
    There exists a constant $C$ such that for all $m,n,A,w$ as before and all $k\in [\![1,2f(\sigma)]\!]$, we have

{

{
\begin{equation}
      \label{eq:342}
      \PP(|\max_{y\in J_k}Z_{\0}^{\n,\bullet}(m+2m^{2/3}y,m)-\max_{y\in J_k^c}Z_{\0}^{\n,\bullet}(m+2m^{2/3}y,m)|\leq A)\leq C A m^{-1/3}(1+f(\sigma)^2).
    \end{equation}
\par}

\par}

  \end{lemma}

\begin{proof}
For the proof of this lemma, we shall require the $\fL^{\mathbf{a}}$ line ensembles from Appendix \ref{sec:brownian-regularity} and the result Proposition \ref{prop:24}. %
  Also, we define $a_k,b_k$ such that $J_k=[a_k,b_k]$. Now, consider the line ensembles $\cP^{\downarrow},\cP^{\uparrow}$ respectively consisting of $m+1$ and $n-m$ lines, with their top lines defined by

{
\begin{equation}
  \label{eq:308}
  \cP^{\downarrow}_1(x)=m^{-1/3}[T_\0^{(m+2m^{2/3}x,m)}-2m-2m^{2/3}x]
\end{equation}
}

and

{
\begin{equation}
  \label{eq:309}
  \cP^{\uparrow}_1(x)= (n-m-1)^{-1/3}[T_{(m+1-2(n-m-1)^{2/3}x,m+1)}^{\n}-2(n-m-1)-2(n-m-1)^{2/3}x].
\end{equation}
We now consider the corresponding $\fL$ line ensembles from Proposition \ref{prop:24}. That is, $\fL^\downarrow$ (resp.\ $\fL^\uparrow$) is defined by using the line ensemble $\cP^\downarrow$ (resp.\ $\cP^\uparrow$) with the parameters $t=1$ (resp.\ $t=m^{2/3}/(n-m-1)^{2/3}$), $k=1$, $\mathbf a=a_k$ (resp.\ $\mathbf a=-m^{2/3}b_k/(n-m-1)^{2/3}+1/(2(n-m-1)^{2/3})$). Here both lengths are bounded above and below by positive constants depending on $\beta'$. Propositions \ref{prop:24} and \ref{prop:25} apply in this range as well: in their proofs, take $s(t)=C_{\beta'}t$ sufficiently large and absorb $1+t^3$ into a constant multiple of $t^3$. The intervals lie in the required windows for large $n$, since $|a_k|,|b_k|\leq f(\sigma)$ and $m,n-m-1\asymp n$. By definition, we have
\begin{equation}
  \label{eq:311}
  Z_{\0}^{\n,\bullet}(m+2m^{2/3}x,m)=2n-1+m^{1/3}\cP^{\downarrow}_1(x)+(n-m-1)^{1/3}\cP^\uparrow_1\left(-\frac{m^{2/3}x}{(n-m-1)^{2/3}}+\frac{1}{2(n-m-1)^{2/3}}\right)
\end{equation}
\par}

and now, we define

{
\begin{equation}
  \label{eq:321}
  \widetilde{Z}_{\0}^{\n,\bullet}(m+2m^{2/3}x,m)=2n-1+m^{1/3}\fL^{\downarrow}_1(x)+(n-m-1)^{1/3}\fL^\uparrow_1\left(-\frac{m^{2/3}x}{(n-m-1)^{2/3}}+\frac{1}{2(n-m-1)^{2/3}}\right).
\end{equation}
\par}

{
For a constant $C_1$, we have
\par}

{

{
\begin{align}
\label{eq:324}
&\PP(|\max_{y\in J_k}Z_{\0}^{\n,\bullet}(m+2m^{2/3}y,m)-\max_{y\in J_k^c}Z_{\0}^{\n,\bullet}(m+2m^{2/3}y,m)|\leq A)\nonumber\\
&\leq C_1\PP(|\max_{y\in J_k}\widetilde Z_{\0}^{\n,\bullet}(m+2m^{2/3}y,m)-\max_{y\in J_k^c}\widetilde Z_{\0}^{\n,\bullet}(m+2m^{2/3}y,m)|\leq A).
\end{align}
\par}

where the constant $C_1$ follows from Proposition \ref{prop:24}(iii)--(iv), since both values of $t$ are bounded above and below by positive constants depending on $\beta'$.
\par}

{
Thus it suffices to show that
\par}

{

{
\begin{equation}
    \label{eq:320}
    \PP(|\max_{y\in J_k}\widetilde{Z}_{\0}^{\n,\bullet}(m+2ym^{2/3},m)-\max_{y\in J_k^c}\widetilde{Z}_{\0}^{\n,\bullet}(m+2ym^{2/3},m)|\leq A)\leq C A m^{-1/3}(1+f(\sigma)^2).
  \end{equation}
\par}

\par}

  We now analyse the process

{
\begin{equation}
  \label{eq:322}
  B(x)=\fL^{\downarrow}_1(x)+m^{-1/3}(n-m-1)^{1/3}\fL^\uparrow_1\left(-\frac{m^{2/3}x}{(n-m-1)^{2/3}}+\frac{1}{2(n-m-1)^{2/3}}\right).
\end{equation}
By the definition of $\fL^\downarrow,\fL^\uparrow$, we note the following properties:
\begin{enumerate}
\item The line ensembles $\fL^\downarrow,\fL^\uparrow$ are independent.
\item Conditional on its values outside $J_k$, $\fL^\downarrow_1\lvert_{J_k}$ is given by an independent Brownian bridge of rate $2$ interpolating between the endpoint values.
\item Conditional on its values outside $-\frac{m^{2/3}}{(n-m-1)^{2/3}}J_k+\frac{1}{2(n-m-1)^{2/3}}$, the first line of $\fL^\uparrow$ on this interval is given by an independent Brownian bridge of rate $2$ interpolating between the endpoint values.
\end{enumerate}
As a result of the above, conditional on $B\lvert_{J_k^c}$, the process $B\lvert_{J_k}$ is an independent Brownian bridge of diffusivity $4$ interpolating between $B(a_k),B(b_k)$. Now, to establish \eqref{eq:320}, we need only show that for some constant $C$, we have

{
\begin{equation}
  \label{eq:323}
  \PP(|\max_{y\in J_k}B(y)-\max_{y\in J_k^c}B(y)|\leq A m^{-1/3})\leq C A m^{-1/3}(1+f(\sigma)^2).
\end{equation}
\par}

\par}

Now, recall the following basic fact about Brownian motion-- with $B'$ being a Brownian motion of diffusivity $v$ on $[0,1]$, we know that for all $m\geq a\geq 0$,

{
\begin{equation}
  \label{eq:325}
  \PP(\max_{x\in [0,1]}B'(x)\geq m\lvert B'(1)=a)= e^{-2m(m-a)/v}.
\end{equation}
\par}

We shall work with $v=4$ since the Brownian motions involved in the definition of $B$ have diffusivity $4$. As a consequence of \eqref{eq:325}, we obtain that for any interval $J\subseteq \RR$ of length $L$, and for $a\geq 0$,
\begin{equation}
  \label{eq:327}
  \PP(\max_{x\in [0,1]}B'(x)\in J\lvert B'(1)=a)\leq L(1+a/2).
\end{equation}
Indeed, differentiating \eqref{eq:325} shows that the conditional density $f_a$ of the maximum satisfies, for $u\geq 0$,
\[
 f_a(a+u)=(u+a/2)e^{-u(u+a)/2}
 \leq u e^{-u^2/2}+a/2\leq 1+a/2.
\]
Integrating this density over $J$ gives \eqref{eq:327}.
To apply this bound to arbitrary endpoint heights $B(a_k),B(b_k)$, we subtract their minimum and, if necessary, reverse the bridge in time. The resulting bridge runs from $0$ to $a=|B(b_k)-B(a_k)|$; its maximum and the target interval are shifted by the same constant.

By using the above, we obtain that %

{
\begin{align}
  \label{eq:326}
  &\PP(|\max_{y\in J_k}B(y)-\max_{y\in J_k^c}B(y)|\leq Am^{-1/3}\big \lvert B\lvert_{J_k^c} )\nonumber\\
  &= \PP(\max_{y\in J_k}B(y)\in [\max_{y\in J_k^c}B(y)-Am^{-1/3},\max_{y\in J_k^c}B(y)+Am^{-1/3}] \big \lvert B\lvert_{J_k^c})\nonumber\\
  &\leq A m^{-1/3}\bigl(2+|B(b_k)-B(a_k)|\bigr).
\end{align}
\par}

{
We now take expectations in \eqref{eq:326}. To explain how Proposition \ref{prop:25} applies, recall that the two interval lengths used to construct $\fL^\downarrow,\fL^\uparrow$ are bounded above and below by positive constants depending on $\beta'$. Thus, at either endpoint $z$ of the corresponding interval, that proposition gives
\[
 \EE\bigl[|\fL_1(z)+z^2|\bigr]
 =\int_0^\infty\PP\bigl(|\fL_1(z)+z^2|\geq r\bigr)\,dr
 \leq C\int_0^\infty e^{-cr^{3/2}}\,dr\leq C,
\]
where $\fL$ denotes either of the two ensembles. In particular, $\EE[|\fL_1(z)|]\leq C+z^2$. For $x\in\{a_k,b_k\}$, both $x$ and the argument of $\fL^\uparrow_1$ in \eqref{eq:322} have absolute value at most $C(1+f(\sigma))$. Moreover, the coefficient $m^{-1/3}(n-m-1)^{1/3}$ in that equation is bounded, since $m,n-m-1\asymp n$. Applying the preceding bound to the two terms in \eqref{eq:322} therefore gives
\[
 \EE[|B(x)|]\leq C(1+f(\sigma)^2),\qquad x\in\{a_k,b_k\}.
\]
The triangle inequality now yields
\begin{equation}\label{eq:tp-endpoint-moment}
 \EE[2+|B(b_k)-B(a_k)|]\leq C(1+f(\sigma)^2).
\end{equation}
Taking expectations of the conditional probability in \eqref{eq:326}, we conclude that
\begin{align}\label{eq:328}
 &\PP(|\max_{y\in J_k}B(y)-\max_{y\in J_k^c}B(y)|\leq A m^{-1/3})\nonumber\\
 &\qquad\leq A m^{-1/3}\EE[2+|B(b_k)-B(a_k)|]
 \leq C A m^{-1/3}(1+f(\sigma)^2).
\end{align}
This completes the proof.
\par}
\end{proof}
The goal now is to obtain a corresponding estimate on the first term on the right hand side of \eqref{eq:304}, and for this, we need to obtain a $\sigma^{1-\varepsilon/2}+e^{-cn^{1/12}}$ estimate on each of the probabilities $\PP(\TP_{A,w,m}^k)$ for $k\in [\![1,2f(\sigma)]\!]$. For doing so, we shall use the following result from \cite{GH20}.

\begin{lemma}[{\cite[Corollary 2.13]{GH20}}]

  \label{lem:139}

{
Fix $\beta'\in(0,1/2)$. There exist positive constants $c,C$ such that for all $\beta'n\leq m\leq n/2$, $|R|\leq cm^{1/9}$, $\sigma\in(0,1)$, $\varepsilon\in(0,2]$, $n$ large enough, and with $M=\argmax_{y\in[R-3/2,R+3/2]}Z_{\0}^{\n,\bullet}(m+2m^{2/3}y,m)$, we have
\begin{align}\label{eq:299.1}
&\PP\bigl(M\in[R-1/2,R+1/2],\ \exists y\in[R-3/2,R+3/2]:\ |y-M|\in[\varepsilon,2\varepsilon],\nonumber\\
&\qquad Z_{\0}^{\n,\bullet}(m+2m^{2/3}y,m)\geq Z_{\0}^{\n,\bullet}(m+2m^{2/3}M,m)-\sigma m^{1/3}\sqrt\varepsilon\bigr)\nonumber\\
&\quad\leq C\max\bigl\{\sigma\exp\bigl(-cR^2+C(1+R^2+\log\sigma^{-1})^{5/6}\bigr),\ e^{-cm^{1/12}}\bigr\}.
\end{align}
\par}

\end{lemma}

In the coordinates of \cite{GH20}, the interval above has
half-length $\frac32(m/n)^{2/3}$. Although Corollary 2.13 therein is
stated for half-length at least $1$, its proof applies to this
fixed range bounded away from zero, with constants depending
on $\beta'$.

Before moving on, we remark that, by using Proposition \ref{prop:24} and with some additional work, Lemma \ref{lem:139} can be strengthened and the $e^{O((\log \sigma^{-1})^{5/6})}$ term appearing therein can be removed. However, since our application is not sensitive to the presence of such subpolynomial errors, we make do with Lemma \ref{lem:139} in the interest of brevity.%
As an immediate application of Lemma \ref{lem:139}, we can now bound $\PP(\TP_{A,w,m}^k)$.
  \begin{lemma}
    \label{lem:91}
    There exist constants $C,c$ such that for all $m,n,A,w$ as before and all $k\in [\![1,2f(\sigma)]\!]$, we have

{

{
\begin{equation}
      \label{eq:341}
      \PP(\TP_{A,w,m}^k)\leq C\sigma^{1-\varepsilon/2}+Ce^{-cn^{1/12}}
    \end{equation}
\par}

\par}

  \end{lemma}
  \begin{proof}

{
Let $R$ be the centre of $J_k$ and, for $j\geq0$, set
\[
 \varepsilon_j=2^j\frac{w}{2m^{2/3}},\qquad
 \sigma_j=\frac{A m^{-1/3}}{\sqrt{\varepsilon_j}}
 =\sqrt2\sigma\,2^{-j/2}.
\]
The case $\sigma\geq1/4$ follows by increasing $C$, so we assume $\sigma<1/4$. Apply Lemma \ref{lem:139} at each scale $\varepsilon_j\leq2$. These annuli cover the event $\TP_{A,w,m}^k$, and the choice of $\sigma_j$ gives the deficit $A$; if there are no such annuli, the event is empty. The exponential factor in \eqref{eq:299.1} is at most $C_\varepsilon\sigma^{-\varepsilon/2}2^{j/4}$, uniformly in $R$, since $-cR^2+C|R|^{5/3}$ is bounded above. A union bound gives
\[
\begin{split}
 \PP(\TP_{A,w,m}^k)
 &\leq C\sigma^{1-\varepsilon/2}\sum_{j\geq0}2^{-j/4}
       +C_D(1+\log n)e^{-cm^{1/12}}\\
 &\leq C\sigma^{1-\varepsilon/2}+Ce^{-c'n^{1/12}}.
\end{split}
\]
Since $5/6<1$, the subexponential correction is absorbed by $C_\varepsilon\sigma^{-\varepsilon/2}2^{j/4}$, whose multiplication by $\sigma_j$ gives the powers $\sigma^{1-\varepsilon/2}2^{-j/4}$. Here we used $m\geq\beta'n$, $|R|\leq f(\sigma)+1=O_D((\log n)^{1/3})$ and the fact that there are $O_D(1+\log n)$ annuli, by \eqref{eq:tp-parameter-range}. The bound on $R$ also ensures the spatial restriction in Lemma \ref{lem:139} for sufficiently large $n$. This completes the proof.
\par}
\end{proof}
 With Lemmas \ref{lem:92} and \ref{lem:91} at hand, we now complete the proof of Proposition \ref{prop:tp-parameters}.

{
\begin{proof}[Proof of Proposition \ref{prop:tp-parameters}]
In view of Lemmas \ref{lem:89} and \ref{lem:90}, we apply Lemmas \ref{lem:92} and \ref{lem:91} and sum over the $2f(\sigma)$ intervals. Since $w\leq n^{2/3}$ and $m\geq\beta'n$, we have $A m^{-1/3}\leq C\sigma$. Consequently,
\begin{align}\label{eq:tp-parameter-sum}
 \PP(\TP_{A,w,m})
 &\leq C\sigma^2+C f(\sigma)\sigma^{1-\varepsilon/2}
       +C\sigma f(\sigma)(1+f(\sigma)^2)
       +C f(\sigma)e^{-cn^{1/12}}\nonumber\\
 &\leq C\sigma^{1-\varepsilon}+Ce^{-c'n^{1/12}}.
\end{align}
For the last inequality, $f(\sigma)=O((\log(e/\sigma))^{1/3})$, so its powers are absorbed by the remaining power of $\sigma$; also $f(\sigma)=O_D((\log n)^{1/3})$, which absorbs the prefactor in the exponential error. This completes the proof.
\end{proof}
\par}
  \begin{proof}[Proof of Proposition \ref{prop:26}]

{
Apply Proposition \ref{prop:tp-parameters} with $D=2$, $\varepsilon=\delta$, $A=\ell^\delta$ and $w=\ell^{2/3-\delta}$. Then
\[
 \sigma=\ell^{-1/3+3\delta/2},\qquad
 n^{-2}\leq \frac{w}{n^{2/3}}\leq1,\qquad
 n^{-2}\leq\sigma\leq1,
\]
since $1\leq\ell\leq n$ and $0<\delta<1/6$. Moreover,
\[
 \sigma^{1-\delta}
 =\ell^{-1/3+11\delta/6-3\delta^2/2}
 \leq\ell^{-1/3+2\delta}.
\]
The exponential error is also bounded by a constant multiple of $\ell^{-1/3+2\delta}$, since $\ell\leq n$. This proves the proposition.
\par}
\end{proof}

\section{Appendix 5: Volume accumulation estimate for finite geodesics in BLPP}
\label{sec:app-vol}
In this short appendix, we discuss the proof of Proposition \ref{prop:11}; note that by Brownian scaling, we can just work with the points $p=-\n,q=\n$. Broadly, the proof is the same as the one in \cite[Section 5]{BB23} used to prove the corresponding exponential LPP results \cite[Proposition 35, Proposition 53]{BB23}. However, there are some superficial differences in the sense that one needs to substitute the basic exponential LPP results used frequently therein with their BLPP counterparts. Though we do not provide the complete argument here adapted to BLPP, we now quickly go through the required substitutions. 

\subsubsection*{\textbf{Transversal fluctuation estimates}}
\cite[Section 5]{BB23} frequently uses moderate deviation estimates for transversal fluctuations in exponential LPP. These need to simply be replaced by the corresponding BLPP estimate (Proposition \ref{prop:38}).
\subsubsection*{\textbf{Moderate deviation parallelogram estimates}}
In the setting of exponential LPP, \cite[Section 5]{BB23} frequently uses moderate deviation ``parallelogram'' estimates (see \cite[Theorem 4.2]{BGZ19}). We note that \cite[Theorem 4.2]{BGZ19} has three parts-- (i) and (ii) yield uniform upper and lower tail estimates for well-separated points lying in a parallelogram, whereas (iii) therein provides a corresponding lower tail estimate for ``constrained'' passage times computed using only the vertex weights lying inside the parallelogram.

For the proof of Proposition \ref{prop:11}, the above results need to be substituted with the appropriate BLPP version. Such estimates for ``unconstrained'' passage times are directly available (see \cite[Propositions 3.15, 3.16]{GH20}) and these yield substitutes of (i) and (ii) in Proposition \cite[Theorem 4.2]{BGZ19}. Unfortunately, we have not been able to locate a statement of the ``constrained'' lower tail estimate (analogue of (ii) in \cite[Theorem 4.2]{BGZ19})) in the literature.

However, such an estimate can be obtained by following the same proof as in the exponential LPP case ((iii) in \cite[Theorem 4.2]{BGZ19}). Indeed, \cite[Appendix C]{BGZ19} presents a version of the tree-based argument from \cite{BSS14} used to obtain the above result; precisely the same strategy works to yield the corresponding estimate in BLPP as well, with applications of the exponential LPP transversal fluctuation estimate being simply substituted by the corresponding BLPP estimate Proposition \ref{prop:38}. We refrain from providing more details on this.

\subsubsection*{\textbf{Conditioning on a geodesic and the FKG inequality}} The barrier construction in \cite{BB23} crucially uses that if one conditions on an exponential LPP geodesic, then the vertex weights off the geodesic become stochastically smaller than the i.i.d.\ $\mathrm{Exp}(1)$ environment due to the FKG inequality (see p.30 in \cite[Section 5.3]{BB23}). In our application, the above is substituted by the corresponding BLPP statement (see \cite[Lemma 4.17]{GH20}).

\subsubsection*{\textbf{Regularity estimates for BLPP distance profiles}} \cite[Section 5]{BB23}, proves a semi-infinite geodesic version of Proposition \ref{prop:45} herein and as part of this proof, a simple regularity estimate is used for `Busemann functions' in exponential LPP (see \cite[Lemma 49]{BB23}). For the proof of the finite geodesic statement Proposition \ref{prop:45}, a regularity estimate for point-to-line weight profiles in exponential LPP is used (see \cite[Proposition 53]{BB23}). In our case, we shall require a corresponding regularity statement for BLPP weight profiles. However, such an estimate is available-- indeed, by using the Brownian regularity of the top line of the ensemble $\cP$ (\cite[Theorem 3.11]{CHH23}) and comparing to Brownian motion, one obtains the following well-known regularity estimate.
\begin{lemma}
  \label{lem:115}
  There is a constant $\alpha$ such that for any interval $[a,b]\subseteq [-1,1]$ and $t\leq n^\alpha$, we have
  \begin{equation}
    \label{eq:587}
    \PP(\max_{x\in [a,b]}\cP_1(x)-\min_{x\in [a,b]}\cP_1(x)\geq t\sqrt{b-a})\leq Ce^{-ct^2}.
  \end{equation}
\end{lemma}
As part of the proof of Proposition \ref{prop:11}, the above Gaussian tail estimate needs to be substituted for the corresponding Gaussian tail estimate used for exponential LPP `Busemann functions' in \cite[Lemma 49, (34)]{BB23}. We emphasize that since Lemma \ref{lem:115} is a Gaussian tail estimate just like the one in \cite[Lemma 49]{BB23}, the final exponent $3/11$ appearing in Proposition \ref{prop:11} matches the exponent appearing in \cite[Proposition 35]{BB23}. Note that this exponent is not expected to be optimal.

\printbibliography
\end{document}